\documentclass[10pt,reqno]{amsart}
\usepackage{amssymb,mathrsfs,graphicx}
\usepackage{ifthen}

\usepackage{hyperref}

\usepackage[margin=1in]{geometry}
\usepackage{caption}
\usepackage{sidecap}
\usepackage{rotating,dsfont}
\usepackage{enumitem}

\usepackage{soul}
\usepackage{cancel}
\usepackage[normalem]{ulem}

\usepackage{colortbl}
\definecolor{black}{rgb}{0.0, 0.0, 0.0}
\definecolor{red}{rgb}{1.0, 0.5, 0.5}
 \usepackage{xcolor}
\provideboolean{shownotes} 
\setboolean{shownotes}{true} 
\newcommand{\margnote}[1]{
\ifthenelse{\boolean{shownotes}}%
{\marginpar{\raggedright\tiny\texttt{#1}}}%
{}%
}
\newcommand{\hole}[1]{
\ifthenelse{\boolean{shownotes}}%
{\begin{center} \fbox{ \rule {.25cm}{0cm} \rule[-.1cm]{0cm}{.4cm}
\parbox{.85\textwidth}{\begin{center} \texttt{#1}\end{center}} \rule
{.25cm}{0cm}}\end{center}} {} }

\title[Regularity and low Mach limit for fractional Euler--Alignment]{Regularity theory and low Mach number limit  for the fractional Euler--alignment system}

\author[Choi]{Young-Pil Choi}
\address[Young-Pil Choi]{\newline Department of Mathematics\newline
Yonsei University, 50 Yonsei-Ro, Seodaemun-Gu, Seoul 03722, Republic of Korea}
\email{ypchoi@yonsei.ac.kr}

\author[Jung]{Jinwook Jung}
\address[Jinwook Jung]{\newline Department of Mathematics and  Research Institute for Natural Sciences \newline
Hanyang University, 222 Wangsimni-ro, Seongdong-gu, Seoul 04763, Republic of Korea}
\email{jinwookjung@hanyang.ac.kr}

\numberwithin{equation}{section}

\newtheorem{theorem}{Theorem}[section]
\newtheorem{lemma}{Lemma}[section]

\newtheorem{proposition}{Proposition}[section]
\newtheorem{remark}{Remark}[section]

\newcommand{\R}{\mathbb R}

\newcommand{\bbp} {\mathbb P}
\newcommand{\bbq} {\mathbb Q}

\newcommand{\bq}{\begin{equation}}
\newcommand{\eq}{\end{equation}}
\newcommand{\e}{\varepsilon}
\newcommand{\lt}{\left}
\newcommand{\rt}{\right}
\newcommand{\lal}{\langle}
\newcommand{\ral}{\rangle}
\newcommand{\pa}{\partial}

\newcommand{\intr}{\int_{\R^d}}
\newcommand{\intor}{\iint_{\R^d \times \R^d}}

\newcommand{\intrr}{\iint_{\R^d \times \R^d}}

\newcommand{\calC}{\mathcal C}
\newcommand{\calD}{\mathcal D}
\newcommand{\calE}{\mathcal E}

\newcommand{\calK}{\mathcal K}

\newcommand{\calR}{\mathcal R}

\newcommand{\calV}{\mathcal V}

\newcommand{\calX}{\mathcal X}

\makeatletter
\def\moverlay{\mathpalette\mov@rlay}
\def\mov@rlay#1#2{\leavevmode\vtop{%
   \baselineskip\z@skip \lineskiplimit-\maxdimen
   \ialign{\hfil$\m@th#1##$\hfil\cr#2\crcr}}}
\newcommand{\charfusion}[3][\mathord]{
    #1{\ifx#1\mathop\vphantom{#2}\fi
        \mathpalette\mov@rlay{#2\cr#3}
      }
    \ifx#1\mathop\expandafter\displaylimits\fi}
\makeatother

\newcommand{\rd}{\textnormal{d}}
\newcommand{\dx}{\textnormal{d}x}

\newcommand{\dt}{\textnormal{d}t}

\newcommand{\dy}{\textnormal{d}y}

\newcommand{\dxi}{\rd\xi}

\newcommand{\sfI}{\mathsf{I}}
\newcommand{\sfJ}{\mathsf{J}}
\newcommand{\sfK}{\mathsf{K}}
\newcommand{\sfL}{\mathsf{L}}

\begin{document}
\allowdisplaybreaks

\date{\today}


\begin{abstract}
 We study the compressible Euler--alignment system with pressure under a singular pressure scaling, where the hypersingular communication weight induces a fractional alignment operator of order $2\alpha$, $0<\alpha<1$. The scaling corresponds to a large-time and small-velocity regime and leads to a low Mach number problem in which the density is forced to remain close to a constant state. Our main result is a uniform regularity theory for this scaled pressure system. We establish uniform estimates with respect to the scaling parameter and construct global strong solutions near the constant state. A key feature of the analysis is that, in the low-order fractional regime $0<\alpha\le\frac12$, the estimates close under the lower Sobolev condition $s>\frac d2+1-2\alpha$, gaining $\alpha$ derivatives over the threshold $s>\frac d2 + 1 - \alpha$ arising from a direct use of the fractional alignment dissipation. This is achieved by combining refined commutator estimates for the singular alignment operator with the density dissipation induced by the pressure scaling.  As an application of the uniform estimates, we justify the low Mach number limit toward the incompressible Navier--Stokes system with fractional dissipation. For general small, possibly ill-prepared initial data, a Helmholtz decomposition combined with dispersive estimates for the acoustic component yields subsequential strong convergence locally in space-time to a distributional solution of the limiting system. For well-prepared initial data, a relative-energy argument further identifies the limit with a prescribed sufficiently regular solution and yields global strong convergence in the fractional dissipation norm.
\end{abstract}

\maketitle \centerline{\date}

\tableofcontents

%
%
%
%
%
%
\section{Introduction}

The Euler--alignment system arises as a hydrodynamic description of collective behavior models with velocity alignment, most notably the Cucker--Smale flocking model \cite{CCP17, CHL17, CS07, HT08, Shv24, Tad23}. In this paper, we consider the compressible Euler--alignment system with pressure and a hypersingular communication weight:
\[\begin{aligned}
&\pa_t \rho+\nabla_x\cdot(\rho u)=0,\quad (t,x)\in\R_+\times\R^d,\\
&\pa_t(\rho u)+\nabla_x\cdot(\rho u\otimes u)+\nabla_x p(\rho) =\rho(t,x)\intr \phi(x-y)(u(t,y)-u(t,x))\rho(t,y)\,\dy,
\end{aligned}\]
where
\[
\phi(x)= \frac{c_{d,\alpha}}{|x|^{d+2\alpha}},\quad 0<\alpha<1,
\]
and $c_{d,\alpha}>0$ is chosen consistently with the normalization of the fractional Laplacian. We also take
\[
p(\rho)=\rho^\gamma,\quad \gamma\ge1.
\]
Throughout, the singular integrals involving $\phi$ are understood in the Cauchy principal value sense.  For this choice of the communication weight, the alignment force has a commutator structure. With the normalization of the fractional Laplacian used here, one has
\[
\intr \phi(x-y)(u(y)-u(x))\rho(y)\,\dy = -\Lambda^{2\alpha}(\rho u)+u\Lambda^{2\alpha}\rho, \quad \Lambda:=(-\Delta)^{\frac12}.
\]
This point of view, developed in the commutator-forcing framework for Euler--alignment systems with singular kernels \cite{DKRT18, ST17a,ST17b,ST18}, shows that singular alignment acts as a density-dependent fractional dissipation. In particular, near the constant state $\rho=1$, the leading linear part of the alignment force is $-\Lambda^{2\alpha}u$.

We study the large-time and small-velocity regime. Introducing the scaling
\[
\rho\mapsto \rho(\e t,x),\quad u\mapsto \e u(\e t,x),\quad \phi\mapsto \e\phi,
\]
we obtain the rescaled system
\bq\label{main_eq_e}
\begin{aligned}
&\pa_t \rho+\nabla_x\cdot(\rho u)=0,\quad (t,x)\in\R_+\times\R^d,\\
&\pa_t(\rho u)+\nabla_x\cdot(\rho u\otimes u) + \frac1{\e^2}\nabla_x p(\rho) =\rho(t,x)\intr \phi(x-y)(u(t,y)-u(t,x))\rho(t,y)\,\dy.
\end{aligned}
\eq
The singular pressure term in \eqref{main_eq_e} forces the density to remain close to the constant state. In view of the commutator identity above, this means that the density-dependent alignment force is expected to reduce, at the leading order, to the linear fractional dissipation $-\Lambda^{2\alpha}u$. Thus the natural limiting equation is the incompressible Navier--Stokes system with fractional dissipation.  The main purpose of this paper is to develop a uniform regularity theory for the scaled pressure system \eqref{main_eq_e}. In particular, we derive estimates that are uniform with respect to $\e$ and construct global strong solutions near the constant state. These estimates also provide a natural framework for the low Mach number limit toward the fractional incompressible Navier--Stokes system.

 The well-posedness theory for Euler--alignment systems has been extensively studied in several regimes. For pressureless Euler--alignment systems, critical threshold phenomena and global regularity have been investigated for bounded, weakly singular, and strongly singular communication weights; see, for instance, \cite{AC21, CCTT16, Cho19, CJ24, DMPW19, DKRT18,KT18,Les19, Tad26, TT14, Tan20}. A particularly important feature of singular alignment is its regularizing effect. In the one-dimensional pressureless setting, a commutator-forcing framework was developed to establish global regularity and flocking for singular alignment models, including the critical fractional regime \cite{ST17a,ST17b,ST18}. Related lower-regularity and continuation-criterion mechanisms have also been developed for topological or special-structure Euler--alignment systems; see, for example, \cite{LRS22,LRSh22}. Euler--alignment systems with pressure have also been studied recently. For the compressible Euler system with singular velocity alignment, small-data global classical or strong solutions and large-time flocking or decay estimates were obtained in \cite{BMTX24,BTX24, CTT21}. In particular, the pressure case with strongly singular alignment was treated in Sobolev spaces for small smooth perturbations of constant states in \cite{BTX24}. More recently, a global well-posedness theory in critical Besov spaces for the compressible Euler--alignment system with pressure was established in \cite{BMTX24} when the order of the fractional alignment operator lies in $(1,2)$, which corresponds to $\frac12<\alpha<1$ in the present notation.

The present work differs from the above results in that we develop estimates that are uniform with respect to the singular scaling parameter $\e$. This is not only a technical requirement for passing to the incompressible limit, but also a regularity question for the scaled pressure system itself: the pressure term becomes singular as $\e\to0$, while the nonlinear commutator structure of the alignment force must be controlled uniformly.

Another key point concerns the regularity required in the low-order fractional regime $0<\alpha\le\frac12$, corresponding to alignment of order $2\alpha\le1$. We note that the critical Besov theory of \cite{BMTX24} applies to alignment of order strictly between $1$ and $2$, which corresponds to $\frac12<\alpha<1$ in the present notation, while the extension to alignment of order not exceeding $1$ was left for future investigation there. In contrast, our uniform Sobolev estimates also cover the low-order regime $0<\alpha\le\frac12$. In this regime, the estimates can be closed under the condition
\[
s>\frac d2+1-2\alpha.
\]
This improves upon the threshold
\[
s>\frac d2+1-\alpha
\]
suggested by a direct use of the fractional alignment dissipation. The improvement is obtained through a refined commutator analysis combined with the pressure-induced density dissipation
\[
\frac1{\e^2}\|\nabla \rho\|_{H^{s+\alpha-1}}^2.
\]
In particular, the regularity threshold is lowered by $\alpha$ derivatives compared with the direct energy estimate and lies below the standard hyperbolic threshold $s>\frac d2+1$. Moreover, the threshold $\frac d2+1-2\alpha$ agrees with the regularity index for the velocity in the critical Besov setting of \cite{BMTX24} when their fractional order is identified with $2\alpha$. At the borderline value $\alpha=\frac12$, our condition reduces to $s>\frac d2$.

For the higher-order regime $\frac12<\alpha<1$, the main issue is different. While the regularity threshold in this regime is not intended as an improvement over the critical-space theories discussed above, the uniform estimate with respect to $\e$ still requires a separate argument adapted to the singular pressure scaling and the nonlinear commutator structure of the alignment force. In particular, the density dissipation is recovered at the level
\[
\frac1{\e^2}\|\nabla \rho\|_{H^{s-\alpha}}^2.
\]
Thus, the lower-regularity improvement is specific to the low-order fractional regime, whereas the derivation of $\e$-uniform estimates is a main feature of both regimes. These estimates, in turn, provide the basis for the incompressible limit throughout the full range $0<\alpha<1$.

Setting $h:=\rho-1$, the system \eqref{main_eq_e} can be rewritten as
\bq\label{h-u-system}
\begin{aligned}
&\pa_t h+\nabla\cdot((1+h)u)=0,\\
&\pa_t u+u\cdot\nabla u+\frac{\gamma}{\e^2}(1+h)^{\gamma-2}\nabla h = -\Lambda^{2\alpha}u-\Lambda^{2\alpha}(hu)+u\Lambda^{2\alpha}h.
\end{aligned}
\eq
The singular pressure term of size $\e^{-2}$ yields the uniform density control needed below. Formally, as $\e\to0$, one expects
\[
\rho^\e\to1
\]
and the velocity to converge to the solution of the incompressible Navier--Stokes system with fractional dissipation:
\bq\label{eq:frac_NS}
\begin{aligned}
&\pa_t v+v\cdot\nabla v+\nabla \pi+\Lambda^{2\alpha}v=0,\\
&\nabla\cdot v=0.
\end{aligned}
\eq
We shall refer to \eqref{eq:frac_NS} as the fractional incompressible Navier--Stokes system. Equations of this type have been studied in connection with local well-posedness, phase transitions, regularity criteria, partial regularity, and nonuniqueness; see, for instance, \cite{BG24,KP02,Lio69,LT20,Tao09,YZ12,Zha10} and the references therein. In the present paper, we first obtain a distributional solution as a subsequential limit for ill-prepared data. For well-prepared data, we compare the scaled solutions with a prescribed sufficiently regular solution on the time interval under consideration and establish strong convergence to that solution.

We now state the main results. For $d\ge 2$ and $0<\alpha<1$, we set
\[
\sigma_\alpha:= 
\begin{cases}
s+\alpha-1, & 0<\alpha\le \frac12,\\[1mm]
s-\alpha, & \frac12<\alpha<1.
\end{cases}
\]
The regularity assumption on $s$ is given by
\bq\label{eq:s-ass}
\begin{cases}
0<\alpha\le \frac12,\quad s>\frac d2+1-2\alpha,\\[1mm]
\frac12<\alpha<1,\quad s> \frac d2.
\end{cases}
\eq

\begin{theorem}\label{thm:global} Let $d\ge2$, $0<\alpha<1$, $\gamma\ge1$, and $0<\e\le1$.  Assume that $s$ satisfies \eqref{eq:s-ass}. There exists $\eta_0>0$, independent of $\e$, such that if the initial data satisfy
\[
\rho_0^\e=1+h_0^\e>0,\quad h_0^\e\in H^s(\R^d),\quad u_0^\e\in H^s(\R^d),
\]
and
\[
\frac{\|h_0^\e\|_{H^s}}{\e}+\|u_0^\e\|_{H^s}\le \eta_0,
\]
then the rescaled system \eqref{main_eq_e} admits a unique global strong solution $(\rho^\e,u^\e)$ with $\rho^\e=1+h^\e$. Moreover, for all $t\ge0$,
\bq\label{main_uniform_est}
\frac{\|h^\e(t)\|_{H^s}^2}{\e^2} +\|u^\e(t)\|_{H^s}^2  +\int_0^t \lt( \frac{\|\nabla h^\e(\tau)\|_{H^{\sigma_\alpha}}^2}{\e^2} +\|\Lambda^\alpha u^\e(\tau)\|_{H^s}^2 \rt) \rd\tau \le C \lt( \frac{\|h_0^\e\|_{H^s}^2}{\e^2} +\|u_0^\e\|_{H^s}^2 \rt),
\eq
where $C>0$ is independent of $t$ and $\e$.
\end{theorem}

\begin{remark}
The threshold in \eqref{eq:s-ass} reflects the different commutator structures in the two regimes. When $0<\alpha\le\frac12$, the estimates can be closed at the lower regularity level $s>\frac d2+1-2\alpha$. In this case the density dissipation is obtained at the level $\|\nabla h^\e\|_{H^{s+\alpha-1}}$. When $\frac12<\alpha<1$, the standard high-order estimate gives the dissipation $\|\nabla h^\e\|_{H^{s-\alpha}}$.
\end{remark}

We next turn to the incompressible limit, a classical singular-limit problem in compressible fluid mechanics. Foundational results for smooth solutions can be found in \cite{KM81,KM82,Sch94,Sch05}. For viscous compressible fluids, both well-prepared and ill-prepared limits have been studied in \cite{LM98,DG99,DGLM99,Dan05}. In the ill-prepared case, the main additional difficulty is the fast acoustic component, whose local decay on unbounded domains follows from dispersive estimates; see \cite{Uka86,MS01,Dan05,FKM19}.  Relative entropy methods provide a robust framework for stability, weak--strong uniqueness, and singular limits of compressible fluid systems; see \cite{FJN12,FN17}.

In the current work, both mechanisms enter naturally. The uniform estimate in Theorem \ref{thm:global} gives the scaled density control $\frac1\e\|\rho^\e-1\|_{H^s}$ together with the fractional alignment dissipation $\|\Lambda^\alpha u^\e\|_{H^s}$. For general small initial data satisfying the uniform assumptions of Theorem \ref{thm:global}, we decompose the velocity into its divergence-free and gradient components. The divergence-free component is locally compact by the projected equation, while the gradient component satisfies a fast acoustic system. Frequency-localized Strichartz estimates show that this acoustic component disperses locally in space-time as $\e\to0$. This yields subsequential strong convergence toward a distributional solution of the fractional incompressible Navier--Stokes system, without requiring well-prepared initial data.

For well-prepared data, we further use the relative energy
\[
\mathscr{H}^\e(t):= \frac12\intr \rho^\e |u^\e-v|^2\,\dx +\frac{1}{\e^2(\gamma-1)} \intr \lt[(\rho^\e)^\gamma-1-\gamma(\rho^\e-1)\rt] \dx,
\]
where $v$ is a sufficiently regular solution to the fractional incompressible Navier--Stokes system. The relative-energy estimate identifies the subsequential limit with $v$ and yields strong convergence in $L^2(0,T;\dot H^\alpha)$. Compared with the classical low Mach number limit, the additional difficulty in both arguments is the nonlinear commutator structure of the Euler--alignment force. To the best of our knowledge, such a low Mach number limit for the compressible Euler--alignment system with hypersingular fractional alignment has not been established previously.

To avoid introducing a separate notation for the logarithmic pressure potential, we state the incompressible limit theorem for $\gamma>1$. The case $\gamma=1$ can be treated in the same way by replacing the pressure potential with the corresponding logarithmic one.

We denote by
\[
\bbp:=I-\nabla\Delta^{-1}\nabla\cdot, \quad \bbq:=\nabla\Delta^{-1}\nabla\cdot
\]
the Leray projection onto divergence-free vector fields and the complementary gradient projection, respectively.

\begin{theorem}\label{thm:incomp-limit}
Let $d\ge2$, $T>0$, $\gamma\ge1$, and let $s$ satisfy \eqref{eq:s-ass}. For each $0<\e\le1$, let $(\rho^\e,u^\e)$ be the global strong solution to \eqref{main_eq_e} given by Theorem \ref{thm:global}. Assume that
\[
\sup_{0<\e\le1}\lt( \frac{\|\rho_0^\e-1\|_{H^s}}{\e}+\|u_0^\e\|_{H^s}\rt)\le\eta_0.
\]

Then, up to a subsequence, there exist $u_0\in H^s(\R^d)$ and $u\in L^\infty(0,T;H^s(\R^d))\cap L^2(0,T;H^{s+\alpha}(\R^d))$ such that
\[
u_0^\e \rightharpoonup u_0 \quad \mbox{weakly in } H^s(\R^d)
\quad \text{and} \quad
\rho^\e -1 \to 0 \quad \mbox{strongly in } L^\infty(0,T;H^s(\R^d)).
\]
Moreover, for every $0<\delta<s$,
\[
\bbq u^\e \to 0 \quad \mbox{strongly in } L^2(0,T;H^{s-\delta}_{loc}(\R^d))
\quad \text{and} \quad
u^\e\to u\quad\mbox{strongly in }L^2(0,T;H^{s-\delta}_{loc}(\R^d)).
\]
The limit $u$ is divergence-free and is a distributional solution to
\[
\pa_t u+u\cdot\nabla u+\nabla\pi+\Lambda^{2\alpha}u=0,\quad \nabla\cdot u=0,
\]
with initial datum $\bbp u_0$.

In addition, let $v$ be a sufficiently regular solution to the fractional incompressible Navier--Stokes system \eqref{eq:frac_NS} on $[0,T]$ with initial datum $v_0$. Assume that
\[
\nabla v\in L^1(0,T;L^\infty(\R^d)),\quad v\in L^\infty(0,T;L^2(\R^d))\cap L^2(0,T;\dot H^\alpha(\R^d)),
\]
and that the associated pressure satisfies
\[
\pi\in W^{1,1}(0,T;L^2(\R^d)),\quad v\cdot\nabla\pi\in L^1(0,T;L^2(\R^d)).
\]
If the initial data are well prepared in the sense that
\[
\mathscr{H}^\e(0)\to0\quad\mbox{as } \e\to0,
\]
then
\[
\sup_{0\le t\le T}\mathscr{H}^\e(t)\to0\quad \text{and} \quad u^\e\to v\quad\mbox{in }  L^2(0,T;\dot H^\alpha(\R^d)) \quad  \mbox{as } \e\to0.
\]
\end{theorem}

\begin{remark}
The first part of Theorem \ref{thm:incomp-limit} does not require the initial data to be well prepared with respect to a prescribed incompressible flow. In particular, both the scaled density fluctuation $(\rho_0^\e-1)/\e$ and the gradient component $\bbq u_0^\e$ may remain of order one as $\e\to0$. The latter generates fast acoustic oscillations, which disperse locally in the whole space $\R^d$. This yields subsequential strong convergence locally in space-time. Under the additional well-preparedness condition $\mathscr{H}^\e(0)\to0$, the relative-energy argument identifies the limit with the prescribed sufficiently regular solution and upgrades the convergence to the global strong convergence stated above.
\end{remark}

We close the introduction by describing the organization of the paper. In Section \ref{sec:low-alpha}, we derive the uniform a priori estimates in the low-order fractional regime $0<\alpha\le\frac12$. In Section \ref{sec:high-alpha}, we establish the corresponding estimates in the higher-order regime $\frac12<\alpha<1$. Section \ref{sec:global} is devoted to local and global well-posedness. Finally, in Section \ref{sec:limit}, we prove the low Mach number limit. The ill-prepared case is treated by acoustic dispersion and local compactness, while the well-prepared case follows from a relative-energy argument.

%
%
%
%
%
%

\section{Uniform estimates in the low-order fractional regime}\label{sec:low-alpha}

In this section, we prove the uniform a priori estimate in the low-order fractional alignment regime
\[
0<\alpha\le \frac12.
\]
Here the alignment dissipation has order $2\alpha\le1$, which is below the first-order threshold. Throughout this section, we assume
\[
s>\frac d2+1-2\alpha.
\]
We use the notation $h=\rho-1$ and the formulation \eqref{h-u-system}.  The main point in this regime is that the alignment dissipation has order
$2\alpha\le1$, while the singular pressure term provides the density dissipation
at the level
\[
\frac1{\e^2}\|\nabla h\|_{H^{s+\alpha-1}}^2.
\]
This allows us to close the uniform estimate below the standard hyperbolic threshold $s>\frac d2+1$.

\begin{proposition}\label{prop:uni-apri-low}
Let $d\ge2$, $0<\alpha\le\frac12$, and $\gamma\ge1$. Assume that
\[
s>\frac d2+1-2\alpha.
\]
Let $(h,u)$ be a smooth solution to \eqref{h-u-system} on $[0,T]$. Suppose that
\bq\label{small-assumption-low}
\sup_{0\le t\le T} \lt( \frac{\|h(t)\|_{H^s}}{\e}+\|u(t)\|_{H^s} \rt) \le \eta
\eq
for some sufficiently small $\eta>0$. Then, for all $t\in[0,T]$,
\[
\frac{\|h(t)\|_{H^s}^2}{\e^2} +\|u(t)\|_{H^s}^2 +\int_0^t\lt( \frac{\|\nabla h(\tau)\|_{H^{s+\alpha-1}}^2}{\e^2} +\|\Lambda^\alpha u(\tau)\|_{H^s}^2 \rt)\rd\tau \le C\lt( \frac{\|h_0\|_{H^s}^2}{\e^2} +\|u_0\|_{H^s}^2 \rt),
\]
where $C>0$ is independent of $T$ and $\e$.
\end{proposition}

By \eqref{small-assumption-low} and Sobolev embedding, $\|h\|_{L^\infty}$ is sufficiently small. Hence $1+h$ is bounded away from zero and infinity. We shall use this fact throughout the section without further comment.

We recall the following product and commutator estimates \cite{CTT21, Li19, MTX21}. Let $\beta >0$ and $1<r<\infty$. Suppose that
\[
1<p_i,q_i\le\infty,\quad  \frac1r=\frac1{p_i}+\frac1{q_i},\quad i=1,2.
\]
Then
\bq\label{eq:frac_prod1}
\|\Lambda^\beta(fg)\|_{L^r} \le C\lt( \|\Lambda^\beta f\|_{L^{p_1}}\|g\|_{L^{q_1}} + \|f\|_{L^{p_2}}\|\Lambda^\beta g\|_{L^{q_2}} \rt)
\eq
and when $\beta\in(0,1)$,
\bq\label{eq:frac_prod2}
\|\Lambda^\beta(fg)-f\Lambda^\beta g - g\Lambda^\beta f\|_{L^r} \le C \|\Lambda^{\beta-\sigma} f\|_{L^{p_1}}\|\Lambda^\sigma g\|_{L^{q_1}},
\eq
for $\sigma \in [0,\beta]$.

Moreover, since $\|h\|_{L^\infty}$ is small, we get
\bq\label{eq:moser_den}
\|\Lambda^\beta((1+h)^{\gamma-2})\|_{L^2} \le C\|\Lambda^\beta h\|_{L^2}.
\eq

%
%
%
%
%
%

\subsection{Fourier interpolation estimates}

We provide the Fourier interpolation estimates used repeatedly in the low-order fractional regime.  

\begin{lemma} \label{lem:low-four-inte}
Let $d\ge2$, $0<\alpha\le\frac12$, and $s>\frac d2+1-2\alpha$. Let
\[
p:=\frac{2s+4\alpha-2}{s+3\alpha-1}, \quad p':=\frac{2s+4\alpha-2}{2s+3\alpha-2}.
\]
Then, for any sufficiently regular $f$, we have
\bq\label{low-four-d-gen}
\||\xi|^s\hat f\|_{L_\xi^p} +\||\xi|\hat f\|_{L_\xi^{p'}} +\||\xi|^{2\alpha}\hat f\|_{L_\xi^{p'}} \le C\|\nabla f\|_{H^{s+\alpha-1}}
\eq
and
\bq\label{low-four-Hs-gene} 
\|\hat f\|_{L_\xi^1} +\||\xi|^{s+\alpha-1}\hat f\|_{L_\xi^p} \le C\|f\|_{H^s}.
\eq
\end{lemma}

\begin{proof}
By the definition of $p$ and $p'$, we have $p,p'\in(1,2)$ and $\frac1p+\frac1{p'}=\frac32$. We shall also use the identities
\[
\frac{sp-(1-\alpha)(2-p)}{2(p-1)}=s+\alpha, \quad \frac{p'-(1-\alpha)(2-p')}{2(p'-1)}=s+\alpha,
\]
and
\[
\frac{2\alpha p'-\alpha(2-p')}{2(p'-1)}=s+3\alpha-1.
\]
We first recall a simple interpolation inequality. Let $1<r<2$ and $a,b\ge0$. Then,
\[
\||\xi|^a\hat f\|_{L_\xi^r} \le \||\xi|^b\hat f\|_{L_\xi^1}^{\frac{2-r}{r}} \lt\| |\xi|^{\frac{ar-b(2-r)}{2(r-1)}}\hat f \rt\|_{L_\xi^2}^{\frac{2(r-1)}{r}}.
\]
Indeed, this follows from writing
\[
|\xi|^{ar}|\hat f|^r = \lt(|\xi|^b|\hat f|\rt)^{2-r} \lt( |\xi|^{\frac{ar-b(2-r)}{r-1}}|\hat f|^2 \rt)^{r-1}
\]
and applying H\"older's inequality.

We next estimate the $L^1_\xi$ factors that appear below. Choose $\delta>0$ sufficiently small so that
\[
0<\delta<\alpha, \quad \frac d2+1-\alpha+\delta\le s+\alpha.
\]
Then, for $\beta=1-\alpha$ and $\beta=\alpha$,
\[
\||\xi|^\beta\hat f\|_{L_\xi^1} \le C \|\Lambda^{\frac d2+\beta-\delta}f\|_{L^2}^{1/2} \|\Lambda^{\frac d2+\beta+\delta}f\|_{L^2}^{1/2}.
\]
The Sobolev orders $\frac d2+\beta\pm\delta$ can be chosen in the interval $[1,s+\alpha]$. Hence, we have
\[
\||\xi|^{1-\alpha}\hat f\|_{L_\xi^1} +\||\xi|^\alpha\hat f\|_{L_\xi^1} \le C\|\nabla f\|_{H^{s+\alpha-1}}.
\]

We now prove \eqref{low-four-d-gen}. Applying the interpolation inequality with $(a,b,r)=(s,1-\alpha,p)$, we obtain
\[
\||\xi|^s\hat f\|_{L_\xi^p} \le \||\xi|^{1-\alpha}\hat f\|_{L_\xi^1}^{\frac{2-p}{p}} \lt\| |\xi|^{s+\alpha}\hat f \rt\|_{L_\xi^2}^{\frac{2(p-1)}{p}}.
\]
The $L^2_\xi$ order is exactly $s+\alpha$, and thus it is controlled by $\|\nabla f\|_{H^{s+\alpha-1}}$. This implies
\[
\||\xi|^s\hat f\|_{L_\xi^p} \le C\|\nabla f\|_{H^{s+\alpha-1}}.
\]
Similarly, applying the interpolation inequality with $(a,b,r)=(1,1-\alpha,p')$, we get
\[
\||\xi|\hat f\|_{L_\xi^{p'}} \le \||\xi|^{1-\alpha}\hat f\|_{L_\xi^1}^{\frac{2-p'}{p'}} \lt\| |\xi|^{s+\alpha}\hat f \rt\|_{L_\xi^2}^{\frac{2(p'-1)}{p'}}.
\]
Thus,
\[
\||\xi|\hat f\|_{L_\xi^{p'}} \le C\|\nabla f\|_{H^{s+\alpha-1}}.
\]
Finally, applying the interpolation inequality with $(a,b,r)=(2\alpha,\alpha,p')$, we find
\[
\||\xi|^{2\alpha}\hat f\|_{L_\xi^{p'}} \le \||\xi|^\alpha\hat f\|_{L_\xi^1}^{\frac{2-p'}{p'}} \lt\| |\xi|^{s+3\alpha-1}\hat f \rt\|_{L_\xi^2}^{\frac{2(p'-1)}{p'}}.
\]
Since
\[
1 \le \frac d2+\alpha < s+3\alpha-1\le s+\alpha,
\]
this order also lies in $[1,s+\alpha]$. Thus, we get
\[
\||\xi|^{2\alpha}\hat f\|_{L_\xi^{p'}} \le C\|\nabla f\|_{H^{s+\alpha-1}}.
\]
This proves \eqref{low-four-d-gen}.

It remains to prove \eqref{low-four-Hs-gene}. Since $s>\frac d2$, we find $\|\hat f\|_{L_\xi^1}\le C\|f\|_{H^s}$. Moreover, applying the interpolation inequality with $(a,b,r)=(s+\alpha-1,0,p)$, we obtain
\[
\||\xi|^{s+\alpha-1}\hat f\|_{L_\xi^p} \le \|\hat f\|_{L_\xi^1}^{\frac{2-p}{p}} \lt\| |\xi|^{s+2\alpha-1}\hat f \rt\|_{L_\xi^2}^{\frac{2(p-1)}{p}}.
\]
Since $2\alpha\le1$, we obtain $s+2\alpha-1\le s$, and thus
\[
\||\xi|^{s+\alpha-1}\hat f\|_{L_\xi^p}
\le C\|f\|_{H^s}.
\]
This completes the proof.
\end{proof}

\begin{remark}\label{rem:low-four-cons}
Since $0<\alpha\le1$, we have
\[
\|\nabla f\|_{H^{s+\alpha-1}} \le C\|\Lambda^\alpha f\|_{H^s}.
\]
Thus, \eqref{low-four-d-gen} also gives
\[
\||\xi|^s\hat f\|_{L_\xi^p} +\||\xi|\hat f\|_{L_\xi^{p'}} +\||\xi|^{2\alpha}\hat f\|_{L_\xi^{p'}} \le C\|\Lambda^\alpha f\|_{H^s}.
\]
Moreover, let
\[
g=(1+h)^{\gamma-2}-1.
\]
If $\|h\|_{L^\infty}$ is sufficiently small, then the Moser estimate  \eqref{eq:moser_den} gives
\[
\|\nabla g\|_{H^{s+\alpha-1}} \le C\|\nabla h\|_{H^{s+\alpha-1}}, \quad \|g\|_{H^s}\le C\|h\|_{H^s}.
\]
Combining these bounds with \eqref{low-four-d-gen} and \eqref{low-four-Hs-gene}, we obtain
\[
\||\xi|\hat g\|_{L_\xi^{p'}} +\||\xi|^{2\alpha}\hat g\|_{L_\xi^{p'}} \le C\|\nabla h\|_{H^{s+\alpha-1}}
\]
and
\[
\|\hat g\|_{L_\xi^1} +\||\xi|^{s+\alpha-1}\hat g\|_{L_\xi^p} \le C\|h\|_{H^s}.
\]
\end{remark}

Finally, we provide the Sobolev embeddings used in the zeroth-order estimates:
\bq\label{low-phy-emb}
\|u\|_{L^{\frac{1}{\frac12-\frac\alpha d}}} \le C\|\Lambda^\alpha u\|_{L^2},
\quad
\|\Lambda^\alpha h\|_{L^{\frac d\alpha}} \le C\|\Lambda^\alpha h\|_{\dot H^{\frac d2-\alpha}} \le C\|h\|_{H^s},
\eq
and
\bq\label{low-den-emb}
\|\Lambda^{1+\alpha}h\|_{L^2} +\|\Lambda^{2\alpha}h\|_{L^{\frac d\alpha}} \le C\|\nabla h\|_{H^{s+\alpha-1}}.
\eq
Indeed, the first estimate in \eqref{low-den-emb} is immediate from $1+\alpha\le s+\alpha$. For the second one, Sobolev embedding gives
\[
\|\Lambda^{2\alpha}h\|_{L^{\frac d\alpha}} \le C\|\Lambda^{2\alpha}h\|_{\dot H^{\frac d2-\alpha}} = C\|\Lambda^{\frac d2+\alpha}h\|_{L^2}.
\]
Since $s>\frac d2$, we have $\frac d2+\alpha\le s+\alpha$, and hence this term is controlled by $\|\nabla h\|_{H^{s+\alpha-1}}$.

%
%
%
%
%
%

\subsection{Zeroth-order energy estimates}\label{ssec:low-zero}

We first derive the zeroth-order estimates in the regime $0<\alpha\le\frac12$. From the continuity equation, we compute
\begin{align*}
\frac12\frac\rd\dt\intr (1+h)^{\gamma-3}|h|^2\,\dx
&= \frac{\gamma-3}{2}\intr (1+h)^{\gamma-4}\pa_t h |h|^2\,\dx +\intr (1+h)^{\gamma-3}h\pa_t h\,\dx\\
&= -\frac{\gamma-3}{2}\intr (1+h)^{\gamma-4} \nabla\cdot((1+h)u)|h|^2\,\dx  -\intr (1+h)^{\gamma-3}h\nabla\cdot((1+h)u)\,\dx\\
&= -\frac{\gamma-3}{2}\intr (1+h)^{\gamma-4} \nabla h\cdot u |h|^2\,\dx  -\frac{\gamma-3}{2}\intr (1+h)^{\gamma-3} (\nabla\cdot u)|h|^2\,\dx\\
&\quad -\frac12\intr (1+h)^{\gamma-3}\nabla(h^2)\cdot u\,\dx -\intr (1+h)^{\gamma-2}h\nabla\cdot u\,\dx\\
&= \frac{4-\gamma}{2}\intr (1+h)^{\gamma-3} (\nabla\cdot u)|h|^2\,\dx +(\gamma-2)\intr (1+h)^{\gamma-3}(\nabla h\cdot u)h\,\dx \\
&\quad  +\intr (1+h)^{\gamma-2}\nabla h\cdot u\,\dx.
\end{align*}
Using the smallness of $h$ in $L^\infty$, we obtain
\[
\frac12\frac\rd\dt\intr (1+h)^{\gamma-3}|h|^2\,\dx  \le C\|\nabla u\|_{L^2}\|h\|_{L^4}^2 +C\|\nabla h\|_{L^2}\|u\|_{L^4}\|h\|_{L^4} +\intr (1+h)^{\gamma-2}\nabla h\cdot u\,\dx .
\]
We use the interpolation estimate
\[
\|f\|_{L^4} \le C\|\Lambda^{\frac d4}f\|_{L^2} \le C\|\Lambda^{\frac d2}f\|_{L^2}^{\frac12}\|f\|_{L^2}^{\frac12}.
\]
Since $s+\alpha>\frac d2$, we have
\[
\|\Lambda^{\frac d2}u\|_{L^2} \le C\|\Lambda^\alpha u\|_{H^s}, \quad \|\Lambda^{\frac d2}h\|_{L^2} \le C\|\nabla h\|_{H^{s+\alpha-1}}.
\]
This yields
\begin{align*}
\frac12\frac\rd\dt\intr (1+h)^{\gamma-3}|h|^2\,\dx
&\le C\|\Lambda^\alpha u\|_{H^s} \|\nabla h\|_{H^{s+\alpha-1}}\|h\|_{H^s}+C\|u\|_{H^s}\|\nabla h\|_{H^{s+\alpha-1}}^2   +\intr (1+h)^{\gamma-2}\nabla h\cdot u\,\dx.
\end{align*}
Multiplying by $\frac{\gamma}{\e^2}$ and using Young's inequality, we get
\bq\label{h-zero-est-low}
\begin{aligned}
\frac12\frac\rd\dt\lt( \frac{\gamma}{\e^2}\intr (1+h)^{\gamma-3}|h|^2\,\dx \rt)
&\le C\lt( \frac{\|h\|_{H^s}}{\e}+\|u\|_{H^s} \rt) \lt( \frac{\|\nabla h\|_{H^{s+\alpha-1}}^2}{\e^2} +\|\Lambda^\alpha u\|_{H^s}^2 \rt)\\
&\quad +\frac{\gamma}{\e^2} \intr (1+h)^{\gamma-2}\nabla h\cdot u\,\dx .
\end{aligned}
\eq
Next, we estimate the velocity. Taking the $L^2$ inner product of the second equation in \eqref{h-u-system} with $u$, we get
\begin{align*}
\frac12\frac\rd\dt\|u\|_{L^2}^2+\|\Lambda^\alpha u\|_{L^2}^2
&= -\intr (u\cdot\nabla u)\cdot u\,\dx -\frac{\gamma}{\e^2} \intr (1+h)^{\gamma-2}\nabla h\cdot u\,\dx\\
&\quad -\intr \Lambda^{2\alpha}(hu)\cdot u\,\dx +\intr |u|^2\Lambda^{2\alpha}h\,\dx .
\end{align*}
By H\"older's inequality and Sobolev embedding,
\[
\lt|\intr (u\cdot\nabla u)\cdot u\,\dx\rt| \le \|\nabla u\|_{L^{\frac d\alpha}}\|u\|_{L^2}\|u\|_{L^{\frac{1}{\frac12-\frac\alpha d}}} \le C\|u\|_{H^s}\|\Lambda^\alpha u\|_{H^s}^2.
\]
For the first nonlinear alignment term, we use
\[
\lt|\intr \Lambda^{2\alpha}(hu)\cdot u\,\dx\rt| \le \|\Lambda^\alpha(hu)\|_{L^2}\|\Lambda^\alpha u\|_{L^2}.
\]
By the fractional product estimate \eqref{eq:frac_prod1} and \eqref{low-phy-emb}, we find
\bq\label{eq:hu-low}
\|\Lambda^\alpha(hu)\|_{L^2} \le C\|h\|_{L^\infty}\|\Lambda^\alpha u\|_{L^2} +C\|u\|_{L^{\frac{1}{\frac12-\frac\alpha d}}}\|\Lambda^\alpha h\|_{L^{\frac d\alpha}} \le C\|h\|_{H^s}\|\Lambda^\alpha u\|_{H^s}.
\eq
Thus,
\[
\lt|\intr \Lambda^{2\alpha}(hu)\cdot u\,\dx\rt| \le C\|h\|_{H^s}\|\Lambda^\alpha u\|_{H^s}^2.
\]
For the remaining alignment term, we use duality, Sobolev embedding, and the fractional product estimate:
\begin{align*}
\lt|\intr |u|^2\Lambda^{2\alpha}h\,\dx\rt|
&= \lt|\intr \Lambda^\alpha h\,\Lambda^\alpha(|u|^2)\,\dx\rt| \le C \|\Lambda^\alpha h\|_{L^{\frac{d}{\alpha}}} \|\Lambda^\alpha u\|_{L^2} \|u\|_{L^{\frac{1}{\frac12-\frac\alpha d}}} \le C\|h\|_{H^s}\|\Lambda^\alpha u\|_{H^s}^2.
\end{align*}
Consequently, for $0<\e\le1$,
\bq\label{u-zero-est-low}
\begin{aligned}
\frac12\frac\rd\dt\|u\|_{L^2}^2+\|\Lambda^\alpha u\|_{L^2}^2
&\le C\lt( \frac{\|h\|_{H^s}}{\e}+\|u\|_{H^s} \rt) \lt( \frac{\|\nabla h\|_{H^{s+\alpha-1}}^2}{\e^2} +\|\Lambda^\alpha u\|_{H^s}^2 \rt)\\
&\quad -\frac{\gamma}{\e^2} \intr (1+h)^{\gamma-2}\nabla h\cdot u\,\dx .
\end{aligned}
\eq
Adding \eqref{h-zero-est-low} and \eqref{u-zero-est-low}, the pressure coupling
terms cancel and we obtain
\bq\label{zero-e-est-low}
\begin{aligned}
&\frac12\frac\rd\dt\lt( \frac{\gamma}{\e^2}\intr (1+h)^{\gamma-3}|h|^2\,\dx +\|u\|_{L^2}^2 \rt) +\|\Lambda^\alpha u\|_{L^2}^2 \\
&\quad \le C\lt( \frac{\|h\|_{H^s}}{\e}+\|u\|_{H^s}\rt) \lt( \frac{\|\nabla h\|_{H^{s+\alpha-1}}^2}{\e^2} +\|\Lambda^\alpha u\|_{H^s}^2 \rt).
\end{aligned}
\eq
It remains to recover a zeroth-order dissipation for $h$. We compute
\begin{align*}
-\frac\rd\dt\intr h\nabla\cdot u\,\dx &= \intr \nabla\cdot((1+h)u)\nabla\cdot u\,\dx -\intr h\nabla\cdot\pa_t u\,\dx\\
&= \intr \nabla\cdot(hu)\nabla\cdot u\,\dx +\|\nabla\cdot u\|_{L^2}^2\\
&\quad -\intr \nabla h\cdot\lt( u\cdot\nabla u + \frac{\gamma}{\e^2}(1+h)^{\gamma-2}\nabla h + \Lambda^{2\alpha}u  + \Lambda^{2\alpha}(hu) - u\Lambda^{2\alpha}h \rt)\dx .
\end{align*}
Thus, using the smallness of $h$ in $L^\infty$,
\begin{align*}
-\frac\rd\dt\intr h\nabla\cdot u\,\dx +\frac{\gamma}{\e^2}\|\nabla h\|_{L^2}^2 &\le C\|\nabla h\|_{L^2}\|u\|_{L^\infty}\|\nabla u\|_{L^2} +C\|h\|_{L^\infty}\|\nabla u\|_{L^2}^2\\
&\quad +\|\nabla \cdot u\|_{L^2}^2 +\frac{C\|h\|_{L^\infty}}{\e^2}\|\nabla h\|_{L^2}^2 +\|\nabla h\|_{L^2}\|\Lambda^{2\alpha}u\|_{L^2} + I,
\end{align*}
where
\[
I:=-\intr \nabla h\cdot \lt(\Lambda^{2\alpha}(hu)-u\Lambda^{2\alpha}h\rt)\dx .
\]
We estimate $I$ using $2\alpha\le1$. By duality and H\"older's inequality,
\[
I \le C\|\Lambda^{1+\alpha}h\|_{L^2}\|\Lambda^\alpha(hu)\|_{L^2} +C\|\nabla h\|_{L^2}\|u\|_{L^{\frac{1}{\frac12-\frac\alpha d} }}\|\Lambda^{2\alpha}h\|_{L^{\frac d\alpha}}.
\]
Using \eqref{low-phy-emb}  and the product estimate \eqref{eq:hu-low}, we obtain
\[
I \le C\|h\|_{H^s} \|\nabla h\|_{H^{s+\alpha-1}}\|\Lambda^\alpha u\|_{H^s}.
\]
Thus, by Young's inequality, and the smallness assumption \eqref{small-assumption-low}, we have
\bq\label{cro-zero-est-low}
\begin{aligned}
-\frac\rd\dt\intr h\nabla\cdot u\,\dx +\frac{\gamma}{2\e^2}\|\nabla h\|_{L^2}^2
&\le C\lt( \frac{\|h\|_{H^s}}{\e}+\|u\|_{H^s} \rt) \lt( \frac{\|\nabla h\|_{H^{s+\alpha-1}}^2}{\e^2} +\|\Lambda^\alpha u\|_{H^s}^2 \rt)\\
&\quad +C_1\lt( \|\Lambda^{2\alpha}u\|_{L^2}^2 +\|\nabla u\|_{L^2}^2 \rt).
\end{aligned}
\eq
This completes the zeroth-order estimates in the low-order fractional regime.

%
%
%
%
%
%

\subsection{High-order energy estimates}\label{ssec:low-h-ene}
 
We next derive the high-order energy estimates. We begin with the estimate for $h$. From the continuity equation,
\begin{align*}
\frac12\frac\rd\dt\intr (1+h)^{\gamma-3}|\Lambda^s h|^2\,\dx &= -\frac{\gamma-3}{2}\intr (1+h)^{\gamma-4} \nabla\cdot((1+h)u)|\Lambda^s h|^2\,\dx\\
&\quad -\intr (1+h)^{\gamma-3}\Lambda^s h \Lambda^s\lt[\nabla\cdot((1+h)u)\rt] \dx .
\end{align*}
We rewrite the right-hand side as
\begin{align*}
\frac12\frac\rd\dt\intr (1+h)^{\gamma-3}|\Lambda^s h|^2\,\dx &= \frac{4-\gamma}{2}\intr (1+h)^{\gamma-3} (\nabla\cdot u)|\Lambda^s h|^2\,\dx\\
&\quad -\intr (1+h)^{\gamma-3}\Lambda^s h \lt[ \Lambda^s\nabla\cdot((1+h)u) -u\cdot\Lambda^s\nabla h -(1+h)\Lambda^s\nabla\cdot u \rt] \dx\\
&\quad +(\gamma-2)\intr  (1+h)^{\gamma-3} (\nabla h\cdot\Lambda^s u)\Lambda^s h\,\dx\\
&\quad +\intr (1+h)^{\gamma-2}\nabla\Lambda^s h\cdot\Lambda^s u\,\dx .
\end{align*}
Using the smallness of $h$ in $L^\infty$, Sobolev embedding, and \eqref{low-phy-emb}, we have
\begin{align*}
\lt|\intr (1+h)^{\gamma-3} (\nabla\cdot u)|\Lambda^s h|^2\,\dx\rt|
&\le C\|\nabla u\|_{L^{\frac d\alpha}} \|\Lambda^s h\|_{L^2} \|\Lambda^s h\|_{L^{\frac{1}{\frac12-\frac\alpha d}}} \le C\|\Lambda^\alpha u\|_{H^s} \|\nabla h\|_{H^{s+\alpha-1}}\|h\|_{H^s}
\end{align*}
and
\begin{align*}
\lt|\intr (1+h)^{\gamma-3}
(\nabla h\cdot\Lambda^s u)\Lambda^s h\,\dx\rt|
&\le C\|\nabla h\|_{L^{\frac d\alpha}}
\|\Lambda^s u\|_{L^{\frac{1}{\frac12-\frac\alpha d}}}
\|\Lambda^s h\|_{L^2} \le C\|\Lambda^\alpha u\|_{H^s}
\|\nabla h\|_{H^{s+\alpha-1}}\|h\|_{H^s}.
\end{align*}
It remains to estimate the commutator term. We claim that
\[
\|\Lambda^s\nabla\cdot((1+h)u) -u\cdot\Lambda^s\nabla h -(1+h)\Lambda^s\nabla\cdot u\|_{L^2} \le C\|\Lambda^\alpha u\|_{H^s}\|\nabla h\|_{H^{s+\alpha-1}}.
\]
Indeed, in the Fourier variables, the corresponding symbol is controlled by
\bq\label{eq:four_bdd_3term}
\lt||\xi|^s\xi-|\eta|^s\eta-|\xi-\eta|^s(\xi-\eta)\rt| \le C\lt(|\eta|^s|\xi-\eta|+|\eta||\xi-\eta|^s\rt).
\eq
Then, by Young's convolution inequality with $\frac1p+\frac1{p'}=\frac32$ and Lemma \ref{lem:low-four-inte},
\begin{align*}
\|\Lambda^s\nabla\cdot((1+h)u) -u\cdot\Lambda^s\nabla h -(1+h)\Lambda^s\nabla\cdot u\|_{L^2} 
&\le C\||\xi|^s\hat u\|_{L_\xi^p} \||\xi|\hat h\|_{L_\xi^{p'}} +C\||\xi|\hat u\|_{L_\xi^{p'}}\||\xi|^s \hat h\|_{L_\xi^p}\\
&\le C\|\Lambda^\alpha u\|_{H^s} \|\nabla h\|_{H^{s+\alpha-1}}.
\end{align*}
Thus,
\[
\frac12\frac\rd\dt\intr (1+h)^{\gamma-3}|\Lambda^s h|^2\,\dx \le  C\|h\|_{H^s}\|\Lambda^\alpha u\|_{H^s} \|\nabla h\|_{H^{s+\alpha-1}} +\intr (1+h)^{\gamma-2}\nabla\Lambda^s h\cdot\Lambda^s u\,\dx .
\]
Multiplying by $\frac{\gamma}{\e^2}$ and using Young's inequality, we get
\bq\label{h-high-est-low}
\begin{aligned}
\frac12\frac\rd\dt\lt( \frac{\gamma}{\e^2}\intr (1+h)^{\gamma-3}|\Lambda^s h|^2\,\dx \rt)
&\le C\lt( \frac{\|h\|_{H^s}}{\e}+\|u\|_{H^s} \rt) \lt( \frac{\|\nabla h\|_{H^{s+\alpha-1}}^2}{\e^2} +\|\Lambda^\alpha u\|_{H^s}^2 \rt)\\
&\quad +\frac{\gamma}{\e^2} \intr (1+h)^{\gamma-2}\nabla\Lambda^s h\cdot\Lambda^s u\,\dx .
\end{aligned}
\eq
We now estimate the high-order velocity energy. Taking the $L^2$ inner product of the second equation in \eqref{h-u-system} with $\Lambda^{2s}u$, we have
\begin{align*}
\frac12\frac\rd\dt\|\Lambda^s u\|_{L^2}^2 +\|\Lambda^{s+\alpha}u\|_{L^2}^2
&= \frac12\intr (\nabla\cdot u)|\Lambda^s u|^2\,\dx  -\intr \Lambda^s u\cdot \lt[\Lambda^s(u\cdot\nabla u)-u\cdot\nabla\Lambda^s u\rt] \dx\\
&\quad -\frac{\gamma}{\e^2} \intr (1+h)^{\gamma-2}\nabla\Lambda^s h\cdot\Lambda^s u\,\dx\\
&\quad -\frac{\gamma}{\e^2}\intr \Lambda^s u\cdot \lt[ \Lambda^s((1+h)^{\gamma-2}\nabla h) -(1+h)^{\gamma-2}\nabla\Lambda^s h \rt] \dx\\
&\quad -\intr \Lambda^{s+\alpha}u\cdot \lt[\Lambda^{s+\alpha}(hu)-u\Lambda^{s+\alpha}h\rt] \dx\\
&\quad +\intr \Lambda^s u\cdot \lt[\Lambda^s(u\Lambda^{2\alpha}h) -\Lambda^\alpha(u\Lambda^{s+\alpha}h)\rt] \dx\\
&=: \sum_{i=1}^6 \sfI_i .
\end{align*}
For $\sfI_1$, by H\"older's inequality and Sobolev embedding,
\[
\sfI_1 \le C\|\nabla u\|_{L^{\frac d\alpha}} \|\Lambda^s u\|_{L^2} \|\Lambda^s u\|_{L^{\frac{1}{\frac12-\frac\alpha d}}} \le C\|u\|_{H^s}\|\Lambda^\alpha u\|_{H^s}^2.
\]
For $\sfI_2$, we use the pointwise Fourier bound
\bq\label{eq:four_bdd_l2}
\lt||\xi|^s\eta_j-|\eta|^s\eta_j\rt| \le C\lt(|\xi-\eta|^s|\eta|+|\xi-\eta||\eta|^s\rt), \quad j=1,\dots,d.
\eq
Thus, by Young's convolution inequality and Lemma \ref{lem:low-four-inte},
\[
\sfI_2 \le C\|\Lambda^s u\|_{L^2} \||\xi|^s\hat u\|_{L_\xi^p} \||\xi|\hat u\|_{L_\xi^{p'}} \le C\|u\|_{H^s}\|\Lambda^\alpha u\|_{H^s}^2.
\]
The pressure coupling term is
\[
\sfI_3 =-\frac{\gamma}{\e^2} \intr (1+h)^{\gamma-2}\nabla\Lambda^s h\cdot\Lambda^s u\,\dx .
\]
For $\sfI_4$, set $g:=(1+h)^{\gamma-2}-1$. Then

\[
\Lambda^s((1+h)^{\gamma-2}\nabla h) -(1+h)^{\gamma-2}\nabla\Lambda^s h = \Lambda^s(g\nabla h)-g\nabla\Lambda^s h.
\]
Using the pointwise bound
\[
\lt||\xi|^s-|\eta|^s\rt| |\eta| \le C\lt(|\xi-\eta|^s|\eta|+|\xi-\eta||\eta|^s\rt),
\]
Young's convolution inequality gives
\[
|\sfI_4| \le \frac{C}{\e^2} \|\Lambda^s g\|_{L^2} \||\xi|^s\hat u\|_{L_\xi^p} \||\xi|\hat h\|_{L_\xi^{p'}} +\frac{C}{\e^2} \|\Lambda^s h\|_{L^2} \||\xi|^s\hat u\|_{L_\xi^p} \||\xi|\hat g\|_{L_\xi^{p'}}.
\]
By Lemma \ref{lem:low-four-inte}, and Remark \ref{rem:low-four-cons}, we obtain
\[
|\sfI_4| \le \frac{C}{\e^2} \|h\|_{H^s} \|\Lambda^\alpha u\|_{H^s} \|\nabla h\|_{H^{s+\alpha-1}}.
\]
Hence, by Young's inequality,
\[
|\sfI_4| \le C\lt( \frac{\|h\|_{H^s}}{\e}+\|u\|_{H^s} \rt) \lt( \frac{\|\nabla h\|_{H^{s+\alpha-1}}^2}{\e^2} +\|\Lambda^\alpha u\|_{H^s}^2 \rt).
\]
For $\sfI_5$, we use the pointwise bound

\bq\label{eq:four_bdd_sa}
\lt||\xi|^{s+\alpha}-|\eta|^{s+\alpha}\rt| \le C\lt( |\xi-\eta|^{s+\alpha} +|\xi-\eta||\eta|^{s+\alpha-1} \rt).
\eq
Thus, by Lemma \ref{lem:low-four-inte}
\begin{align*}
|\sfI_5| &\le \|\Lambda^{s+\alpha}u\|_{L^2} \|\Lambda^{s+\alpha}(hu)-u\Lambda^{s+\alpha}h\|_{L^2}\\
&\le C\|\Lambda^{s+\alpha}u\|_{L^2} \lt( \|\Lambda^{s+\alpha}u\|_{L^2}\|\hat h\|_{L_\xi^1} +\||\xi|\hat u\|_{L_\xi^{p'}} \||\xi|^{s+\alpha-1}\hat h\|_{L_\xi^p} \rt)\\
&\le C\|h\|_{H^s}\|\Lambda^\alpha u\|_{H^s}^2.
\end{align*}
For $\sfI_6$, we use the pointwise estimate

\begin{align*}
\lt||\xi|^s|\eta|^{2\alpha} -|\xi|^\alpha|\eta|^{s+\alpha}\rt|
&= |\xi|^\alpha|\eta|^{2\alpha} \lt||\xi|^{s-\alpha}-|\eta|^{s-\alpha}\rt|\\
&\le C|\xi|^\alpha|\eta|^{2\alpha}|\xi-\eta| \lt( |\xi|^{s-\alpha-1}+|\eta|^{s-\alpha-1} \rt)\\
&\le C|\xi-\eta|^s|\eta|^{2\alpha} +C|\xi-\eta||\eta|^{s+2\alpha-1} +C|\xi|^\alpha|\xi-\eta||\eta|^{s+\alpha-1}.
\end{align*}
Hence, by Lemma \ref{lem:low-four-inte}, we estimate
\begin{align*}
\sfI_6 &\le C\intrr |\xi|^s |\hat u(\xi)| |\xi-\eta|^s |\hat u(\xi-\eta)| |\eta|^{2\alpha} |\hat h(\eta)|\,\rd\eta \dxi\\
&\quad + C\intrr |\xi|^s |\hat u(\xi)| |\xi-\eta| |\hat u(\xi-\eta)|  |\eta|^{s+2\alpha-1} |\hat h(\eta)| \,\rd\eta \dxi\\
&\quad+ C\intrr |\xi|^{s+\alpha} |\hat u(\xi)| |\xi-\eta||\hat u(\xi-\eta)| |\eta|^{s+\alpha-1} |\hat h(\eta)|\,\rd\eta \dxi\\
&\le C\| \Lambda^s u\|_{L^2} \| |\xi|^s \hat u\|_{L_\xi^p} \| |\xi|^{2\alpha} \hat h\|_{L_\xi^{p'}}  + C\| |\xi|^s \hat u\|_{L_\xi^p} \| |\xi|\hat u\|_{L_\xi^{p'}} \|\Lambda^{s+2\alpha-1} h\|_{L^2}\\
&\quad + C\|\Lambda^{s+\alpha} u\|_{L^2}\| |\xi| \hat u\|_{L_\xi^{p'}}\| |\xi|^{s+\alpha-1} \hat h\|_{L_\xi^p}\\
&\le  C\|u\|_{H^s} \|\Lambda^\alpha u\|_{H^s}\|\nabla h\|_{H^{s+\alpha-1}}+C\|h\|_{H^s} \|\Lambda^\alpha u\|_{H^s}^2.
\end{align*}
Collecting the estimates for $\sfI_i$, we obtain

\begin{align*}
\frac12\frac\rd\dt\|\Lambda^s u\|_{L^2}^2 +\|\Lambda^{s+\alpha}u\|_{L^2}^2
&\le C\lt( \frac{\|h\|_{H^s}}{\e}+\|u\|_{H^s} \rt) \lt( \frac{\|\nabla h\|_{H^{s+\alpha-1}}^2}{\e^2} +\|\Lambda^\alpha u\|_{H^s}^2 \rt)\\
&\quad -\frac{\gamma}{\e^2} \intr (1+h)^{\gamma-2} \nabla\Lambda^s h\cdot\Lambda^s u\,\dx .
\end{align*}
We finally combine this with \eqref{h-high-est-low}, the pressure coupling terms cancel and hence
\bq\label{high-ene-l-alp-est}
\begin{aligned}
&\frac12\frac\rd\dt\lt( \frac{\gamma}{\e^2} \intr (1+h)^{\gamma-3}|\Lambda^s h|^2\,\dx +\|\Lambda^s u\|_{L^2}^2 \rt) +\|\Lambda^{s+\alpha}u\|_{L^2}^2 \\
&\quad \le C\lt( \frac{\|h\|_{H^s}}{\e}+\|u\|_{H^s} \rt)  \lt( \frac{\|\nabla h\|_{H^{s+\alpha-1}}^2}{\e^2} +\|\Lambda^\alpha u\|_{H^s}^2 \rt).
\end{aligned}
\eq

%
%
%
%
%
%

\subsection{High-order density dissipation estimate}\label{ssec:low-den-diss}

We now recover the high-order dissipation for $h$ in the case
$0<\alpha\le\frac12$. We consider the mixed quantity
\[
\intr \Lambda^{s+\alpha-1}h\,
\Lambda^{s+\alpha-1}\nabla\cdot u\,\dx
=
\intr \Lambda^{s+2\alpha-1}h\,
\Lambda^{s-1}\nabla\cdot u\,\dx .
\]
Using \eqref{h-u-system}, we compute
\begin{align*}
& -\frac\rd\dt
\intr \Lambda^{s+\alpha-1}h\,
\Lambda^{s+\alpha-1}\nabla\cdot u\,\dx\\
&\quad = \intr \Lambda^{s+\alpha-1}\nabla\cdot((1+h)u)\, \Lambda^{s+\alpha-1}\nabla\cdot u\,\dx  -\intr \Lambda^{s+\alpha-1}h\, \Lambda^{s+\alpha-1}\nabla\cdot\pa_t u\,\dx\\
&\quad = \intr \nabla\cdot\lt( \Lambda^{s+\alpha-1}(hu) -u\Lambda^{s+\alpha-1}h \rt) \Lambda^{s+\alpha-1}\nabla\cdot u\,\dx +\|\Lambda^{s+\alpha-1}\nabla\cdot u\|_{L^2}^2\\
&\quad\quad +\intr \Lambda^{s+\alpha-1}h \lt[ \Lambda^{s+\alpha-1}\nabla\cdot(u\cdot\nabla u) -u\cdot\nabla\Lambda^{s+\alpha-1}\nabla\cdot u \rt] \dx  -\frac{\gamma}{\e^2}\|\Lambda^{s+\alpha}h\|_{L^2}^2\\
&\quad\quad -\frac{\gamma}{\e^2} \intr \nabla\Lambda^{s+\alpha-1}h\cdot \Lambda^{s+\alpha-1}\lt[ \lt((1+h)^{\gamma-2}-1\rt)\nabla h \rt] \dx\\
&\quad\quad +\intr \Lambda^{s+2\alpha-1}h\, \Lambda^{s+2\alpha-1}\nabla\cdot((1+h)u)\,\dx\\
&\quad\quad -\intr \Lambda^{s+\alpha-1}h\, \Lambda^{s+\alpha-1}\nabla\cdot(u\Lambda^{2\alpha}h)\,\dx\\
&\quad =: \sum_{i=1}^7 \sfJ_i .
\end{align*}
For $\sfJ_1$, we write
\begin{align*}
\sfJ_1 &= \intr \nabla\cdot\Lambda^{s+\alpha-1}(hu)\, \Lambda^{s+\alpha-1}\nabla\cdot u\,\dx  -\intr (\nabla\cdot u)\Lambda^{s+\alpha-1}h\, \Lambda^{s+\alpha-1}\nabla\cdot u\,\dx\\
&\quad -\intr u\cdot\nabla\Lambda^{s+\alpha-1}h\, \Lambda^{s+\alpha-1}\nabla\cdot u\,\dx .
\end{align*}
Thus, by the product estimate and Sobolev embedding,
\begin{align*}
|\sfJ_1| &\le C\|\Lambda^{s+\alpha}u\|_{L^2} \lt( \|h\|_{L^\infty}\|\Lambda^{s+\alpha}u\|_{L^2} +\|u\|_{L^\infty}\|\Lambda^{s+\alpha}h\|_{L^2} \rt)  +C\|\nabla u\|_{L^{\frac d\alpha}}
\|\Lambda^{s+\alpha-1}h\|_{L^{\frac{1}{\frac12-\frac\alpha d}}}
\|\Lambda^{s+\alpha}u\|_{L^2}\\
&\le C\lt( \frac{\|h\|_{H^s}}{\e}+\|u\|_{H^s} \rt) \lt( \frac{\|\nabla h\|_{H^{s+\alpha-1}}^2}{\e^2} +\|\Lambda^\alpha u\|_{H^s}^2 \rt).
\end{align*}
The second term is simply estimated by
\[
\sfJ_2 =\|\Lambda^{s+\alpha-1}\nabla\cdot u\|_{L^2}^2\le C\|\Lambda^{s+\alpha} u\|_{L^2}^2
\]
For $\sfJ_3$, we use the pointwise Fourier bound
\[
\lt||\xi|^{s+\alpha-1}\xi_j -|\eta|^{s+\alpha-1}\eta_j\rt| \le C|\xi-\eta|^{s+\alpha} +C|\xi-\eta||\eta|^{s+\alpha-1}, \quad j=1,\dots,d.
\]
Then Young's convolution inequality and Lemma \ref{lem:low-four-inte} give
\begin{align*}
|\sfJ_3| &\le C\intrr   |\xi|^{s+\alpha-1}|\hat h(\xi)| |\xi-\eta|^{s+\alpha}|\hat u(\xi-\eta)| |\eta||\hat u(\eta)|\,\rd\eta\,\rd\xi\\
&\quad +C\intrr    |\xi|^{s+\alpha-1}|\hat h(\xi)| |\xi-\eta||\hat u(\xi-\eta)| |\eta|^{s+\alpha}|\hat u(\eta)|\,\rd\eta\,\rd\xi\\
&\le C\|\Lambda^{s+\alpha}u\|_{L^2} \||\xi|^{s+\alpha-1}\hat h\|_{L_\xi^p} \||\xi|\hat u\|_{L_\xi^{p'}}\\
&\le C\|h\|_{H^s}\|\Lambda^\alpha u\|_{H^s}^2.
\end{align*}
The fourth term gives the desired dissipation:
\[
\sfJ_4=-\frac{\gamma}{\e^2}\|\Lambda^{s+\alpha}h\|_{L^2}^2.
\]
For the pressure commutator term, we set $g=(1+h)^{\gamma-2}-1$. Then, by Lemma \ref{lem:low-four-inte} and Remark \ref{rem:low-four-cons},
\begin{align*}
|\sfJ_5| &\le \frac{C}{\e^2} \|\Lambda^{s+\alpha}h\|_{L^2} \|\Lambda^{s+\alpha-1}(g\nabla h)\|_{L^2}\\
&\le \frac{C}{\e^2} \|\Lambda^{s+\alpha}h\|_{L^2} \lt( \|\hat g\|_{L_\xi^1}\|\Lambda^{s+\alpha}h\|_{L^2} +\||\xi|^{s+\alpha-1}\hat g\|_{L_\xi^p} \||\xi|\hat h\|_{L_\xi^{p'}} \rt)\\
&\le \frac{C}{\e^2}\|h\|_{H^s}\|\nabla h\|_{H^{s+\alpha-1}}^2 .
\end{align*}
For $\sfJ_6$, we use the continuity equation:
\[
\sfJ_6 = \intr \Lambda^{s+2\alpha-1}h\, \Lambda^{s+2\alpha-1}\nabla\cdot((1+h)u)\,\dx = -\intr \Lambda^{s+2\alpha-1}h\, \Lambda^{s+2\alpha-1}\pa_t h\,\dx = -\frac12\frac\rd\dt
\|\Lambda^{s+2\alpha-1}h\|_{L^2}^2 .
\]
Finally, for $\sfJ_7$, since $0<\alpha\le\frac12$, we have $1\le s+3\alpha-1\le s+\alpha$. Thus, by Lemma \ref{lem:low-four-inte},
\begin{align*}
|\sfJ_7| &= \lt|\intr \nabla\Lambda^{s+\alpha-1}h\cdot \Lambda^{s+\alpha-1}(u\Lambda^{2\alpha}h)\,\dx\rt|\\
&\le C\|\Lambda^{s+\alpha}h\|_{L^2} \lt( \||\xi|^{s+\alpha-1}\hat u\|_{L_\xi^p} \||\xi|^{2\alpha}\hat h\|_{L_\xi^{p'}} +\|\hat u\|_{L_\xi^1} \|\Lambda^{s+3\alpha-1}h\|_{L^2} \rt)\\
&\le C\|u\|_{H^s} \|\nabla h\|_{H^{s+\alpha-1}}^2.
\end{align*}
Collecting the estimates for $\sfJ_i$, we obtain
\bq\label{den-di-l-alp}
\begin{aligned}
\frac\rd\dt&\lt( \frac12\|\Lambda^{s+2\alpha-1}h\|_{L^2}^2 -\intr \Lambda^{s+2\alpha-1}h\, \Lambda^{s-1}\nabla\cdot u\,\dx \rt) +\frac{\gamma}{2\e^2}\|\Lambda^{s+\alpha}h\|_{L^2}^2\\
&\le C\lt( \frac{\|h\|_{H^s}}{\e}+\|u\|_{H^s} \rt) \lt( \frac{\|\nabla h\|_{H^{s+\alpha-1}}^2}{\e^2} +\|\Lambda^\alpha u\|_{H^s}^2 \rt) +\|\Lambda^{s+\alpha-1}\nabla\cdot u\|_{L^2}^2 .
\end{aligned}
\eq

%
%
%
%
%
%

\subsection{Completion of the low-order estimate} \label{ssec:low-comp}

We now combine the estimates obtained above. Adding \eqref{zero-e-est-low} and \eqref{high-ene-l-alp-est}, and then adding a small multiple of \eqref{cro-zero-est-low} and
\eqref{den-di-l-alp}, we obtain, for $\theta>0$ sufficiently small,
\begin{align*}
&\frac\rd\dt\bigg[ \frac{\gamma}{2\e^2} \intr (1+h)^{\gamma-3} \lt(|h|^2+|\Lambda^s h|^2\rt)\,\dx +\frac12\lt(\|u\|_{L^2}^2+\|\Lambda^s u\|_{L^2}^2\rt)\\
&\quad\quad\quad +\theta\|\Lambda^{s+2\alpha-1}h\|_{L^2}^2 -2\theta\intr h\nabla\cdot u\,\dx -2\theta\intr \Lambda^{s+2\alpha-1}h\, \Lambda^{s-1}\nabla\cdot u\,\dx \bigg]\\
&\quad\quad\quad +\frac{\theta\gamma}{2\e^2} \lt( \|\nabla h\|_{L^2}^2 +\|\Lambda^{s+\alpha}h\|_{L^2}^2 \rt) +\frac12\lt( \|\Lambda^\alpha u\|_{L^2}^2 +\|\Lambda^{s+\alpha}u\|_{L^2}^2 \rt)\\
&\quad\le C\lt( \frac{\|h\|_{H^s}}{\e}+\|u\|_{H^s} \rt) \lt( \frac{\|\nabla h\|_{H^{s+\alpha-1}}^2}{\e^2} +\|\Lambda^\alpha u\|_{H^s}^2 \rt).
\end{align*}
Indeed, the terms $\|\Lambda^{2\alpha}u\|_{L^2}^2$, $\|\nabla\cdot u\|_{L^2}^2$, and  $\|\Lambda^{s+\alpha-1}\nabla\cdot u\|_{L^2}^2$, which appear in \eqref{cro-zero-est-low} and \eqref{den-di-l-alp} are absorbed by $\|\Lambda^\alpha u\|_{L^2}^2+\|\Lambda^{s+\alpha}u\|_{L^2}^2$ after choosing $\theta>0$ sufficiently small.

Since $\|h\|_{L^\infty}$ is sufficiently small, we have
\[
\intr (1+h)^{\gamma-3} \lt(|h|^2+|\Lambda^s h|^2\rt) \dx \simeq \|h\|_{H^s}^2.
\]
Moreover, by Young's inequality, for any sufficiently small $\delta>0$,
\[
\lt|\intr h\nabla\cdot u\,\dx\rt| \le \delta\frac{\|h\|_{L^2}^2}{\e^2} +C_\delta\|u\|_{H^s}^2,
\]
and
\[
\lt|\intr \Lambda^{s+2\alpha-1}h\, \Lambda^{s-1}\nabla\cdot u\,\dx\rt| \le \delta\frac{\|h\|_{H^s}^2}{\e^2} +C_\delta\|u\|_{H^s}^2,
\]
where we used $s+2\alpha-1\le s$ and $0<\e\le1$. Hence, if $\theta>0$ is chosen sufficiently small, the modified energy
\begin{align*}
\calE_\theta(t) &:=\frac{\gamma}{2\e^2} \intr (1+h)^{\gamma-3} \lt(|h|^2+|\Lambda^s h|^2\rt)\,\dx +\frac12\lt(\|u\|_{L^2}^2+\|\Lambda^s u\|_{L^2}^2\rt)\\
&\quad +\theta\|\Lambda^{s+2\alpha-1}h\|_{L^2}^2 -2\theta\intr h\nabla\cdot u\,\dx -2\theta\intr \Lambda^{s+2\alpha-1}h\, \Lambda^{s-1}\nabla\cdot u\,\dx
\end{align*}
is equivalent to
\[
\frac{\|h\|_{H^s}^2}{\e^2}+\|u\|_{H^s}^2.
\]
Similarly,
\[
\|\nabla h\|_{L^2}^2+\|\Lambda^{s+\alpha}h\|_{L^2}^2 \simeq \|\nabla h\|_{H^{s+\alpha-1}}^2,
\]
up to lower-order terms already controlled by the modified energy. Hence, for some constants $c_0,C_0>0$ independent of $T$ and $\e$, we have
\[
c_0\lt( \frac{\|h\|_{H^s}^2}{\e^2}+\|u\|_{H^s}^2 \rt) \le \calE_\theta(t) \le C_0\lt( \frac{\|h\|_{H^s}^2}{\e^2}+\|u\|_{H^s}^2 \rt),
\]
and
\[
\frac\rd\dt\calE_\theta(t) +c_0\lt( \frac{\|\nabla h\|_{H^{s+\alpha-1}}^2}{\e^2} +\|\Lambda^\alpha u\|_{H^s}^2 \rt) \le C\lt( \frac{\|h\|_{H^s}}{\e}+\|u\|_{H^s} \rt) \lt( \frac{\|\nabla h\|_{H^{s+\alpha-1}}^2}{\e^2} +\|\Lambda^\alpha u\|_{H^s}^2 \rt).
\]
By the smallness assumption \eqref{small-assumption-low}, if $\eta>0$ is sufficiently small, the right-hand side is absorbed by the dissipation term. Consequently,
\[
\frac\rd\dt\calE_\theta(t) +\frac{c_0}{2}\lt( \frac{\|\nabla h\|_{H^{s+\alpha-1}}^2}{\e^2} +\|\Lambda^\alpha u\|_{H^s}^2 \rt) \le0.
\]
Integrating over $[0,t]$, we conclude that
\[
\frac{\|h(t)\|_{H^s}^2}{\e^2} +\|u(t)\|_{H^s}^2 +\int_0^t\lt( \frac{\|\nabla h(\tau)\|_{H^{s+\alpha-1}}^2}{\e^2} +\|\Lambda^\alpha u(\tau)\|_{H^s}^2 \rt)\rd\tau \le C\lt( \frac{\|h_0\|_{H^s}^2}{\e^2} +\|u_0\|_{H^s}^2 \rt),
\]
where $C>0$ is independent of $T$ and $\e$. This proves Proposition \ref{prop:uni-apri-low}.

%
%
%
%
%
%
 \section{Uniform estimates in the higher-order fractional regime}\label{sec:high-alpha}

In this section, we provide the uniform a priori estimate in the higher-order fractional alignment regime
\[
\frac12<\alpha<1.
\]
Here the alignment dissipation has order $2\alpha>1$, so the commutator estimates are closer to the standard high-order Sobolev estimates than in the low-order regime treated in Section \ref{sec:low-alpha}. Throughout this section, we assume
\[
s>\frac d2.
\]
We continue to use the formulation \eqref{h-u-system}. The density dissipation will be recovered from the singular pressure term at the level $\frac1{\e^2}\|\nabla h\|_{H^{s-\alpha}}^2$. 

\begin{proposition}\label{prop:uni-apri-high}
Let $d\ge2$, $\frac12<\alpha<1$, and $\gamma\ge1$. Assume that $s>\frac d2$. Let $(h,u)$ be a smooth solution to \eqref{h-u-system} on $[0,T]$. Suppose that
\bq\label{small-assumption-high}
\sup_{0\le t\le T} \lt( \frac{\|h(t)\|_{H^s}}{\e}+\|u(t)\|_{H^s} \rt) \le \eta
\eq
for some sufficiently small $\eta>0$. Then, for all $t\in[0,T]$,
\[
\frac{\|h(t)\|_{H^s}^2}{\e^2} +\|u(t)\|_{H^s}^2 +\int_0^t\lt( \frac{\|\nabla h(\tau)\|_{H^{s-\alpha}}^2}{\e^2} +\|\Lambda^\alpha u(\tau)\|_{H^s}^2 \rt)\rd\tau  \le C\lt( \frac{\|h_0\|_{H^s}^2}{\e^2} +\|u_0\|_{H^s}^2 \rt),
\]
where $C>0$ is independent of $T$ and $\e$.
\end{proposition}

%
%
%
%
%
%
 \subsection{Fourier interpolation estimates}

We provide the Fourier interpolation estimates used in the higher-order fractional regime. They play the same role as Lemma \ref{lem:low-four-inte}, but the relevant density dissipation level is now $\|\nabla h\|_{H^{s-\alpha}}$.

\begin{lemma}\label{lem:high-four-inter}
Let $d\ge2$, $\frac12<\alpha<1$, and $s>\frac d2$. Let
\[
q_1 := \frac{2s}{s+\alpha}, \quad q_1' := \frac{2s}{2s-\alpha}, \quad q_2 := \frac{2s}{s-\alpha+1}, \quad q_2' := \frac{2s}{2s+\alpha-1}, \quad q_3 := \frac{2s}{2s-1}.
\]
Then, for any sufficiently regular $f$, we have
\bq\label{high-four-vel-gen}
\||\xi|^\alpha \hat f\|_{L_\xi^1}  +\||\xi|^s\hat f\|_{L_\xi^{q_1}} +\||\xi|\hat f\|_{L_\xi^{ q_2'}} +\||\xi|^{1+\alpha}\hat f\|_{L_\xi^{q_3}} \le C\|\Lambda^\alpha f\|_{H^s}.
\eq
Moreover,
\bq\label{high-four-den-d-gen}
\| |\xi|^{1-\alpha} \hat f\|_{L_\xi^1} +\| |\xi|^{s-2\alpha+1} \hat f\|_{L_\xi^{q_1}}+   \| |\xi| \hat f\|_{L_\xi^{q_1'}} + \| |\xi|^s \hat f \|_{L_\xi^{q_2}}\le C\|\nabla f\|_{H^{s-\alpha}},
\eq
and
\bq\label{high-four-Hs-gen}
\| \hat f\|_{L_\xi^1} + \||\xi|^{s-\alpha} \hat f\|_{L_\xi^{q_1}} + \||\xi|^\alpha \hat f\|_{L_\xi^{q_1'}}+\||\xi|^{s+\alpha-1}\hat f\|_{L_\xi^{q_2}} \le C\|f\|_{H^s}.
\eq
\end{lemma}

\begin{proof}
By the definition of $q_i$, and $q_i'$, we have $q_i, q_i'\in(1,2)$, $\frac1{q_i}+\frac1{q_i'}=\frac32$ and $\frac1{q_1} + \frac1{q_2} + \frac{1}{q_3} = 2$. We shall use the identities
\[
\frac{sq_1 - \alpha (2-q_1)}{q_1 - 1} = 2s+2\alpha, \quad \frac{(s-\alpha)q_1}{q_1 -1} = 2s, \quad \frac{q_1' - (1-\alpha) (2-q_1')}{q_1 ' -1} = 2s+2-2\alpha, \quad \frac{\alpha q_1'}{q_1' -1} =2s,
\]
\[
\frac{(s-2\alpha+1)q_1 - (1-\alpha)(2-q_1)}{q_1-1} = 2s+2-2\alpha, \quad \frac{sq_2 - (1-\alpha)(2-q_2)}{q_2 -1} = 2s +2 -2\alpha,  
\]
and
\[
\frac{(s+\alpha-1)q_2}{q_2 -1} = 2s, \quad \frac{q_2' - \alpha(2-q_2')}{q_2'-1} = 2s+2\alpha, \quad \frac{(1+\alpha)q_3 -\alpha(2-q_3)}{q_3-1} = 2s+2\alpha.
\]
We first prove \eqref{high-four-vel-gen}. Since $s+\alpha>\frac d2+\alpha$, we may choose $\delta>0$ sufficiently small so that
\[
\frac d2+\alpha+\delta<s+\alpha.
\]
Then
\[
\||\xi|^\alpha\hat f\|_{L_\xi^1} \le C \|\Lambda^{\frac d2+\alpha-\delta}f\|_{L^2}^{1/2} \|\Lambda^{\frac d2+\alpha+\delta}f\|_{L^2}^{1/2} \le C\|\Lambda^\alpha f\|_{H^s}.
\]
Next, using the interpolation inequality
\[
\||\xi|^a\hat f\|_{L_\xi^r} \le \||\xi|^b\hat f\|_{L_\xi^1}^{\frac{2-r}{r}} \lt\| |\xi|^{\frac{ar-b(2-r)}{2(r-1)}}\hat f \rt\|_{L_\xi^2}^{\frac{2(r-1)}{r}}, \quad 1<r<2,
\]
with $(a,b,r)=(s,\alpha,q_1)$, we obtain
\[
\||\xi|^s\hat f\|_{L_\xi^{q_1}} \le \||\xi|^\alpha\hat f\|_{L_\xi^1}^{\frac{2- q_1}{q_1}} \|\Lambda^{s+\alpha}f\|_{L^2}^{\frac{2(q_1-1)}{q_1}} \le C\|\Lambda^\alpha f\|_{H^s}.
\]
With $(a,b,r)=(1,\alpha, q_2')$, the $L^2_\xi$ order is $s+\alpha$, and hence
\[
\||\xi|\hat f\|_{L_\xi^{q_2'}} \le C\|\Lambda^\alpha f\|_{H^s}.
\]
Finally, with $(a,b,r)=(1+\alpha,\alpha,q_3)$, the $L^2_\xi$ order is $s+\alpha$. Thus
\[
  \||\xi|^{1+\alpha}\hat f\|_{L_\xi^{q_3}}   \le C\|\Lambda^\alpha f\|_{H^s}.
\]
This proves \eqref{high-four-vel-gen}.

We now prove \eqref{high-four-den-d-gen}. Since $s+1-\alpha>\frac d2+1-\alpha>1$, we may choose $\delta>0$ sufficiently small so that
\[
\frac d2+1-\alpha+\delta<s+1-\alpha, \quad 1 \le \frac d2 +1 -\alpha-\delta.
\]
Then
\[
\||\xi|^{1-\alpha}\hat f\|_{L_\xi^1} \le C \|\Lambda^{\frac d2+1-\alpha-\delta}f\|_{L^2}^{1/2} \|\Lambda^{\frac d2+1-\alpha+\delta}f\|_{L^2}^{1/2} \le C\|\nabla f\|_{H^{s-\alpha}}.
\]
With  $(a,b,r)=(s-2\alpha+1,1-\alpha, q_1)$ and $(a,b,r)=(1,1-\alpha, q_1')$, the $L^2_\xi$ order is $s+1-\alpha$ and this implies
\[
\||\xi|^{s-2\alpha+1} \hat f\|_{L_\xi^{q_1}}\le C\|\nabla f\|_{H^{s-\alpha}},
\]
and
\[
 \||\xi|\hat f\|_{L_\xi^{q_1'}} \le C\|\nabla f\|_{H^{s-\alpha}}.
\]
Moreover, with  $(a,b,r)=(s,1-\alpha, q_2)$  the $L^2_\xi$ order is again $s+1-\alpha$. Hence,
\[
\||\xi|^s \hat f\|_{L_\xi^{q_2}} \le C\|\nabla f\|_{H^{s-\alpha}}.
\]
It remains to prove \eqref{high-four-Hs-gen}. Since $s>\frac d2$, we get
\[
\|\hat f\|_{L_\xi^1}\le C\|f\|_{H^s}.
\]
With $(a,b,r)=(s-\alpha,0, q_1)$, $(\alpha,0, q_1')$, and $(s+\alpha-1,0, q_2)$, we have
\[
\| |\xi|^{s-\alpha} \hat f\|_{L_\xi^{q_1}} \le \|\hat f\|_{L_\xi^1}^{\frac{2-q_1}{q_1}} \| |\xi|^s \hat f\|_{L_\xi^2}^{\frac{2(q_1-1)}{q_1}} \le C\|f\|_{H^s},
\]
\[
\| |\xi|^\alpha \hat f\|_{L_\xi^{q_1'}} \le \| \hat f\|_{L_\xi^1}^{\frac{2-q_1'}{q_1'}} \| |\xi|^s \hat f\|_{L_\xi^2}^{\frac{2(q_1'-1)}{q_1'}} \le C\| f\|_{H^s},
\]
and
\[
\| |\xi|^{s+\alpha-1} \hat f\|_{L_\xi^{q_2}} \le \| \hat f\|_{L_\xi^1}^{\frac{2-q_2}{q_2}} \| |\xi|^s \hat f\|_{L_\xi^2}^{\frac{2(q_2-1)}{q_2}} \le C\| f\|_{H^s},
\]
respectively. This proves \eqref{high-four-Hs-gen} and completes the proof.
\end{proof}

\begin{remark}\label{rem:high-four-cons}
Let $g=(1+h)^{\gamma-2}-1$. Since $\|h\|_{L^\infty}$ is sufficiently small, the Moser estimate \eqref{eq:moser_den} gives $\|g\|_{H^s}\le C\|h\|_{H^s}$. Therefore, \eqref{high-four-Hs-gen} implies
\[
\|\hat g\|_{L_\xi^1} +\||\xi|^{s+\alpha-1}\hat g\|_{L_\xi^{q_2}} \le C\|h\|_{H^s}.
\]
In addition, by \eqref{high-four-den-d-gen} and the Moser estimate,
\[
\||\xi|\hat g\|_{L_\xi^{q_1'}} +\||\xi|^s \hat g\|_{L_\xi^{q_2}} \le C\|\nabla g\|_{H^{s-\alpha}} \le C\|\nabla h\|_{H^{s-\alpha}}.
\]
\end{remark}
We shall also use the following direct consequence of \eqref{high-four-den-d-gen}. Since $\alpha>\frac12$, we have
\[
\||\xi|\hat f\|_{L_\xi^{q_1'}}
\le C\|\nabla f\|_{H^{s-\alpha}}
\le C\|\Lambda^\alpha f\|_{H^s}.
\]

%
%
%
%
%
%

\subsection{Zeroth-order energy estimates}

We derive the zeroth-order estimates in the regime $\frac12<\alpha<1$. Since the algebraic identities are the same as in Section \ref{ssec:low-zero}, we only indicate the estimates at the density dissipation level
\[
\frac1{\e^2}\|\nabla h\|_{H^{s-\alpha}}^2.
\]
From the continuity equation, the same computation leading to \eqref{h-zero-est-low} gives
\bq\label{h-zero-est-high}
\begin{aligned}
\frac12\frac\rd\dt\lt( \frac{\gamma}{\e^2}\intr (1+h)^{\gamma-3}|h|^2\,\dx \rt)
&\le C\lt( \frac{\|h\|_{H^s}}{\e}+\|u\|_{H^s} \rt) \lt( \frac{\|\nabla h\|_{H^{s-\alpha}}^2}{\e^2} +\|\Lambda^\alpha u\|_{H^s}^2 \rt)\\
&\quad +\frac{\gamma}{\e^2} \intr (1+h)^{\gamma-2}\nabla h\cdot u\,\dx .
\end{aligned}
\eq
Indeed, compared with the low-order case, the relevant bounds are now
\[
\|\Lambda^{\frac d2}h\|_{L^2} \le C\|\nabla h\|_{H^{s-\alpha}}, \quad \|\nabla u\|_{L^2} +\|\Lambda^{\frac d2}u\|_{L^2} \le C\|\Lambda^\alpha u\|_{H^s},
\]
which follow from $s>\frac d2$ and $\frac12<\alpha<1$.

Next, taking the $L^2$ inner product of the velocity equation with $u$, we have
\begin{align*}
\frac12\frac\rd\dt\|u\|_{L^2}^2+\|\Lambda^\alpha u\|_{L^2}^2 &= -\intr (u\cdot\nabla u)\cdot u\,\dx -\frac{\gamma}{\e^2} \intr (1+h)^{\gamma-2}\nabla h\cdot u\,\dx\\
&\quad -\intr \Lambda^{2\alpha}(hu)\cdot u\,\dx +\intr |u|^2\Lambda^{2\alpha}h\,\dx .
\end{align*}
The three nonlinear terms are estimated exactly as in Section \ref{ssec:low-zero}. Using H\"older's inequality, the fractional product estimate, and Sobolev embedding, we obtain
\begin{align*}
&\lt|\intr (u\cdot\nabla u)\cdot u\,\dx\rt| +\lt|\intr \Lambda^{2\alpha}(hu)\cdot u\,\dx\rt| +\lt|\intr |u|^2\Lambda^{2\alpha}h\,\dx\rt| \\
&\quad \le C\lt( \frac{\|h\|_{H^s}}{\e}+\|u\|_{H^s} \rt) \lt( \frac{\|\nabla h\|_{H^{s-\alpha}}^2}{\e^2} +\|\Lambda^\alpha u\|_{H^s}^2 \rt).
\end{align*}
Hence, for $0<\e\le1$,
\bq\label{u-zero-est-high}
\begin{aligned}
\frac12\frac\rd\dt\|u\|_{L^2}^2+\|\Lambda^\alpha u\|_{L^2}^2 &\le C\lt( \frac{\|h\|_{H^s}}{\e}+\|u\|_{H^s} \rt) \lt( \frac{\|\nabla h\|_{H^{s-\alpha}}^2}{\e^2} +\|\Lambda^\alpha u\|_{H^s}^2 \rt)\\
&\quad -\frac{\gamma}{\e^2} \intr (1+h)^{\gamma-2}\nabla h\cdot u\,\dx .
\end{aligned}
\eq
Adding \eqref{h-zero-est-high} and \eqref{u-zero-est-high}, the pressure coupling terms cancel and hence
\bq\label{zero-energy-est-high}
\begin{aligned}
&\frac12\frac\rd\dt\lt( \frac{\gamma}{\e^2}\intr (1+h)^{\gamma-3}|h|^2\,\dx +\|u\|_{L^2}^2 \rt) +\|\Lambda^\alpha u\|_{L^2}^2\\
&\quad \le C\lt( \frac{\|h\|_{H^s}}{\e}+\|u\|_{H^s} \rt) \lt( \frac{\|\nabla h\|_{H^{s-\alpha}}^2}{\e^2} +\|\Lambda^\alpha u\|_{H^s}^2 \rt).
\end{aligned}
\eq
It remains to recover the zeroth-order dissipation for $h$. As in Section \ref{ssec:low-zero}, we compute
\begin{align*}
-\frac\rd\dt\intr h\nabla\cdot u\,\dx +\frac{\gamma}{\e^2}\|\nabla h\|_{L^2}^2 &\le C\|\nabla h\|_{L^2}\|u\|_{L^\infty}\|\nabla u\|_{L^2} +C\|h\|_{L^\infty}\|\nabla u\|_{L^2}^2+\|\nabla\cdot u\|_{L^2}^2 \\
&\quad  +\frac{C\|h\|_{L^\infty}}{\e^2}\|\nabla h\|_{L^2}^2 +II,
\end{align*}
where
\[
II:=-\intr \nabla h\cdot \lt(\Lambda^{2\alpha}((1+h)u)-u\Lambda^{2\alpha}h\rt)\dx .
\]
This is the only point where the zeroth-order argument differs from the low-order regime. We first have
\[
II = -\frac12\frac{d}{dt}\|\Lambda^\alpha h\|_{L^2}^2  + \intr \Lambda^\alpha h \Lambda^\alpha (\nabla h \cdot u)\,\dx.
\]
Here, the fractional product estimate \eqref{eq:frac_prod2} implies
\[
\|\Lambda^\alpha (\nabla h \cdot u) - u \cdot \nabla \Lambda^\alpha h - \nabla h \cdot \Lambda^\alpha u\|_{\frac{1}{1-\frac\alpha d}} \le C\|\nabla h\|_{L^2}\|\Lambda^\alpha u\|_{\frac{1}{\frac12 -\frac\alpha d}}.
\]
Since $2\alpha >1$, $s>1$, and $s+\alpha >2\alpha$, we obtain
\[\begin{aligned}
\intr \Lambda^\alpha  h (u \cdot \nabla \Lambda^\alpha h)\,\dx &= -\frac12\intr (\nabla\cdot u) |\Lambda^\alpha h|^2\,\dx \\
&\le C\|\nabla u\|_{\frac{d}{2-2\alpha}}\|\nabla h\|_{L^2}^2\\
&\le C\|h\|_{H^s}\|\Lambda^\alpha u\|_{H^s}\|\nabla h\|_{H^{s-\alpha}}.
\end{aligned}\]
and
\[\begin{aligned}
\intr \Lambda^\alpha h (\nabla h \cdot \Lambda^\alpha u )\,\dx &\le C\|\Lambda^\alpha h\|_{\frac{d}{\alpha}}\|\nabla h\|_{L^2}\|\Lambda^\alpha u\|_{\frac{1}{\frac12-\frac\alpha d}} \le C\|h\|_{H^s}\|\Lambda^\alpha u\|_{H^s}\|\nabla h\|_{H^{s-\alpha}}
\end{aligned}
\]
Hence, 
\[\begin{aligned}
II &= -\frac12\frac{d}{dt}\|\Lambda^\alpha h\|_{L^2}^2 + \intr \Lambda^\alpha h \lt[  \Lambda^\alpha (\nabla h \cdot u) - u \cdot \nabla \Lambda^\alpha h - \nabla h \cdot \Lambda^\alpha u \rt]\,\dx\\
&\quad -\frac12\intr (\nabla \cdot u) |\Lambda^\alpha h|^2\,\dx + \intr \Lambda^\alpha h \lt(\nabla h \cdot \Lambda^\alpha u \rt)\,\dx\\
&\le  -\frac12\frac{d}{dt}\|\Lambda^\alpha h\|_{L^2}^2 + C\|h\|_{H^s}\|\Lambda^\alpha u\|_{H^s}\|\nabla h\|_{H^{s-\alpha}}.
\end{aligned}\]
Consequently, by Young's inequality and the smallness assumption \eqref{small-assumption-high},  we obtain
\bq\label{cross-zero-est-high}
\begin{aligned}
\frac\rd\dt\lt(\frac12\|\Lambda^\alpha h\|_{L^2}^2- \intr h\nabla\cdot u\,\dx\rt) +\frac{\gamma}{2\e^2}\|\nabla h\|_{L^2}^2 &\le C\lt( \frac{\|h\|_{H^s}}{\e}+\|u\|_{H^s} \rt) \lt( \frac{\|\nabla h\|_{H^{s-\alpha}}^2}{\e^2} +\|\Lambda^\alpha u\|_{H^s}^2 \rt)\\
&\quad +C_1\|\nabla u\|_{L^2}^2.
\end{aligned}
\eq
This completes the zeroth-order estimates in the higher-order fractional regime.

%
%
%
%
%
%

\subsection{High-order energy estimates}

We next derive the high-order energy estimates in the regime $\frac12<\alpha<1$. Since the algebraic structure is the same as in Section \ref{ssec:low-h-ene}, we only give the estimates with the density dissipation level $\|\nabla h\|_{H^{s-\alpha}}$.

We first estimate the high-order density energy. From the continuity equation,
the same computation as in Section \ref{ssec:low-h-ene} gives
\begin{align*}
\frac12\frac\rd\dt \intr (1+h)^{\gamma-3}|\Lambda^s h|^2\,\dx &= \frac{4-\gamma}{2}\intr (1+h)^{\gamma-3} (\nabla\cdot u)|\Lambda^s h|^2 \,\dx\\
&\quad -\intr (1+h)^{\gamma-3}\Lambda^s h \lt[ \Lambda^s\nabla\cdot((1+h)u) -u\cdot\Lambda^s\nabla h -(1+h)\Lambda^s\nabla\cdot u \rt] \dx\\
&\quad +(\gamma-2)\intr (1+h)^{\gamma-3} (\nabla h\cdot\Lambda^s u)\Lambda^s h\,\dx +\intr (1+h)^{\gamma-2} \nabla\Lambda^s h\cdot\Lambda^s u\,\dx .
\end{align*}
 Using $s+\alpha > \frac d2 +\alpha$,
\[
\lt|\intr (1+h)^{\gamma-3} (\nabla\cdot u)|\Lambda^s h|^2\,\dx\rt| \le C\|\nabla u\|_{L^{\frac{d}{1-\alpha}}} \|\Lambda^s h\|_{L^2}\|\Lambda^s h\|_{L^{\frac{1}{\frac12-\frac{1-\alpha}{d}}}} \le C\|h\|_{H^s} \|\Lambda^\alpha u\|_{H^s} \|\nabla h\|_{H^{s-\alpha}} .
\]
Similarly, by H\"older's inequality and Sobolev embedding,
\[
\lt|\intr (1+h)^{\gamma-3} (\nabla h\cdot\Lambda^s u)\Lambda^s h\,\dx\rt| \le C\|\nabla h\|_{L^{\frac d\alpha}} \|\Lambda^s u\|_{L^{\frac{1}{\frac12-\frac\alpha d}}} \|\Lambda^s h\|_{L^2} \le C\|h\|_{H^s} \|\Lambda^\alpha u\|_{H^s} \|\nabla h\|_{H^{s-\alpha}} .
\]
For the second term, we recall the formula \eqref{eq:four_bdd_3term}:
\[
\lt||\xi|^s\xi-|\eta|^s\eta-|\xi-\eta|^s(\xi-\eta)\rt| \le C\lt(|\eta|^s|\xi-\eta|+|\eta||\xi-\eta|^s\rt).
\]
By Young's convolution inequality with $\frac{1}{q_i} + \frac{1}{q_i'} = \frac32$ for $i=1,2$ and Lemma \ref{lem:high-four-inter},
\[\begin{aligned}
\|\Lambda^s\nabla\cdot((1+h)u) -u\cdot\Lambda^s\nabla h -(1+h)\Lambda^s\nabla\cdot u\|_{L^2} &\le C\lt(\||\xi|^s \hat u\|_{L_\xi^{q_1}}\| |\xi|\hat h\|_{L_\xi^{q_1'}} +  \||\xi|^s \hat h\|_{L_\xi^{q_2}}\| |\xi|\hat u\|_{L_\xi^{q_2'}}\rt)\\
&\le C\|\Lambda^\alpha u\|_{H^s}\|\nabla h\|_{H^{s-\alpha}}.
\end{aligned}\]
Thus, we have
\[
\frac12\frac\rd\dt \intr (1+h)^{\gamma-3}|\Lambda^s h|^2\,\dx \le C\|h\|_{H^s} \|\Lambda^\alpha u\|_{H^s} \|\nabla h\|_{H^{s-\alpha}} +\intr (1+h)^{\gamma-2} \nabla\Lambda^s h\cdot\Lambda^s u\,\dx .
\]
Multiplying by $\frac{\gamma}{\e^2}$ and using Young's inequality, we obtain
\bq\label{h-high-est-high}
\begin{aligned}
\frac12\frac\rd\dt\lt( \frac{\gamma}{\e^2} \intr (1+h)^{\gamma-3}|\Lambda^s h|^2\,\dx \rt)
&\le C\lt( \frac{\|h\|_{H^s}}{\e}+\|u\|_{H^s} \rt) \lt( \frac{\|\nabla h\|_{H^{s-\alpha}}^2}{\e^2} +\|\Lambda^\alpha u\|_{H^s}^2 \rt)\\
&\quad +\frac{\gamma}{\e^2} \intr (1+h)^{\gamma-2} \nabla\Lambda^s h\cdot\Lambda^s u\,\dx .
\end{aligned}
\eq
We now estimate the high-order velocity energy. Taking the $L^2$ inner product of the velocity equation with $\Lambda^{2s}u$, we get
\begin{align*}
\frac12\frac\rd\dt\|\Lambda^s u\|_{L^2}^2 +\|\Lambda^{s+\alpha}u\|_{L^2}^2 &= \frac12\intr (\nabla\cdot u)|\Lambda^s u|^2\,\dx -\intr \Lambda^s u\cdot \lt[\Lambda^s(u\cdot\nabla u)-u\cdot\nabla\Lambda^s u\rt] \dx\\
&\quad -\frac{\gamma}{\e^2} \intr (1+h)^{\gamma-2}\nabla\Lambda^s h\cdot\Lambda^s u\,\dx\\
&\quad -\frac{\gamma}{\e^2}\intr \Lambda^s u\cdot \lt[ \Lambda^s((1+h)^{\gamma-2}\nabla h) -(1+h)^{\gamma-2}\nabla\Lambda^s h \rt] \dx\\
&\quad -\intr \Lambda^{s+\alpha}u\cdot \lt[\Lambda^{s+\alpha}(hu)-u\Lambda^{s+\alpha}h\rt] \dx\\
&\quad +\intr \Lambda^s u\cdot \lt[ \Lambda^s(u\Lambda^{2\alpha}h) -\Lambda^\alpha(u\Lambda^{s+\alpha}h) \rt] \dx\\
&=: \sum_{i=1}^6 \sfK_i .
\end{align*}
The first two terms are estimated as in Section \ref{ssec:low-h-ene}. Using H\"older's inequality, Young's convolution inequality, and Lemma \ref{lem:high-four-inter}, we have
\[
|\sfK_1|+|\sfK_2| \le C\|u\|_{H^s}\|\Lambda^\alpha u\|_{H^s}^2 .
\]
The third term is the pressure coupling term.  

For $\sfK_4$, set $g=(1+h)^{\gamma-2}-1$. Then,
\[
\Lambda^s((1+h)^{\gamma-2}\nabla h) -(1+h)^{\gamma-2}\nabla\Lambda^s h = \Lambda^s(g\nabla h)-g\nabla\Lambda^s h .
\]
We estimate the resulting trilinear form directly. Using the pointwise bound \eqref{eq:four_bdd_l2}, Young's convolution inequality, Lemma \ref{lem:high-four-inter}, and Remark \ref{rem:high-four-cons}, we obtain
\[
|\sfK_4| \le \frac{C}{\e^2} \|\Lambda^s g\|_{L^2} \||\xi|^s\hat u\|_{L_\xi^{q_1}} \||\xi|\hat h\|_{L_\xi^{q_1'}}  +\frac{C}{\e^2} \|\Lambda^s h\|_{L^2} \||\xi|^s\hat u\|_{L_\xi^{q_1}} \||\xi|\hat g\|_{L_\xi^{q_1'}}.
\]
Since $\|h\|_{L^\infty}$ is sufficiently small, the Moser estimate gives
\[
\|\Lambda^s g\|_{L^2}\le C\|h\|_{H^s}.
\]
Together with Lemma \ref{lem:high-four-inter} and Remark \ref{rem:high-four-cons}, this yields
\[
|\sfK_4| \le \frac{C}{\e^2} \|h\|_{H^s} \|\Lambda^\alpha u\|_{H^s} \|\nabla h\|_{H^{s-\alpha}} \le C\lt( \frac{\|h\|_{H^s}}{\e}+\|u\|_{H^s} \rt) \lt( \frac{\|\nabla h\|_{H^{s-\alpha}}^2}{\e^2} +\|\Lambda^\alpha u\|_{H^s}^2 \rt).
\]
For $\sfK_5$, the pointwise bound \eqref{eq:four_bdd_sa}, Young's convolution inequality, and Lemma \ref{lem:high-four-inter} give
\[
|\sfK_5| \le C\|\Lambda^{s+\alpha}u\|_{L^2} \lt( \|\hat h\|_{L_\xi^1}\|\Lambda^{s+\alpha}u\|_{L^2} +\||\xi|^{s+\alpha-1}\hat h\|_{L_\xi^{q_2}}\||\xi| \hat u\|_{L_\xi^{q_2'}} \rt) \le C\|h\|_{H^s}\|\Lambda^\alpha u\|_{H^s}^2.
\]
It remains to estimate $\sfK_6$. We use the pointwise bound
\begin{align*}
&\lt||\xi|^{s-\alpha}|\eta|^{2\alpha} -|\eta|^{s+\alpha}\rt| \\
&\quad \le C|\eta|^{2\alpha} |\xi-\eta| \lt( |\xi|^{s-\alpha-1} + |\eta|^{s-\alpha-1}\rt)\\
&\quad \le C|\xi|^{-\alpha} (|\xi-\eta|^s |\eta|^{2\alpha} + |\xi-\eta| |\eta|^{s+2\alpha-1}) + C|\xi-\eta||\eta|^{s+\alpha-1}\\
&\quad \le C|\xi-\eta|^s |\eta|^\alpha + C|\xi|^{-\alpha} |\xi-\eta|^{s+\alpha} |\eta|^\alpha + C|\xi|^{-\alpha} |\xi-\eta|^{1+\alpha} |\eta|^{s+\alpha-1} + C|\xi-\eta| |\eta|^{s+\alpha-1}.
\end{align*}
Hence, by Young's convolution inequality, Lemma \ref{lem:high-four-inter}, and Remark \ref{rem:high-four-cons},
\begin{align*}
|\sfK_6|  &\le C \intrr (|\xi|^{s+\alpha} |\hat u(\xi)|)  (|\xi-\eta|^s |\hat u(\xi-\eta)|)  (|\eta|^\alpha |\hat h(\eta)|)\,\rd\eta \dxi \\
&\quad + C\intrr (|\xi|^s |\hat u(\xi)|) (|\xi-\eta|^{s+\alpha} |\hat u(\xi-\eta)|) (|\eta|^\alpha |\hat h(\eta)|)\,\rd\eta \dxi\\
&\quad + C\intrr (|\xi|^s |\hat u(\xi)|) (|\xi-\eta|^{1+\alpha} |\hat u(\xi-\eta)|) (|\eta|^{s+\alpha-1} |\hat h(\eta)|)\,\rd\eta \dxi\\
&\quad +  C \intrr (|\xi|^{s+\alpha} |\hat u(\xi)|)  (|\xi-\eta| |\hat u(\xi-\eta)|)  (|\eta|^{s+\alpha-1} |\hat h(\eta)|)\,\rd\eta \dxi \\
&\le C\|\Lambda^{s+\alpha}u\|_{L^2} \||\xi|^s\hat u\|_{L_\xi^{q_1}} \||\xi|^\alpha\hat h\|_{L_\xi^{q_1'}} +C\||\xi|^s\hat u\|_{L_\xi^{q_1}} \||\xi|^{1+\alpha}\hat u\|_{L_\xi^{q_3}} \||\xi|^{s+\alpha-1}\hat h\|_{L_\xi^{q_2}}\\
&\quad +C\|\Lambda^{s+\alpha}u\|_{L^2} \||\xi|\hat u\|_{L_\xi^{q_2'}} \||\xi|^{s+\alpha-1}\hat h\|_{L_\xi^{q_2}}\\
&\le C\|h\|_{H^s}\|\Lambda^\alpha u\|_{H^s}^2.
\end{align*}
Collecting the estimates for $\sfK_i$, we obtain
\bq\label{u-high-est-high}
\begin{aligned}
\frac12\frac\rd\dt\|\Lambda^s u\|_{L^2}^2 +\|\Lambda^{s+\alpha}u\|_{L^2}^2
&\le C\lt( \frac{\|h\|_{H^s}}{\e}+\|u\|_{H^s} \rt) \lt( \frac{\|\nabla h\|_{H^{s-\alpha}}^2}{\e^2} +\|\Lambda^\alpha u\|_{H^s}^2 \rt)\\
&\quad -\frac{\gamma}{\e^2} \intr (1+h)^{\gamma-2} \nabla\Lambda^s h\cdot\Lambda^s u\,\dx .
\end{aligned}
\eq
Combining \eqref{h-high-est-high} and \eqref{u-high-est-high}, the pressure coupling terms cancel and we obtain
\bq\label{high-e-h-alp-est}
\begin{aligned}
&\frac12\frac\rd\dt\lt( \frac{\gamma}{\e^2} \intr (1+h)^{\gamma-3}|\Lambda^s h|^2\,\dx +\|\Lambda^s u\|_{L^2}^2 \rt) +\|\Lambda^{s+\alpha}u\|_{L^2}^2 \\
&\quad \le C\lt( \frac{\|h\|_{H^s}}{\e}+\|u\|_{H^s} \rt) \lt( \frac{\|\nabla h\|_{H^{s-\alpha}}^2}{\e^2} +\|\Lambda^\alpha u\|_{H^s}^2 \rt).
\end{aligned}
\eq
 
%
%
%
%
%
%

\subsection{High-order density dissipation estimate}

We now recover the high-order density dissipation in the regime $\frac12<\alpha<1$. In parallel with Section \ref{ssec:low-den-diss}, we estimate the mixed quantity:
\begin{align*}
&-\frac\rd\dt \intr \Lambda^{s-\alpha}h\,\Lambda^{s-\alpha}\nabla\cdot u\,\dx \\
&\quad = \intr \nabla\cdot\lt(\Lambda^{s-\alpha}(hu)-u\Lambda^{s-\alpha}h\rt) \Lambda^{s-\alpha}\nabla\cdot u\,\dx   +\|\Lambda^{s-\alpha}\nabla\cdot u\|_{L^2}^2 \\
&\quad \quad +\intr \Lambda^{s-\alpha}h \lt[ \Lambda^{s-\alpha}\nabla\cdot(u\cdot\nabla u) -u\cdot\nabla\Lambda^{s-\alpha}\nabla\cdot u \rt] \dx  -\frac{\gamma}{\e^2}\|\Lambda^{s+1-\alpha} h\|_{L^2}^2 \\
&\quad\quad -\frac{\gamma}{\e^2} \intr \nabla\Lambda^{s-\alpha}h\cdot \Lambda^{s-\alpha} \lt[ \lt((1+h)^{\gamma-2}-1\rt)\nabla h \rt] \dx   +\intr \Lambda^{s-\alpha}h\nabla\cdot \Lambda^{s+\alpha}((1+h)u)\,\dx \\
&\quad\quad -\intr \Lambda^{s-\alpha}h \nabla\cdot \Lambda^{s-\alpha}(u\Lambda^{2\alpha}h) \, \dx \\ 
&\quad =: \sum_{i=1}^7 \sfL_i .
\end{align*}

For $\sfL_1$, by the pointwise bound
\bq\label{eq:four_bdd_trans}
||\xi|^{s-\alpha}\xi_j - |\eta|^{s-\alpha}\eta_j| \le C|\xi-\eta| |\eta|^{s-\alpha} + C|\xi-\eta|^{s-\alpha+1}, \quad j=1,\dots, d,
\eq
and $\alpha<s-\alpha+1 < s+\alpha$, we deduce from Lemma \ref{lem:high-four-inter} that
\[\begin{aligned}
\sfL_1 & = \intr  \lt(\Lambda^{s-\alpha}\nabla\cdot(hu)-u \cdot \nabla\Lambda^{s-\alpha}h\rt) \Lambda^{s-\alpha}\nabla\cdot u\,\dx  -\intr (\nabla\cdot u)(\Lambda^{s-\alpha}h)  (\Lambda^{s-\alpha}\nabla\cdot u)\,\dx \\
&\le C\|\Lambda^{s-\alpha+1} u\|_{L^2}\lt(\||\xi|^{s-\alpha} \hat h\|_{L_\xi^{q_1}}\||\xi|\hat u\|_{L_\xi^{q_1'}} + \|\hat h\|_{L_\xi^1} \|\Lambda^{s-\alpha+1} u\|_{L^2}\rt)  +C\|\nabla \cdot u\|_{L^{\frac{d}{\alpha}}}\|\Lambda^s h\|_{L^2}\|\Lambda^{s-\alpha+1}u\|_{L^2}\\
&\le C\|h\|_{H^s}\|\Lambda^\alpha u\|_{H^s}^2.
\end{aligned}\]
The second term is estimated by
\[
\sfL_2 = \|\Lambda^{s-\alpha}\nabla\cdot u\|_{L^2}^2 \le \|\Lambda^{s-\alpha+1} u\|_{L^2}^2.
\]
For $\sfL_3$, the pointwise bound \eqref{eq:four_bdd_trans} and Lemma \ref{lem:high-four-inter} give
\begin{align*}
|\sfL_3| &\le C\||\xi|^{s-\alpha} \hat h\|_{L_\xi^{q_1}} \| |\xi| \hat u\|_{L_\xi^{q_1'}} \|\Lambda^{s-\alpha+1} u\|_{L^2}  \le C\|h\|_{H^s} \|\Lambda^\alpha u\|_{H^s}^2.
\end{align*}
The fourth term gives the desired density dissipation.

For the pressure commutator term, set $g=(1+h)^{\gamma-2}-1$. Using the bound
\[
|\xi|^{s-\alpha} \le C\lt(|\xi-\eta|^{s-\alpha}+|\eta|^{s-\alpha}\rt),
\]
Young's convolution inequality and Lemma \ref{lem:high-four-inter} give
\[\begin{aligned}
|\sfL_5| &\le \frac{C}{\e^2}\|\Lambda^{s-\alpha+1} h\|_{L^2} \|\Lambda^{s-\alpha}(g\nabla h)\|_{L^2}  \\
&\le \frac{C}{\e^2}\|\Lambda^{s-\alpha+1} h\|_{L^2} \lt( \||\xi|^{s-\alpha}\hat g\|_{L_\xi^{q_1}} \||\xi|\hat h\|_{L_\xi^{q_1'}} +\|\hat g\|_{L_\xi^1}\|\Lambda^{s-\alpha+1} h\|_{L^2} \rt)\\
&\le \frac{C}{\e^2} \|h\|_{H^s}\|\nabla h\|_{H^{s-\alpha}}^2.
\end{aligned}\]
For $\sfL_6$, we find
\[
\sfL_6 = -\frac12\frac{d}{dt}\|\Lambda^s h\|_{L^2}^2.
\]
It remains to estimate $\sfL_7$. We decompose
\begin{align*}
\sfL_7 &= \intr \nabla\Lambda^{s-\alpha} h \cdot\lt[ \Lambda^{s-\alpha} (u\Lambda^{2\alpha} h) - u \Lambda^{s+\alpha} h\rt]\dx\\
&\quad +\intr \Lambda^s h \lt[ \Lambda^\alpha( u \cdot \nabla\Lambda^{s-\alpha} h) - u \cdot \nabla\Lambda^s h - \Lambda^\alpha u \cdot \nabla \Lambda^{s-\alpha} h \rt]\dx\\
&\quad + \intr \Lambda^s h (u \cdot \nabla \Lambda^s h)\,\dx + \intr \Lambda^s h (\Lambda^\alpha u) \cdot \nabla\Lambda^{s-\alpha} h\,\dx\\
&=: \sfL_{7,1}+\sfL_{7,2}+\sfL_{7,3} + \sfL_{7,4}.
\end{align*}
For $\sfL_{7,1}$, we use the pointwise bound
\begin{align*}
 \lt| |\xi|^{s-\alpha}|\eta|^{2\alpha} -|\eta|^{s+\alpha} \rt| 
&\le C|\eta|^{2\alpha}|\xi-\eta| (|\xi|^{s-\alpha-1} + |\eta|^{s-\alpha-1})\\
&\le C|\xi|^{-\alpha} |\xi-\eta|^s |\eta|^{2\alpha} + C|\xi|^{-\alpha} |\xi-\eta||\eta|^{s+2\alpha-1} + C|\xi-\eta| |\eta|^{s+\alpha -1}\\
& \le C|\xi-\eta|^s|\eta|^\alpha + C|\xi|^{-\alpha}|\xi-\eta|^{s+\alpha}|\eta|^\alpha + C|\xi-\eta||\eta|^{s+\alpha-1} \\
&\quad  + C|\xi|^{-\alpha}|\xi-\eta|^{1+\alpha}  |\eta|^{s+\alpha-1}.
\end{align*}
Then, Young's convolution inequality and Lemma \ref{lem:high-four-inter} yield
\begin{align*}
|\sfL_{7,1}|
&\le C \intrr (|\xi|^{s-\alpha +1} |\hat h(\xi)|) (|\xi-\eta|^s |\hat u(\xi-\eta)|) (|\eta|^\alpha |\hat h(\eta)|)\,\rd\eta \dxi\\
&\quad + C \intrr (|\xi|^{s-2\alpha +1} |\hat h(\xi)|) (|\xi-\eta|^{s+\alpha} |\hat u(\xi-\eta)|) (|\eta|^\alpha |\hat h(\eta)|)\,\rd\eta \dxi\\
&\quad +C \intrr (|\xi|^{s-\alpha +1} |\hat h(\xi)|) (|\xi-\eta| |\hat u(\xi-\eta)|) (|\eta|^{s+\alpha-1} |\hat h(\eta)|)\,\rd\eta \dxi\\
&\quad + C \intrr (|\xi|^{s-2\alpha+1} |\hat h(\xi)|) (|\xi-\eta|^{1+\alpha} |\hat u(\xi-\eta)|) (|\eta|^{s+\alpha-1} |\hat h(\eta)|)\,\rd\eta \dxi\\
&\le C\|\Lambda^{s-\alpha+1}h\|_{L^2} \||\xi|^s\hat u\|_{L_\xi^{q_1}} \||\xi|^\alpha\hat h\|_{L_\xi^{q_1'}} + C\||\xi|^{s-2\alpha+1}\hat h\|_{L_\xi^{q_1}} \|\Lambda^{s+\alpha}u\|_{L^2} \||\xi|^\alpha\hat h\|_{L_\xi^{q_1'}} \\
&\quad +  C\|\Lambda^{s-\alpha+1}h\|_{L^2}    \||\xi|\hat u\|_{L_\xi^{q_2'}}\||\xi|^{s+\alpha-1} \hat h\|_{L_\xi^{q_2}} +C\||\xi|^{s-2\alpha+1} \hat h\|_{L_\xi^{q_1}} \| |\xi|^{1+\alpha} \hat u\|_{L_\xi^{q_3}} \| |\xi|^{s+\alpha-1} \hat h\|_{L_\xi^{q_2}}\\
&\le C\|h\|_{H^s} \|\nabla h\|_{H^{s-\alpha}} \|\Lambda^\alpha u\|_{H^s}.
\end{align*}
For $\sfL_{7,2}$, by the fractional commutator estimates,
\[
\sfL_{7,2} \le C\|\Lambda^s h\|_{L^2}\|\Lambda^\alpha u\|_{L^\infty} \|\Lambda^{s-\alpha+1} h\|_{L^2} \le C\|h\|_{H^s}\|\nabla h\|_{H^{s-\alpha}} \|\Lambda^\alpha u\|_{H^s},
\]
where we used $s>\frac d2$.

For $\sfL_{7,3}$, Sobolev embedding gives
\[
\sfL_{7,3} = -\frac12 \intr (\nabla\cdot u) |\Lambda^s h|^2\,\dx \le C\|\nabla u\|_{L^{\frac{d}{1-\alpha}}} \|\Lambda^{s-\alpha+1} h\|_{L^2}\|\Lambda^s h\|_{L^2} \le C\|h\|_{H^s}\|\nabla h\|_{H^{s-\alpha}} \|\Lambda^\alpha u\|_{H^s},
\]
where we used $H^{\frac d2 -1+\alpha}(\R^d) \hookrightarrow L^{\frac{d}{1-\alpha}}(\R^d)$ and $s+\alpha >\frac d2 +\alpha$.

For $\sfL_{7,4}$, by using $s>\frac d2$,
\[
\sfL_{7,4} \le C\|\Lambda^s h\|_{L^2}\|\Lambda^\alpha u\|_{L^\infty} \|\Lambda^{s-\alpha+1} h\|_{L^2} \le C\|h\|_{H^s}\|\nabla h\|_{H^{s-\alpha}} \|\Lambda^\alpha u\|_{H^s}.
\]
Thus, we have
\[
|\sfL_7| \le C\|h\|_{H^s} \|\nabla h\|_{H^{s-\alpha}} \|\Lambda^\alpha u\|_{H^s} \le C\lt( \frac{\|h\|_{H^s}}{\e} +\|u\|_{H^s} \rt) \lt( \frac{\|\nabla h\|_{H^{s-\alpha}}^2}{\e^2} +\|\Lambda^\alpha u\|_{H^s}^2 \rt).
\]
Collecting the estimates for $\sfL_i$, we obtain
\bq\label{density-diss-high-alpha}
\begin{aligned}
&\frac\rd\dt \lt( \frac12\|\Lambda^s h\|_{L^2}^2 -\intr \Lambda^{s-\alpha}h\,\Lambda^{s-\alpha}\nabla\cdot u\,\dx\rt) + \frac{\gamma}{2\e^2}\|\Lambda^{s+1-\alpha} h\|_{L^2}^2 \\
&\quad \le C\lt( \frac{\|h\|_{H^s}}{\e} +\|u\|_{H^s} \rt) \lt( \frac{  \|\nabla h\|_{H^{s-\alpha}}^2  }{\e^2} +\|\Lambda^\alpha u\|_{H^s}^2 \rt)   + C_1\lt( \|\Lambda^s u\|_{L^2}^2 +\|\Lambda^{s+\alpha}u\|_{L^2}^2 \rt).
\end{aligned}
\eq

%
%
%
%
%
%

\subsection{Completion of the higher-order estimate}

We now conclude the proof of Proposition \ref{prop:uni-apri-high}. Adding \eqref{zero-energy-est-high} and \eqref{high-e-h-alp-est}, and then adding a sufficiently small multiple of \eqref{cross-zero-est-high} and \eqref{density-diss-high-alpha}, we obtain, for $\theta>0$ sufficiently small,
\begin{align*}
&\frac\rd\dt\overline\calE_\theta(t) +\frac{\theta\gamma}{2\e^2} \lt( \|\nabla h\|_{L^2}^2+\|\Lambda^{s-\alpha+1} h\|_{L^2}^2 \rt) +\frac12\lt( \|\Lambda^\alpha u\|_{L^2}^2+\|\Lambda^{s+\alpha}u\|_{L^2}^2 \rt)\\
&\quad \le C\lt( \frac{\|h\|_{H^s}}{\e}+\|u\|_{H^s} \rt) \lt( \frac{\|\nabla h\|_{H^{s-\alpha}}^2}{\e^2} +\|\Lambda^\alpha u\|_{H^s}^2 \rt),
\end{align*}
 where $\overline\calE_\theta$ is the modified energy defined by
\begin{align*}
\overline\calE_\theta(t) &:=\frac{\gamma}{2\e^2} \intr (1+h)^{\gamma-3} \lt(|h|^2+|\Lambda^s h|^2\rt)\,\dx +\frac12\lt(\|u\|_{L^2}^2+\|\Lambda^s u\|_{L^2}^2\rt)\\
&\quad  +\theta \lt(\|\Lambda^\alpha h\|_{L^2}^2 + \|\Lambda^s h\|_{L^2}^2 \rt)-2\theta\intr h\nabla\cdot u\,\dx   - 2\theta \intr \Lambda^{s-\alpha}h\,\Lambda^{s-\alpha}\nabla\cdot u\,\dx.
\end{align*}
By self-adjointness of $\Lambda$, we get
\[
\intr \Lambda^{s-\alpha}h\,\Lambda^{s-\alpha}\nabla\cdot u\,\dx = \intr \Lambda^{s-2\alpha+1}h\,\Lambda^{s-1}\nabla\cdot u\,\dx.
\]
Since $2\alpha>1$, we have $s-2\alpha+1\le s$, and thus
\[
\lt| \intr \Lambda^{s-\alpha}h\,\Lambda^{s-\alpha}\nabla\cdot u\,dx \rt| \le C\|h\|_{H^s}\|u\|_{H^s}.
\]
Exactly as in Section \ref{ssec:low-comp}, the smallness of
$\|h\|_{L^\infty}$ and Young's inequality imply that, if $\theta>0$ is chosen
sufficiently small, then
\[
\overline\calE_\theta(t) \simeq \frac{\|h\|_{H^s}^2}{\e^2}+\|u\|_{H^s}^2.
\]
Moreover,
\[
\|\nabla h\|_{L^2}^2+\|\Lambda^{s-\alpha+1} h\|_{L^2}^2 \simeq \|\nabla h\|_{H^{s-\alpha}}^2
\]
up to lower-order terms controlled by the modified energy. Hence, for some constants $c_0,C_0>0$ independent of $T$ and $\e$,
\[
c_0\lt( \frac{\|h\|_{H^s}^2}{\e^2}+\|u\|_{H^s}^2 \rt) \le \overline\calE_\theta(t) \le C_0\lt( \frac{\|h\|_{H^s}^2}{\e^2}+\|u\|_{H^s}^2 \rt),
\]
and
\[
\frac\rd\dt\overline\calE_\theta(t) +c_0\lt( \frac{\|\nabla h\|_{H^{s-\alpha}}^2}{\e^2} +\|\Lambda^\alpha u\|_{H^s}^2 \rt) \le C\lt( \frac{\|h\|_{H^s}}{\e}+\|u\|_{H^s} \rt) \lt( \frac{\|\nabla h\|_{H^{s-\alpha}}^2}{\e^2} +\|\Lambda^\alpha u\|_{H^s}^2 \rt).
\]
By the smallness assumption \eqref{small-assumption-high}, if $\eta>0$ is sufficiently small, the right-hand side is absorbed by the dissipation term. Thus
\[
\frac\rd\dt\overline\calE_\theta(t) +\frac{c_0}{2}\lt( \frac{\|\nabla h\|_{H^{s-\alpha}}^2}{\e^2} +\|\Lambda^\alpha u\|_{H^s}^2 \rt) \le0.
\]
Integrating over $[0,t]$, we obtain
\[
\frac{\|h(t)\|_{H^s}^2}{\e^2} +\|u(t)\|_{H^s}^2 +\int_0^t\lt( \frac{\|\nabla h(\tau)\|_{H^{s-\alpha}}^2}{\e^2} +\|\Lambda^\alpha u(\tau)\|_{H^s}^2 \rt)\rd\tau \le C\lt( \frac{\|h_0\|_{H^s}^2}{\e^2} +\|u_0\|_{H^s}^2 \rt),
\]
where $C>0$ is independent of $T$ and $\e$. This proves Proposition
\ref{prop:uni-apri-high}.

%
%
%
%
%
%

\section{Global well-posedness}\label{sec:global}

In this section, we prove the global well-posedness result. We first establish local well-posedness and a continuation criterion for \eqref{h-u-system}. We then combine the local theory with the uniform a priori estimates obtained in Sections \ref{sec:low-alpha} and \ref{sec:high-alpha} to extend the solution globally in time.

%
%
%
%
%
%

\subsection{Local well-posedness and continuation}

We first establish the local theory needed to initiate the global bootstrap. The low-order regime $0<\alpha\le\frac12$ requires a separate construction since the index in \eqref{eq:s-ass} may lie below the classical hyperbolic threshold, while the higher-order regime follows from the critical Besov theory of \cite{BMTX24}.

\begin{theorem} \label{thm:loc}
Let $d\ge2$, $0<\alpha<1$, $\gamma\ge1$, $0<\e\le1$, and let $s$ satisfy \eqref{eq:s-ass}. There exists $\eta_\star>0$, independent
of $\e$, such that, if
\[
h_0,u_0\in H^s(\R^d), \quad \frac{\|h_0\|_{H^s}}{\e}+\|u_0\|_{H^s} \le\eta_\star,
\]
then there exists $T_\e>0$ such that \eqref{h-u-system} admits a unique strong solution satisfying
\begin{align*}
&(h,u)\in C\lt([0,T_\e];H^s(\R^d)\times H^s(\R^d)\rt),\\
&\nabla h\in L^2\lt(0,T_\e;H^{\sigma_\alpha}(\R^d)\rt),
\quad
\Lambda^\alpha u\in L^2\lt(0,T_\e;H^s(\R^d)\rt).
\end{align*}
Let $[0,T_*)$ be its maximal interval in this class. If $T_*<\infty$, then at least one of the bounds
\[
\frac12\le1+h(t,x)\le\frac32, \quad 0\le t<T_*, \quad x\in\R^d,
\quad
\text{or}
\quad 
\sup_{0\le t<T_*} \lt( \frac{\|h(t)\|_{H^s}}{\e}+\|u(t)\|_{H^s} \rt) <\eta_\star
\]
must fail as $t\uparrow T_*$.
\end{theorem}

The time $T_\e$ may depend on $\e$. Only the smallness threshold and
the a priori constants used below are uniform in $\e$.

\subsubsection{Stability in the low-order regime}

We begin with the estimate that yields uniqueness and identifies the limit of the viscous approximations.

For $\rho>0$, define
\[
\Pi(\rho):=
\begin{cases}
\displaystyle \frac{\rho^\gamma-\gamma\rho+\gamma-1}{\gamma-1}, &\gamma>1,\\[2mm]
\rho\log\rho-\rho+1, &\gamma=1.
\end{cases}
\]
Then
\[
\Pi''(\rho)=\frac{p'(\rho)}{\rho}, \quad \Pi(\rho_1 | \rho_2):=\Pi(\rho_1)-\Pi(\rho_2)-\Pi'(\rho_2)(\rho_1-\rho_2).
\]

\begin{lemma}\label{lem:stab}
Let $d\ge2$, $0<\alpha\le\frac12$, and
\[
s>\frac d2+1-2\alpha.
\]
Let $(h_i,u_i)$, $i=1,2$, be two strong solutions to \eqref{h-u-system} on $[0,T]$, and set $\rho_i:=1+h_i$. Assume that
\[
\frac12\le\rho_i(t,x)\le\frac32, \quad \sup_{0\le t\le T} \lt( \frac{\|h_i(t)\|_{H^s}}{\e}+\|u_i(t)\|_{H^s} \rt) \le\eta, \quad i=1,2,
\]
for a sufficiently small constant $\eta>0$, and that
\[
\nabla h_i\in L^2(0,T;H^{s+\alpha-1}(\R^d)), \quad \Lambda^\alpha u_i\in L^2(0,T;H^s(\R^d)).
\]
Then there exists $C>0$, independent of $\e$, such that
\[
\frac{\|h_1(t)-h_2(t)\|_{L^2}^2}{\e^2} +\|u_1(t)-u_2(t)\|_{L^2}^2 \le C_T\lt( \frac{\|h_1(0)-h_2(0)\|_{L^2}^2}{\e^2} +\|u_1(0)-u_2(0)\|_{L^2}^2 \rt)
\]
for all $t\in[0,T]$, where
\[
C_T = C\exp\lt[ C\int_0^T\lt( 1+\frac{\|\nabla h_2(\tau)\|_{H^{s+\alpha-1}}^2}{\e^2} \rt)\rd\tau \rt].
\]
In particular, the solution is unique in this class.
\end{lemma}

\begin{proof}
Set
\[
H:=h_1-h_2, \quad U:=u_1-u_2, \quad A_i:=(1+h_i)^{\gamma-2}, \quad i=1,2,
\]
and define
\[
\calC(f,v):=-\Lambda^{2\alpha}(fv)+v\Lambda^{2\alpha}f.
\]
Subtracting the two equations gives
\bq\label{eq:d-h}
\pa_tH+\nabla\cdot U=-\nabla\cdot\lt(h_1U+Hu_2\rt)
\eq
and
\bq\label{eq:d-u}
\pa_tU+u_1\cdot\nabla U+U\cdot\nabla u_2 +\frac{\gamma}{\e^2}A_1\nabla H+\Lambda^{2\alpha}U  =-\frac{\gamma}{\e^2}(A_1-A_2)\nabla h_2 +\calC(h_1,U)+\calC(H,u_2).
\eq
We first consider the relative energy
\[
\calE_0(t):=\frac12\intr\rho_1|U|^2\,\dx+\frac1{\e^2}\intr\Pi(\rho_1 | \rho_2)\,\dx.
\]
Since $\rho_1$ and $\rho_2$ remain in a fixed compact subset of $(0,\infty)$, Taylor's theorem gives
\bq\label{eq:E0-equiv}
c_0\lt(\|U\|_{L^2}^2+\frac{\|H\|_{L^2}^2}{\e^2}\rt) \le\calE_0 \le C_0\lt(\|U\|_{L^2}^2+\frac{\|H\|_{L^2}^2}{\e^2}\rt).
\eq
A direct relative-energy computation yields
\[
\begin{aligned}
\frac\rd\dt\calE_0 &+\frac12\intor\phi(x-y)\rho_1(x)\rho_1(y)|U(x)-U(y)|^2\,\dx\dy\\
&=-\intr\rho_1U\otimes U:\nabla u_2\,\dx -\frac1{\e^2}\intr \lt[p(\rho_1)-p(\rho_2)-p'(\rho_2)H\rt] \nabla\cdot u_2\,\dx +\intr\rho_1U\cdot\calC(H,u_2)\,\dx.
\end{aligned}
\]
Here, Sobolev embedding and the condition $s>\frac d2+1-2\alpha$ give
\[
\|\nabla u_2\|_{L^{\frac d{2\alpha}}} \le C\|\nabla u_2\|_{\dot H^{\frac d2-2\alpha}} \le C\|u_2\|_{H^s}.
\]
Hence
\[
\lt|\intr\rho_1U\otimes U:\nabla u_2\,\dx\rt| \le C\|\nabla u_2\|_{L^{\frac d{2\alpha}}} \|U\|_{L^{\frac{2d}{d-2\alpha}}}^2 \le C\|u_2\|_{H^s}\|\Lambda^\alpha U\|_{L^2}^2 \le C\eta\|\Lambda^\alpha U\|_{L^2}^2.
\]
Taylor's theorem also gives

\[
|p(\rho_1) - p(\rho_2) - p'(\rho_2) H| \le C|H|^2.
\]
Consequently, H\"older's inequality and the homogeneous Sobolev embedding yield
\begin{align*}
&\frac1{\e^2} \lt| \intr [p(\rho_1) - p(\rho_2) - p'(\rho_2)H] (\nabla\cdot u_2)\,\dx \rt| \\
&\quad \le \frac C{\e^2} \|\nabla u_2\|_{L^{\frac d{2\alpha}}} \|H\|_{L^{\frac{2d}{d-2\alpha}}}^2 \le C\|u_2\|_{H^s} \frac{\|\Lambda^\alpha H\|_{L^2}^2}{\e^2} \le C\eta \frac{\|\Lambda^\alpha H\|_{L^2}^2}{\e^2}.
\end{align*}
For the remaining term, one proceeds as \eqref{eq:hu-low} to get
\[
\begin{aligned}
\|\Lambda^\alpha (h_1 U)\|_{L^2}  \le C\|h_1\|_{H^s}\|\Lambda^\alpha U\|_{L^2}.
\end{aligned}
\]
Similarly, we obtain
\[\begin{aligned}
\|\Lambda^\alpha (Hu_2)\|_{L^2} &\le C( \|\Lambda^\alpha H\|_{L^2}\|u_2\|_{L^\infty}+ \|H\|_{L^{\frac{1}{\frac12-\frac\alpha d}}} \|\Lambda^\alpha u_2\|_{L^{\frac d\alpha}}) \le C\|u_2\|_{H^s} \|\Lambda^\alpha H\|_{L^2}
\end{aligned}\]
and
\[\begin{aligned}
\|\Lambda^\alpha (\rho_1 U\cdot u_2)\|_{L^2} &\le C\lt(\|\Lambda^\alpha (\rho_1 U)\|_{L^2}  \|u_2\|_{L^\infty} + \|\rho_1 U\|_{L^{\frac12-\frac\alpha d}} \|\Lambda^\alpha u_2\|_{L^{\frac d\alpha}}\rt)\\
&\le C\|u_2\|_{H^s} (\|\Lambda^\alpha U\|_{L^2} + \|\Lambda^\alpha (h_1 U)\|_{L^2})\\
&\le C\|u_2\|_{H^s}(1+\|h_1\|_{H^s})\|\Lambda^\alpha U\|_{L^2},
\end{aligned}\]
where we used $s>\frac d2$.

Then, the self-adjointness of $\Lambda^\alpha$ gives
\[\begin{aligned}
\lt| \intr \rho_1 U \cdot \calC (H, u_2)\,\dx\rt|&\le C\|\Lambda^\alpha (\rho_1 U)\|_{L^2}\|\Lambda^\alpha (H u_2)\|_{L^2}+ C\|\Lambda^\alpha H\|_{L^2}\|\Lambda^\alpha (\rho_1 U\cdot u_2)\|_{L^2}\\
&\le C\|u_2\|_{H^s}(1+\|h_1\|_{H^s}) \|\Lambda^\alpha U\|_{L^2} \|\Lambda^\alpha H\|_{L^2}.
\end{aligned}\]
Together with the preceding estimates, we obtain
\bq\label{eq:l2_diff_est}
\frac\rd\dt\calE_0+c\|\Lambda^\alpha U\|_{L^2}^2 \le C\eta\lt( \calE_0+\frac{\|\Lambda^\alpha H\|_{L^2}^2}{\e^2} +\|\Lambda^\alpha U\|_{L^2}^2 \rt).
\eq
To recover the fractional density norm, set
\[
\lal\Lambda\ral:=(I-\Delta)^{1/2}, \quad \calK_\alpha:=\lal\Lambda\ral^{2\alpha-2}\nabla, \quad \calD_\alpha:=\Lambda\lal\Lambda\ral^{\alpha-1}.
\]
The Fourier symbol of $\calK_\alpha$ is $m_\alpha(\xi) = i\xi\lal\xi\ral^{2\alpha-2}$. Since $0<\alpha\le\frac12$, we have $|m_\alpha(\xi)| \le C$,
\[
\lal\xi\ral^{1-\alpha}|m_\alpha(\xi)|  = |\xi|\lal\xi\ral^{\alpha-1} \le C\lal\xi\ral^\alpha,\quad
|\xi|^\alpha|m_\alpha(\xi)|  = |\xi|^{1+\alpha}\lal\xi\ral^{2\alpha-2} \le C\lal\xi\ral^\alpha.
\]
The last inequality uses $2\alpha\le1$ at high frequencies. Plancherel's theorem now gives
\bq\label{eq:K-est}
\|\calK_\alpha f\|_{L^2} \le C\|f\|_{L^2}, \quad \|\calK_\alpha f\|_{H^{1-\alpha}} + \|\Lambda^\alpha\calK_\alpha f\|_{L^2} \le C\|f\|_{H^\alpha}.
\eq
Moreover, $\lal\xi\ral^{2\alpha} \le C\lt( 1+|\xi|^2\lal\xi\ral^{2\alpha-2} \rt)$, and thus
\bq\label{eq:D-equiv}
\|f\|_{H^\alpha}^2 \le C\lt( \|f\|_{L^2}^2 + \|\calD_\alpha f\|_{L^2}^2 \rt).
\eq
Introduce
\[
\calX(t):=\lal\calK_\alpha H,U\ral, \quad \calE_\theta:=\calE_0+\theta\calX.
\]
Since $|\calX|\le C\|H\|_{L^2}\|U\|_{L^2}\le C\calE_0$, $\calE_\theta$ is equivalent to $\calE_0$ for $\theta>0$ sufficiently small, uniformly for $0<\e\le1$.

Set $R:=h_1U+Hu_2$. Using \eqref{eq:d-h} and \eqref{eq:d-u}, we obtain
\begin{align*}
\frac\rd\dt\calX &=\|\lal\Lambda\ral^{\alpha-1}\nabla\cdot U\|_{L^2}^2 +\lal\lal\Lambda\ral^{\alpha-1}\nabla\cdot R, \lal\Lambda\ral^{\alpha-1}\nabla\cdot U\ral -\lal\calK_\alpha H,u_1\cdot\nabla U+U\cdot\nabla u_2\ral \\
&\quad -\frac{\gamma}{\e^2}\lal\calK_\alpha H,A_1\nabla H\ral -\frac{\gamma}{\e^2}\lal\calK_\alpha H,(A_1-A_2)\nabla h_2\ral -\lal\calK_\alpha H,\Lambda^{2\alpha}U\ral +\lal\calK_\alpha H,\calC(h_1,U)+\calC(H,u_2)\ral.
\end{align*}
We also have $|\xi|\lal\xi\ral^{\alpha-1} \le C|\xi|^\alpha$, $0<\alpha\le1$. Indeed, the left-hand side is bounded by $C|\xi|$ for $|\xi|\le1$ and is comparable to $|\xi|^\alpha$ for $|\xi|\ge1$. Plancherel's theorem gives
\[
\|\lal\Lambda\ral^{\alpha-1}\nabla\cdot U\|_{L^2} \le C\|\Lambda^\alpha U\|_{L^2}.
\]
Note that
\[
\|R\|_{H^\alpha}\le C\eta\lt(\|U\|_{H^\alpha}+\|H\|_{H^\alpha}\rt),
\]
and hence
\[
\lt| \lal\lal\Lambda\ral^{\alpha-1}\nabla\cdot R, \lal\Lambda\ral^{\alpha-1}\nabla\cdot U\ral \rt| \le C\eta\lt( \calE_0+\frac{\|\calD_\alpha H\|_{L^2}^2}{\e^2} +\|\Lambda^\alpha U\|_{L^2}^2 \rt).
\]
For the principal pressure term,
\[
\lal\calK_\alpha H,\nabla H\ral = \|\calD_\alpha H\|_{L^2}^2.
\]
Writing $A_1=1+(A_1-1)$, we obtain
\[
\begin{aligned}
 \lt| \lal\calK_\alpha H,(A_1-1)\nabla H\ral \rt|
& \le \|\calK_\alpha H\|_{H^{1-\alpha}} \|(A_1-1)\nabla H\|_{H^{\alpha-1}}\\
& \le C\|A_1-1\|_{H^s}\|H\|_{H^\alpha}^2\\
& \le C\eta\lt( \|H\|_{L^2}^2 + \|\calD_\alpha H\|_{L^2}^2 \rt).
\end{aligned}
\]
Here we used  $\|A_1-1\|_{H^s} \le C\|h_1\|_{H^s} \le C\eta$. After reducing $\eta$ if necessary, the term containing $\|\calD_\alpha H\|_{L^2}^2$ is absorbed into the principal part.
Thus
\bq\label{eq:p-coer}
\lal\calK_\alpha H,A_1\nabla H\ral \ge \frac12\|\calD_\alpha H\|_{L^2}^2 - C\eta\|H\|_{L^2}^2.
\eq

For the remaining pressure term, the mean value theorem gives
\[
\|A_1-A_2\|_{L^2} \le C\|H\|_{L^2}.
\]
By \eqref{eq:K-est} and Sobolev embedding,
\[
\|\calK_\alpha H\|_{L^{\frac{2d}{d-2+2\alpha}}}  \le C\|\calK_\alpha H\|_{H^{1-\alpha}} \le C\|H\|_{H^\alpha},\quad
\|\nabla h_2\|_{L^{\frac d{1-\alpha}}}  \le C\|\nabla h_2\|_{\dot H^{\frac d2-1+\alpha}} \le C\|\nabla h_2\|_{H^{s+\alpha-1}}.
\]
The last inequality follows from $s>\frac d2$. Moreover,
\[
\frac{d-2+2\alpha}{2d} + \frac12 + \frac{1-\alpha}{d} = 1.
\]
Set
\[
M_2(t) := \frac1{\e} \|\nabla h_2(t)\|_{H^{s+\alpha-1}}.
\]
Then H\"older's inequality, \eqref{eq:D-equiv}, \eqref{eq:E0-equiv}, and Young's inequality give, for every $\kappa>0$,
\[
\frac1{\e^2} \lt| \lal\calK_\alpha H,(A_1-A_2)\nabla h_2\ral \rt| \le \frac C{\e^2} \|H\|_{H^\alpha} \|H\|_{L^2} \|\nabla h_2\|_{H^{s+\alpha-1}} \le \frac{\kappa}{\e^2} \|\calD_\alpha H\|_{L^2}^2 + C_\kappa\lt(1+M_2(t)^2\rt)\calE_0.
\]
For the first convective term, duality, \eqref{eq:K-est}, and the multiplier property of $H^s$ give
\[
\lt| \lal\calK_\alpha H,u_1\cdot\nabla U\ral \rt| \le \|\calK_\alpha H\|_{H^{1-\alpha}} \|u_1\cdot\nabla U\|_{H^{\alpha-1}} \le C\|u_1\|_{H^s} \|H\|_{H^\alpha} \|U\|_{H^\alpha}.
\]
For the second convective term,
\[
\lt| \lal\calK_\alpha H,U\cdot\nabla u_2\ral \rt| \le C\|\calK_\alpha H\|_{L^{\frac{2d}{d-2+2\alpha}}} \|U\|_{L^{\frac{2d}{d-2\alpha}}} \|\nabla u_2\|_{L^d} \le C\|u_2\|_{H^s} \|H\|_{H^\alpha} \|\Lambda^\alpha U\|_{L^2}.
\]
These estimates imply
\[
\lt|\lal\calK_\alpha H,u_1\cdot\nabla U+U\cdot\nabla u_2\ral\rt| \le C\eta\lt( \calE_0+\frac{\|\calD_\alpha H\|_{L^2}^2}{\e^2} +\|\Lambda^\alpha U\|_{L^2}^2 \rt).
\]
For the linear fractional term, \eqref{eq:K-est} gives
\[
\lt|\lal\calK_\alpha H,\Lambda^{2\alpha}U\ral\rt| \le C\|H\|_{H^\alpha}\|\Lambda^\alpha U\|_{L^2}.
\]
Hence, for every $\kappa>0$,
\[
\lt|\lal\calK_\alpha H,\Lambda^{2\alpha}U\ral\rt| \le \frac{\kappa}{\e^2}\|\calD_\alpha H\|_{L^2}^2 +C_\kappa\|\Lambda^\alpha U\|_{L^2}^2 +C_\kappa\calE_0.
\]
For the nonlinear alignment terms, fractional product estimates \eqref{eq:hu-low} give
\[\begin{aligned}
\lt|\lal\calK_\alpha H,\Lambda^{2\alpha}(h_1U)\ral\rt|   \le C\|\Lambda^\alpha \mathcal{K}_\alpha H\|_{L^2} \|\Lambda^\alpha(h_1 U)\|_{L^2} \le C\|h_1\|_{H^s}\|H\|_{H^\alpha}\|U\|_{H^\alpha}.
\end{aligned}\]
Choose $q_1,q_2$ by
\[
\frac1{q_1}=\frac12-\frac{1-\alpha}{d},
\quad
\frac1{q_2}=\frac12-\frac{\alpha}{d}.
\]
Then $\frac1{q_1}+\frac1{q_2}+\frac1d=1$, and
\[
\|\calK_\alpha H\|_{L^{q_1}}\le C\|H\|_{H^\alpha},
\quad
\|U\|_{L^{q_2}}\le C\|U\|_{H^\alpha}.
\]
Since $0<\alpha\le\frac12$ and $s>\frac d2+1-2\alpha$, we get
\[
\|\Lambda^{2\alpha}h_1\|_{L^d}\le C\|h_1\|_{H^s},
\]
and hence
\[
\lt|\intr\calK_\alpha H\cdot U\Lambda^{2\alpha}h_1\,\dx\rt| \le C\|h_1\|_{H^s}\|H\|_{H^\alpha}\|U\|_{H^\alpha}.
\]
Similarly,
\[
\lt|\lal\calK_\alpha H,\Lambda^{2\alpha}(Hu_2)\ral\rt| + \lt|\lal\calK_\alpha H,u_2\Lambda^{2\alpha}H\ral\rt| \le C\|u_2\|_{H^s}\|H\|_{H^\alpha}^2.
\]
Combining these bounds gives
\[
\lt| \lal\calK_\alpha H,\calC(h_1,U)+\calC(H,u_2)\ral \rt| \le C\eta\lt( \calE_0+\frac{\|\calD_\alpha H\|_{L^2}^2}{\e^2} +\|\Lambda^\alpha U\|_{L^2}^2 \rt).
\]
Now, we combine \eqref{eq:l2_diff_est} with $\theta\frac\rd\dt\calX$ for $\theta>0$ sufficiently small and write $\mathcal{E}_\theta$ as
\[
\mathcal{E}_\theta := \mathcal{E}_0 + \theta\mathcal{X}.
\]
For the choice of $\theta$, we first choose $\kappa>0$ so that the terms containing $\e^{-2}\|\calD_\alpha H\|_{L^2}^2$ are absorbed by \eqref{eq:p-coer}. Next choose $\theta>0$ so that the terms proportional to $\theta C_\kappa\|\Lambda^\alpha U\|_{L^2}^2$ are absorbed by the alignment dissipation. A further reduction of $\eta$ absorbs the remaining nonlinear terms and gives
\[
\frac\rd\dt\calE_\theta +c_1\|\Lambda^\alpha U\|_{L^2}^2 +\frac{c_2}{\e^2}\|\calD_\alpha H\|_{L^2}^2 \le C\lt(1+M_2(t)^2\rt)\calE_\theta.
\]
Since
\[
\int_0^T M_2(t)^2\,\dt = \int_0^T\frac{\|\nabla h_2(t)\|_{H^{s+\alpha-1}}^2}{\e^2}\,\dt <\infty,
\]
Gr\"onwall's inequality and \eqref{eq:E0-equiv} give the stated bound. When the initial data agree, $\calE_\theta(0)=0$, so $H\equiv0$ and $U\equiv0$.
\end{proof}

%
%
%
%
%
%

\subsubsection{Viscous construction in the low-order regime}

We construct solutions by a viscous approximation, following the standard scheme for Euler-type systems; see, for instance, \cite[proof of Theorem 3.1]{CJ22}. Let $J_\delta$ be a nonnegative Friedrichs mollifier and set
\[
h_{0,\delta}:=J_\delta h_0, \quad u_{0,\delta}:=J_\delta u_0.
\]
After reducing the smallness constant if necessary, Sobolev embedding gives
\[
\frac12\le1+h_0\le\frac32.
\]
Since $J_\delta$ is positivity preserving and preserves constants, the same bound holds for $1+h_{0,\delta}$. Consider
\bq\label{eq:visc}
\begin{aligned}
&\pa_t h_\delta +\nabla\cdot\lt((1+h_\delta)u_\delta\rt) =\delta\Delta h_\delta,\\
&\pa_t u_\delta +u_\delta\cdot\nabla u_\delta +\frac{\gamma}{\e^2}(1+h_\delta)^{\gamma-2}\nabla h_\delta =-\Lambda^{2\alpha}u_\delta -\Lambda^{2\alpha}(h_\delta u_\delta) +u_\delta\Lambda^{2\alpha}h_\delta +\delta\Delta u_\delta,
\end{aligned}
\eq
with initial data $(h_{0,\delta},u_{0,\delta})$. For each fixed $\delta>0$, standard semilinear parabolic theory gives a smooth solution on a maximal time interval. The approximate solution can be continued as long as its $H^s$ norm remains bounded and $1+h_\delta$ remains in a compact subset of $(0,\infty)$.

The a priori argument of Proposition \ref{prop:uni-apri-low} extends to \eqref{eq:visc}. We record the additional contribution generated by the density viscosity. Define
\[
a(z):=(1+z)^{\gamma-3}, \quad a_\delta:=a(h_\delta), \quad g_r:=\Lambda^r h_\delta, \quad r=0,s.
\]
For
\[
\calE_r(t) := \frac12\intr a_\delta|g_r|^2 \,\dx,
\]
we have
\[
\frac\rd\dt\calE_r = \frac12\intr a'(h_\delta)\pa_t h_\delta|g_r|^2 \,\dx + \intr a_\delta g_r\pa_t g_r \,\dx.
\]
The density viscosity in \eqref{eq:visc} gives
\[
(\pa_t h_\delta)_{\rm visc} = \delta\Delta h_\delta, \quad (\pa_t g_r)_{\rm visc} = \delta\Delta g_r.
\]
Hence its contribution to $\frac\rd\dt\calE_r$ is
\[
\calV_r = \frac{\delta}{2} \intr a'(h_\delta)\Delta h_\delta|g_r|^2 \,\dx + \delta\intr a_\delta g_r\Delta g_r \,\dx.
\]
For the first term, integration by parts gives
\[
\frac{\delta}{2} \intr a'(h_\delta)\Delta h_\delta|g_r|^2 \,\dx = -\frac{\delta}{2} \intr a''(h_\delta) |\nabla h_\delta|^2|g_r|^2 \,\dx -\delta\intr a'(h_\delta) \nabla h_\delta\cdot\nabla g_r\,g_r \,\dx.
\]
Similarly,
\[
\delta\intr a_\delta g_r\Delta g_r \,\dx  = -\delta\intr a_\delta|\nabla g_r|^2 \,\dx -\delta\intr a'(h_\delta) \nabla h_\delta\cdot\nabla g_r\,g_r \,\dx.
\]
Combining the two identities, we obtain
\[
\calV_r = -\delta\intr a_\delta|\nabla g_r|^2 \,\dx  -2\delta\intr a'(h_\delta) \nabla h_\delta\cdot\nabla g_r\,g_r \,\dx  -\frac{\delta}{2} \intr a''(h_\delta) |\nabla h_\delta|^2|g_r|^2 \,\dx.
\]
Since $\frac12 \le 1+h_\delta \le \frac32$, there exist constants $c,C>0$ such that

\[
a_\delta\ge c, \quad |a'(h_\delta)| + |a''(h_\delta)| \le C.
\]
For $r=0$, we have $g_0=h_\delta$, and hence
\[
\begin{aligned}
\calV_0 &\le -c\delta\|\nabla h_\delta\|_{L^2}^2  + C\delta \lt( \|h_\delta\|_{L^\infty} + \|h_\delta\|_{L^\infty}^2 \rt) \|\nabla h_\delta\|_{L^2}^2\\
&\le -c\delta\|\nabla h_\delta\|_{L^2}^2  + C\delta \lt( \|h_\delta\|_{H^s} + \|h_\delta\|_{H^s}^2 \rt) \|\nabla h_\delta\|_{L^2}^2.
\end{aligned}
\]
We next consider $r=s$. If $d>2$, Sobolev embedding gives

\[
\|\nabla h_\delta\|_{L^d} \le C\|h_\delta\|_{H^s}
\quad
\text{and}
\quad
\|g_s\|_{L^{\frac{2d}{d-2}}} \le C\|\nabla g_s\|_{L^2}.
\]
Consequently,
\[
\intr |\nabla h_\delta| |\nabla g_s| |g_s| \,\dx \le \|\nabla h_\delta\|_{L^d} \|\nabla g_s\|_{L^2} \|g_s\|_{L^{\frac{2d}{d-2}}} \le C\|h_\delta\|_{H^s} \|\nabla g_s\|_{L^2}^2
\]
and
\[
\intr |\nabla h_\delta|^2|g_s|^2 \,\dx \le \|\nabla h_\delta\|_{L^d}^2 \|g_s\|_{L^{\frac{2d}{d-2}}}^2 \le C\|h_\delta\|_{H^s}^2 \|\nabla g_s\|_{L^2}^2.
\]
If $d=2$, choose $p>2$ sufficiently close to $2$ so that $H^{s-1}(\R^2) \hookrightarrow L^p(\R^2)$, and set $q:=\frac{2p}{p-2}$. Since $s>1$, we find
\[
\|g_s\|_{H^1} = \|\Lambda^s h_\delta\|_{H^1} \le C\lt( \|\nabla h_\delta\|_{L^2} + \|\Lambda^{s+1}h_\delta\|_{L^2} \rt) \le C\|\nabla h_\delta\|_{H^s}.
\]
Here the first inequality follows by splitting into low and high frequencies. The embedding $H^1(\R^2)\hookrightarrow L^q(\R^2)$ then gives $\|g_s\|_{L^q} \le C\|\nabla h_\delta\|_{H^s}$. It follows that
\[
\intr |\nabla h_\delta| |\nabla g_s| |g_s| \,\dx \le \|\nabla h_\delta\|_{L^p} \|\nabla g_s\|_{L^2} \|g_s\|_{L^q} \le C\|h_\delta\|_{H^s} \|\nabla h_\delta\|_{H^s}^2
\]
and
\[
\intr |\nabla h_\delta|^2|g_s|^2 \,\dx \le \|\nabla h_\delta\|_{L^p}^2 \|g_s\|_{L^q}^2 \le C\|h_\delta\|_{H^s}^2 \|\nabla h_\delta\|_{H^s}^2.
\]
Thus, in every dimension $d\ge2$,

\[
\calV_s \le -c\delta \|\nabla\Lambda^s h_\delta\|_{L^2}^2 + C\delta \lt( \|h_\delta\|_{H^s} + \|h_\delta\|_{H^s}^2 \rt) \|\nabla h_\delta\|_{H^s}^2.
\]
Combining the estimates for $\calV_0$ and $\calV_s$, we obtain
\[
\calV_0+\calV_s  \le -c\delta\|\nabla h_\delta\|_{H^s}^2 + C\delta \lt( \|h_\delta\|_{H^s} + \|h_\delta\|_{H^s}^2 \rt) \|\nabla h_\delta\|_{H^s}^2.
\]
Under the bootstrap smallness condition, the last term is absorbed into the first one. The velocity viscosity contributes to $-\delta\|\nabla u_\delta\|_{H^s}^2$. 

We also consider the mixed density-dissipation functionals
\[
\calX_0 := -\intr h_\delta\nabla\cdot u_\delta \,\dx
\quad
\text{and}
\quad
\calX_s := \frac12 \|\Lambda^{s+2\alpha-1}h_\delta\|_{L^2}^2 - \intr \Lambda^{s+2\alpha-1}h_\delta \Lambda^{s-1}\nabla\cdot u_\delta \,\dx.
\]
Their viscous contributions are
\[
(\calX_0)_{\rm visc} = 2\delta\intr \nabla h_\delta\cdot\nabla\nabla\cdot u_\delta \,\dx
\]
and
\[
(\calX_s)_{\rm visc} = -\delta \|\Lambda^{s+2\alpha}h_\delta\|_{L^2}^2  + 2\delta\intr \Lambda^{s+2\alpha-1}\nabla h_\delta\cdot \Lambda^{s-1}\nabla\nabla\cdot u_\delta \,\dx.
\]
Since $2\alpha\le1$, we get $s+2\alpha\le s+1$. Young's inequality now gives
\[
|(\calX_0)_{\rm visc}| \le C\delta\lt( \frac{\|\nabla h_\delta\|_{H^s}^2}{\e^2} + \|\nabla u_\delta\|_{H^s}^2 \rt)
\quad 
\text{and}
\quad
(\calX_s)_{\rm visc} \le C\delta\lt( \frac{\|\nabla h_\delta\|_{H^s}^2}{\e^2} + \|\nabla u_\delta\|_{H^s}^2 \rt).
\]
After multiplying the mixed estimates by the small constants used in the modified energy, these terms are absorbed by the viscous dissipation. As a result, we obtain
\bq\label{eq:visc-est}
\begin{aligned}
&\frac{\|h_\delta(t)\|_{H^s}^2}{\e^2} +\|u_\delta(t)\|_{H^s}^2  +\int_0^t\lt( \frac{\|\nabla h_\delta\|_{H^{s+\alpha-1}}^2}{\e^2} + \|\Lambda^\alpha u_\delta\|_{H^s}^2 \rt)\rd\tau  +\delta\int_0^t\lt( \frac{\|\nabla h_\delta\|_{H^s}^2}{\e^2} + \|\nabla u_\delta\|_{H^s}^2 \rt)\rd\tau\\
&\quad \le C\lt( \frac{\|h_{0,\delta}\|_{H^s}^2}{\e^2} + \|u_{0,\delta}\|_{H^s}^2 \rt),
\end{aligned}
\eq
where $C>0$ is independent of $\delta$ and $\e$.

Since $\|h_{0,\delta}\|_{H^s} \le \|h_0\|_{H^s}$ and $\|u_{0,\delta}\|_{H^s} \le \|u_0\|_{H^s}$, the first-exit argument based on \eqref{eq:visc-est} improves the bootstrap bound uniformly in $\delta$. The density remains between two positive constants, and the parabolic continuation criterion gives a common existence time $T_\e>0$ for all $\delta$, with $\e$ fixed. Set $T:=T_\e$. For fixed $\e>0$, the equations and \eqref{eq:visc-est} give uniform bounds for the time derivatives. Indeed, by Moser-type estimate,
\[
\|\nabla\cdot\lt((1+h_\delta)u_\delta\rt)\|_{H^{s-1}} \le C\lt( 1+\|h_\delta\|_{H^s} \rt) \|u_\delta\|_{H^s}.
\]
Moreover, since $2\alpha\le1$,
\begin{align*}
&\|u_\delta\cdot\nabla u_\delta\|_{H^{s-1}} + \|(1+h_\delta)^{\gamma-2}\nabla h_\delta\|_{H^{s-1}}  + \|\Lambda^{2\alpha}u_\delta\|_{H^{s-1}} + \|\Lambda^{2\alpha}(h_\delta u_\delta)-u_\delta\Lambda^{2\alpha}h_\delta\|_{H^{s-1}} \\
&\quad  \le C_\e\lt( 1+\|h_\delta\|_{H^s} +\|u_\delta\|_{H^s} \rt)^2.
\end{align*}
The constant $C_\e$ may depend on $\e$, which is fixed in the construction.  Indeed, the pointwise estimate

\[\begin{aligned}
||\xi|^{s+2\alpha -1} - |\xi|^{s-1} |\eta|^{2\alpha}| &\le C |\xi|^{s-1} |\xi-\eta|(|\xi|^{2\alpha-1} + |\eta|^{2\alpha-1})\\
&\le C\lt( |\xi-\eta|^{s+2\alpha -1} + |\xi-\eta| |\xi|^{s+2\alpha-2} + |\xi|^s |\eta|^{2\alpha-1}\rt).
\end{aligned}\]
Hence,
\[\begin{aligned}
&\|\Lambda^{s+2\alpha -1} (h_\delta u_\delta) - \Lambda^{s-1}(u_\delta \Lambda^{2\alpha} h_\delta)\|_{L^2}\\
&\quad \le C\|\hat u_\delta\|_{L_\xi^1} \|h_\delta\|_{H^{s+2\alpha-1}} + C\| |\xi| \hat h_\delta\|_{L_\xi^{\tilde{p}'}}\| |\xi|^{s+2\alpha-2} \hat u_\delta\|_{L_\xi^{\tilde{p}}} + C\|h_\delta\|_{H^s} \| |\xi|^{2\alpha-1} \hat u_\delta\|_{L_\xi^1}\\
&\quad \le C\|h_\delta\|_{H^s}\|u_\delta\|_{H^s},
\end{aligned}\]
where $\tilde{p}$, $\tilde{p}' \in (1,2)$ are given by
\[
\tilde p := \frac{2s+4\alpha-2}{s+2\alpha}, \quad \tilde p':= \frac{2s+4\alpha-2}{2s+4\alpha-3},
\]
which satisfy
\[\begin{aligned}
\||\xi|^{s+2\alpha-2} \hat f\|_{L_\xi^p} &\le C\| \hat f\|_{L_\xi^1}^{\frac{2-\tilde p}{\tilde p}} \| |\xi|^{s+2\alpha-1} \hat f\|_{L_\xi^2}^{\frac{2(\tilde p -1)}{\tilde p}}\le C\|f\|_{H^s}, \\
 \| |\xi|\hat f\|_{L_\xi^{\tilde p'}} &\le \| \hat f\|_{L_\xi^1}^{\frac{2-\tilde p'}{\tilde p'}} \| |\xi|^{s+2\alpha-1} \hat f\|_{L_\xi^2}^{\frac{2(\tilde p' -1)}{\tilde p'}}\le C\|f\|_{H^s}.
 \end{aligned}
\]
On the other hand, \eqref{eq:visc-est} gives
\[
\delta \|\Delta h_\delta\|_{L^2(0,T;H^{s-1})} + \delta \|\Delta u_\delta\|_{L^2(0,T;H^{s-1})} \le C\delta^{1/2}.
\]
Consequently, $\pa_t h_\delta$ and $\pa_t u_\delta$ are bounded in $L^2\lt(0,T;H^{s-1}(\R^d)\rt)$ uniformly in $\delta$.

Fix $R>0$ and $0<\sigma<1$. The compact and continuous embeddings $H^s(B_R) \Subset H^{s-\sigma}(B_R) \hookrightarrow H^{s-1}(B_R)$ allow us to apply the Aubin--Lions lemma on $B_R$. A diagonal argument
then yields, up to a subsequence,
\bq\label{eq:loc-conv}
h_\delta\to h, \quad u_\delta\to u \quad\mbox{strongly in }  L^2\lt(0,T;H^{s-\sigma}_{\rm loc}(\R^d)\rt).
\eq
We choose $\sigma>0$ sufficiently small so that $s-\sigma>\frac d2$ and $s-\sigma>\alpha$. Since $s-\sigma>\frac d2$, the space $H^{s-\sigma}_{\rm loc}$ is an algebra. The local strong convergence and the uniform $L^\infty(0,T;H^s)$ bound give
\[
(1+h_\delta)u_\delta \to (1+h)u \quad\mbox{strongly in }  L^2\lt(0,T;H^{s-\sigma}_{\rm loc}(\R^d)\rt)
\]
and
\[
u_\delta\cdot\nabla u_\delta \to u\cdot\nabla u \quad\mbox{strongly in }  L^2\lt(0,T;H^{s-\sigma-1}_{\rm loc}(\R^d)\rt).
\]
For the pressure term, define
\[
Q_\gamma(z) := \begin{cases}
\frac{\gamma}{\gamma-1}(1+z)^{\gamma-1}, & \gamma>1,\\[2mm]
\log(1+z), & \gamma=1.
\end{cases}
\]
Then
\[
\nabla Q_\gamma(h_\delta) = \gamma(1+h_\delta)^{\gamma-2}\nabla h_\delta.
\]
The density bounds and the Sobolev composition estimate give
\[
Q_\gamma(h_\delta) \to Q_\gamma(h) \quad\mbox{strongly in }  L^2\lt(0,T;H^{s-\sigma}_{\rm loc}(\R^d)\rt).
\]
Thus the continuity, convection, and pressure terms pass to the limit in the sense of distributions.

It remains to treat the nonlocal alignment terms. Since $s>\frac d2$,
\[
\|h_\delta u_\delta\|_{L^\infty(0,T;L^2)} \le C\|h_\delta\|_{L^\infty(0,T;H^s)} \|u_\delta\|_{L^\infty(0,T;L^2)} \le C.
\]
After extracting a further subsequence,
\[
h_\delta u_\delta \rightharpoonup G \quad\mbox{weakly in }  L^2\lt(0,T;L^2(\R^d)\rt).
\]
The local strong convergence \eqref{eq:loc-conv} identifies $G=hu$. Let $\varphi \in C_c^\infty\lt((0,T)\times\R^d;\R^d\rt)$. Since $\Lambda^{2\alpha}\varphi \in L^2\lt(0,T;L^2(\R^d)\rt)$, the weak convergence of $h_\delta u_\delta$ gives
\[
\int_0^T \lal \Lambda^{2\alpha}(h_\delta u_\delta), \varphi \ral \,\dt  = \int_0^T \lal h_\delta u_\delta, \Lambda^{2\alpha}\varphi \ral \,\dt \to \int_0^T \lal hu, \Lambda^{2\alpha}\varphi \ral \,\dt.
\]
Thus
\[
\Lambda^{2\alpha}(h_\delta u_\delta) \to \Lambda^{2\alpha}(hu)
\]
in the sense of distributions.

For the remaining alignment term,
\[
\int_0^T\intr u_\delta\Lambda^{2\alpha}h_\delta \cdot\varphi \,\dx\dt = \int_0^T\intr \Lambda^\alpha h_\delta \Lambda^\alpha(u_\delta\cdot\varphi) \,\dx\dt.
\]
The uniform $H^s$ bound gives
\[
\Lambda^\alpha h_\delta \rightharpoonup \Lambda^\alpha h \quad\mbox{weakly in }  L^2\lt(0,T;L^2(\R^d)\rt),
\]
whereas \eqref{eq:loc-conv} and $s-\sigma>\alpha$ give
\[
\Lambda^\alpha(u_\delta\cdot\varphi) \to \Lambda^\alpha(u\cdot\varphi) \quad\mbox{strongly in }  L^2\lt(0,T;L^2(\R^d)\rt).
\]
It follows that
\[
u_\delta\Lambda^{2\alpha}h_\delta \to u\Lambda^{2\alpha}h
\]
in the sense of distributions. The preceding viscous estimate also gives
\[
\delta \|\Delta h_\delta\|_{L^2(0,T;H^{s-1})} + \delta \|\Delta u_\delta\|_{L^2(0,T;H^{s-1})} \le C\delta^{1/2} \to0.
\]
Hence
\[
\delta\Delta h_\delta\to0, \quad \delta\Delta u_\delta\to0 \quad\mbox{strongly in }  L^2\lt(0,T;H^{s-1}(\R^d)\rt).
\]
It follows that $(h,u)$ solves \eqref{h-u-system}. The uniform bounds and weak lower semicontinuity give
\[
\nabla h \in L^2\lt(0,T;H^{s+\alpha-1}(\R^d)\rt), \quad \Lambda^\alpha u \in L^2\lt(0,T;H^s(\R^d)\rt),
\]
and the limit satisfies the estimate of Proposition
\ref{prop:uni-apri-low}.

We finally verify continuity at the top Sobolev level. The time derivative bounds obtained above imply
\[
(h,u) \in C\lt([0,T]; H^{s-1}(\R^d)\times H^{s-1}(\R^d) \rt).
\]
Together with the uniform $H^s$ bound, this gives
\[
(h,u) \in C_w\lt([0,T]; H^s(\R^d)\times H^s(\R^d) \rt).
\]
Interpolation also yields
\[
(h,u) \in C\lt([0,T]; H^{s-\sigma}(\R^d)\times H^{s-\sigma}(\R^d) \rt)
\]
for every $0<\sigma<1$.

Since
\[
h_{0,\delta}\to h_0, \quad u_{0,\delta}\to u_0 \quad\mbox{strongly in }  H^s(\R^d),
\]
the weak formulation and the weak $H^s$ continuity identify
\[
(h,u)(0)=(h_0,u_0).
\]
To upgrade the weak continuity to strong $H^s$ continuity, consider
\[
\calE_{\rm b}(t) := \frac{\gamma}{2\e^2} \intr (1+h)^{\gamma-3} \lt( |h|^2+|\Lambda^s h|^2 \rt) \dx  +\frac12\lt( \|u\|_{L^2}^2 + \|\Lambda^s u\|_{L^2}^2 \rt).
\]
Repeating the basic-energy calculation after a spatial Friedrichs regularization gives
\[
\frac\rd\dt\calE_{\rm b} = F
\]
for some $F\in L^1(0,T)$. The integrability of $F$ follows from the estimates in Sections \ref{ssec:low-zero} and \ref{ssec:low-h-ene} and Proposition \ref{prop:uni-apri-low}. Consequently,
\[
\calE_{\rm b} \in W^{1,1}(0,T) \subset C([0,T]).
\]
Let $t_n\to t$. The strong continuity in $H^{s-\sigma}$, with $s-\sigma>\frac d2$, gives
\[
(1+h(t_n))^{\gamma-3} \to (1+h(t))^{\gamma-3} \quad\mbox{in } L^\infty(\R^d).
\]
On the other hand,
\[
h(t_n)\rightharpoonup h(t), \quad u(t_n)\rightharpoonup u(t) \quad\mbox{weakly in }  H^s(\R^d).
\]
Weak lower semicontinuity applies separately to the weighted top-order density term and the top-order velocity term. Since their sum converges by the continuity of $\calE_{\rm b}$ and the lower-order terms converge strongly, both top-order norms converge. Weak convergence together with convergence of these norms gives
\[
(h,u) \in C\lt([0,T];H^s(\R^d)\times H^s(\R^d)\rt).
\]
Lemma \ref{lem:stab} gives uniqueness and completes the construction for $0<\alpha\le\frac12$.

%
%
%
%
%
%
 
\subsubsection{Higher-order regime}

Suppose that $\frac12<\alpha<1$. The critical Besov local theory for the Euler--alignment system with pressure applies; see \cite[Theorem 4.2]{BMTX24}. Its fractional order is $2\alpha\in(1,2)$ in the present notation. For each fixed $\e>0$, set $P_\e(\rho):=\e^{-2}\rho^\gamma$. The density variable used in \cite{BMTX24} is obtained by a smooth change of variables from $h/\e$. On the range $\frac12\le\rho\le\frac32$, Sobolev composition estimates show that its $H^s$ norm is equivalent to $\|h\|_{H^s}/\e$, with constants independent of $\e$. Since $s>\frac d2$ and $\frac d2+1-2\alpha<\frac d2$, the $H^s$ regularity assumed here implies the critical Besov regularity required in \cite{BMTX24}. Moreover, the acoustic coefficient enters the symmetrized system through the skew-symmetric coupling between the density and velocity variables, so the corresponding smallness threshold is independent of $\e$.

Applying the local theory to smooth approximations and using Proposition \ref{prop:uni-apri-high}, we obtain uniform $H^s$ bounds on the existence interval. Compactness preserves the required regularity, and uniqueness in the critical Besov class identifies the limit. The energy-continuity argument above then gives
\[
(h,u)\in C\lt([0,T];H^s(\R^d)\times H^s(\R^d)\rt).
\]
This proves the existence and uniqueness assertions of Theorem \ref{thm:loc} in the higher-order regime.

The constructions above are invariant under time translation. For each fixed $\e>0$, a solution in the small-data class can be restarted from any time $t_0$ as long as
\[
\frac12\le1+h(t_0,x)\le\frac32
\quad\text{and}\quad
\frac{\|h(t_0)\|_{H^s}}{\e}+\|u(t_0)\|_{H^s}<\eta_\star.
\]
The lifespan of the restarted solution depends only on these bounds and the distance of the density from zero. If the maximal time were finite while both bounds remained strict, one could restart the solution at a time sufficiently close to the maximal time and extend it further. This proves the continuation assertion of Theorem \ref{thm:loc}.

%
%
%
%
%
%

\subsection{Global continuation}

We now complete the proof of Theorem \ref{thm:global}. Let $(h,u)$ be the local strong solution to \eqref{h-u-system} on its maximal interval of existence $[0,T_*)$. We prove that the solution norm
remains uniformly small on $[0,T_*)$. The a priori estimate used below is Proposition \ref{prop:uni-apri-low} in the regime $0<\alpha\le \frac12$, and Proposition \ref{prop:uni-apri-high} in the regime $\frac12<\alpha<1$. In both cases, we write the corresponding estimate in the form
\bq\label{uni-apri-comb}
\frac{\|h(t)\|_{H^s}^2}{\e^2} +\|u(t)\|_{H^s}^2 +\int_0^t \lt( \frac{\|\nabla h (\tau)\|_{H^{\sigma_\alpha}}^2}{\e^2} +\|\Lambda^\alpha u (\tau)\|_{H^s}^2 \rt) \rd\tau \le C\lt( \frac{\|h_0\|_{H^s}^2}{\e^2} +\|u_0\|_{H^s}^2 \rt),
\eq
where
\[
\sigma_\alpha =
\begin{cases}
s+\alpha-1, & 0<\alpha\le \frac12,\\[1mm]
s-\alpha, & \frac12<\alpha<1.
\end{cases}
\]
Set
\[
T_\eta := \sup\lt\{ T\in(0,T_*)\,:\,\sup_{0\le t\le T} \lt( \frac{\|h(t)\|_{H^s}}{\e} +\|u(t)\|_{H^s} \rt) \le \eta \rt\}.
\]
For $\eta>0$ fixed sufficiently small, the smallness of the initial data gives $T_\eta>0$. By Propositions \ref{prop:uni-apri-low} or \ref{prop:uni-apri-high}, according to the value of $\alpha$, for every $T<T_\eta$ and every $t\in[0,T]$, we have
\[
\frac{\|h(t)\|_{H^s}^2}{\e^2} +\|u(t)\|_{H^s}^2 \le C\lt( \frac{\|h_0\|_{H^s}^2}{\e^2} +\|u_0\|_{H^s}^2 \rt).
\]
Choosing the initial size sufficiently small, we make the right-hand side less than $\frac{\eta^2}8$. Hence, we have
\[
\sup_{0\le t\le T} \lt( \frac{\|h(t)\|_{H^s}}{\e} +\|u(t)\|_{H^s} \rt) \le \frac{\eta}{2}
\]
for every $T<T_\eta$. By continuity, this improves the bootstrap assumption and hence implies $T_\eta=T_*$. Consequently,
\[
\sup_{0\le t<T_*} \lt( \frac{\|h(t)\|_{H^s}}{\e} + \|u(t)\|_{H^s} \rt) \le \frac{\eta}{2}.
\]
Since $s>\frac d2$ and $0<\e\le1$, Sobolev embedding gives
\[
\sup_{0\le t<T_*} \|h(t)\|_{L^\infty} \le C\sup_{0\le t<T_*}\|h(t)\|_{H^s} \le C\e\eta.
\]
After reducing $\eta$ if necessary, we obtain
\[
\frac12 \le 1+h(t,x) \le \frac32, \quad 0\le t<T_*, \quad x\in\R^d.
\]
Both conditions in the continuation criterion remain strict. Hence $T_*=\infty$. The estimate \eqref{main_uniform_est} follows from \eqref{uni-apri-comb}. This proves Theorem \ref{thm:global}.

%
%
%
%
%
%
\section{Low Mach number limit}\label{sec:limit}

In this section, we prove Theorem \ref{thm:incomp-limit}. We first treat ill-prepared initial data by combining acoustic dispersion with local compactness. We then consider well-prepared data and establish convergence to a prescribed sufficiently regular incompressible solution by a relative-energy argument.

%
%
%
%
%
%
\subsection{Ill-prepared data: acoustic decay and compactness}

We first prove the compactness assertion for ill-prepared initial data. The main ingredient is a frequency-localized decay estimate for the acoustic component. This estimate follows from wave Strichartz estimates; compare \cite[Proposition 4.1]{Dan05} and \cite[Lemma 4.2]{CHH26p}.

For a Fourier multiplier $m(D)$, we use the convention
\[
D:=-i\nabla, \quad \widehat{m(D)f}(\xi) = m(\xi)\widehat f(\xi).
\]
We also recall that the Leray projections are given by
\[
\bbq := \nabla\Delta^{-1}\nabla\cdot = -\nabla\Lambda^{-2}\nabla\cdot, \quad \bbp := I-\bbq.
\]
In particular,
\[
\nabla\cdot\bbp v=0
\]
and $\bbq v$ is the gradient component of $v$.

\begin{lemma}\label{lem:acoustic-decay}
Let $c>0$, $T>0$, and suppose that $(a^\e,w^\e) \in C([0,T];L^2(\R^d)\times L^2(\R^d))$ satisfies
\bq\label{eq:abstract-acoustic}
\begin{aligned}
&\pa_t a^\e+\frac{c}{\e}\nabla\cdot w^\e=F^\e,\\
&\pa_t w^\e+\frac{c}{\e}\nabla a^\e=G^\e,\quad \bbq w^\e=w^\e.
\end{aligned}
\eq
Assume that
\[
\sup_{0<\e\le1} \|(a^\e,w^\e)\|_{L^\infty(0,T;L^2)} <\infty.
\]
Let $\chi\in C_c^\infty(\R^d\setminus\{0\})$. If
\[
\sup_{0<\e\le1} \lt( \|\chi(D)F^\e\|_{L^1(0,T;L^2)} + \|\chi(D)G^\e\|_{L^1(0,T;L^2)} \rt) <\infty,
\]
then, for every compact set $K\Subset\R^d$,
\[
\chi(D)a^\e\to0, \quad \chi(D)w^\e\to0
\]
strongly in $L^2(0,T;L^2(K))$ as $\e\to0$.
\end{lemma}

\begin{proof}
Set $d^\e := \Lambda^{-1}\nabla\cdot w^\e$. Then $\nabla\cdot w^\e = \Lambda d^\e$. Since $\bbq w^\e=w^\e$ and $\bbq = -\nabla\Lambda^{-2}\nabla\cdot$, we have
\[
w^\e = -\nabla\Lambda^{-2}\nabla\cdot w^\e = -\nabla\Lambda^{-1}d^\e.
\]
In particular, Plancherel's theorem gives $\|w^\e(t)\|_{L^2} = \|d^\e(t)\|_{L^2}$. 

Using $\nabla\cdot w^\e = \Lambda d^\e$, the first equation in \eqref{eq:abstract-acoustic} becomes
\[
\pa_t a^\e + \frac{c}{\e}\Lambda d^\e = F^\e.
\]
Applying $\Lambda^{-1}\nabla\cdot$ to the second equation and using $\Lambda^{-1}\nabla\cdot\nabla a^\e = \Lambda^{-1}\Delta a^\e = -\Lambda a^\e$, we obtain
\[
\pa_t d^\e - \frac{c}{\e}\Lambda a^\e = \Lambda^{-1}\nabla\cdot G^\e.
\]
Setting $z_\pm^\e := a^\e\pm i d^\e$, we find
\[
\pa_t z_\pm^\e \mp i\frac{c}{\e}\Lambda z_\pm^\e = F^\e \pm i\Lambda^{-1}\nabla\cdot G^\e.
\]
Let
\[
p_d:=\frac{2(d+1)}{d-1}.
\]
The pair $(p_d,p_d)$ is wave-admissible and its corresponding Sobolev index is $\frac12$. Since $\chi$ is supported in a fixed compact subset of $\R^d\setminus\{0\}$, the $\dot H^{1/2}$ and $L^2$ norms are equivalent on the range of $\chi(D)$. Thus, the standard wave Strichartz estimate \cite{KT98}, together with the change of variables $t=\e\tau/c$, gives
\bq\label{eq:scaled-strichartz}
\|\chi(D)e^{\pm ict\Lambda/\e}f\|_{L^{p_d}((0,T)\times\R^d)} \le C_{\chi,c}\e^{1/p_d}\|f\|_{L^2}.
\eq
By Duhamel's formula, Minkowski's inequality, and \eqref{eq:scaled-strichartz},

\[
\|\chi(D)z_\pm^\e\|_{L^{p_d}((0,T)\times\R^d)}  \le C_{\chi,c}\e^{1/p_d} \|\chi(D)z_\pm^\e(0)\|_{L^2}  + C_{\chi,c}\e^{1/p_d} \int_0^T \| \chi(D) \lt( F^\e \pm i\Lambda^{-1}\nabla\cdot G^\e \rt)(\tau) \|_{L^2}\, \rd\tau.
\]
Since $\Lambda^{-1}\nabla\cdot$ is an $L^2$-bounded Fourier multiplier and commutes with $\chi(D)$, the assumptions of the lemma imply that the right-hand side tends to zero as $\e\to0$. Hence
\[
\chi(D)z_\pm^\e \to0 \quad\mbox{strongly in } L^{p_d}((0,T)\times\R^d).
\]
Since
\[
a^\e=\frac{z_+^\e+z_-^\e}{2}, \quad d^\e=\frac{z_+^\e-z_-^\e}{2i},
\]
we obtain
\[
\chi(D)a^\e\to0, \quad \chi(D)d^\e\to0
\]
strongly in $L^{p_d}((0,T)\times\R^d)$. Moreover,
\[
\chi(D)w^\e = -\nabla\Lambda^{-1}\chi(D)d^\e.
\]
Since $\nabla\Lambda^{-1}$ is bounded on $L^{p_d}$, it follows that
\[
\chi(D)w^\e \to0 \quad\mbox{strongly in } L^{p_d}((0,T)\times\R^d).
\]
Finally, since $p_d>2$, H\"older's inequality gives, for every compact set $K\Subset\R^d$,
\[
\|\chi(D)a^\e\|_{L^2(0,T;L^2(K))} + \|\chi(D)w^\e\|_{L^2(0,T;L^2(K))} \to0.
\]
This completes the proof.
\end{proof}

We now apply Lemma \ref{lem:acoustic-decay} to the acoustic component of the Euler--alignment system. Set
\[
r^\e:=\frac{\rho^\e-1}{\e} = \frac{h^\e}{\e}, \quad a^\e:=\sqrt{\gamma}\,r^\e, \quad w^\e:=\bbq u^\e.
\]
By Theorem \ref{thm:global},
\bq\label{eq:uniform-acoustic}
\sup_{0<\e\le1}\sup_{0\le t\le T} \lt( \|r^\e(t)\|_{H^s} + \|u^\e(t)\|_{H^s} \rt) \le C_T.
\eq
The continuity equation gives
\[
\pa_t a^\e + \frac{\sqrt{\gamma}}{\e}\nabla\cdot w^\e = -\nabla\cdot(a^\e u^\e).
\]
For the velocity equation, we write
\[
\frac{\gamma}{\e} (1+\e r^\e)^{\gamma-2}\nabla r^\e = \frac{\sqrt{\gamma}}{\e}\nabla a^\e + \gamma b^\e(r^\e)\nabla r^\e, \quad b^\e(z) := \frac{(1+\e z)^{\gamma-2}-1}{\e}.
\]
Since $\e r^\e=h^\e$ is uniformly small in $L^\infty$, the algebra property of $H^s$, together with the Moser estimate \eqref{eq:moser_den}, gives
\bq\label{eq:beps-uniform}
\|b^\e(r^\e)\|_{H^s} \le C\|r^\e\|_{H^s},
\eq
where $C>0$ is independent of $\e$.

Since $\bbq$ commutes with Fourier multipliers and acts as the
identity on gradient fields,
\[
\bbq\nabla a^\e = \nabla a^\e, \quad \bbq\Lambda^{2\alpha}u^\e = \Lambda^{2\alpha}w^\e.
\]
Applying $\bbq$ to the velocity equation, we obtain
\[
\pa_t w^\e + \frac{\sqrt{\gamma}}{\e}\nabla a^\e = G^\e,
\]
where
\[
G^\e := -\bbq(u^\e\cdot\nabla u^\e) -\Lambda^{2\alpha}w^\e -\gamma\bbq\lt(b^\e(r^\e)\nabla r^\e\rt) +\e\bbq\calR^\e, \quad \calR^\e := -\Lambda^{2\alpha}(r^\e u^\e) + u^\e\Lambda^{2\alpha}r^\e.
\]
Here we estimate
\bq\label{eq:alignment-rem-neg}
\|\calR^\e\|_{H^{-\alpha}}
\le
C\|r^\e\|_{H^s}\|u^\e\|_{H^s}.
\eq
due to $s>\frac d2$ and $s>\alpha$.

Let $\chi\in C_c^\infty(\R^d\setminus\{0\})$. Since $\chi$ is supported in a fixed annulus, Fourier multipliers of finite order are bounded on the range of $\chi(D)$. In particular,
\[
\|\chi(D)\Lambda^{2\alpha}f\|_{L^2} \le C_\chi\|f\|_{L^2}, \quad \|\chi(D)f\|_{L^2} \le C_\chi\|f\|_{H^{-\alpha}}.
\]
Moreover, since $s>\frac d2$ and $d\ge2$, we have $s>1$. Thus, by \eqref{eq:uniform-acoustic}, \eqref{eq:beps-uniform}, and \eqref{eq:alignment-rem-neg}, we obtain
\[
\|\chi(D)\nabla\cdot(a^\e u^\e)\|_{L^2} \le C_\chi\|a^\e u^\e\|_{L^2} \le C_\chi\|a^\e\|_{L^\infty} \|u^\e\|_{L^2} \le C_\chi,
\]
and
\[
\|\chi(D)G^\e\|_{L^2} \le C_\chi\lt( \|u^\e\|_{L^\infty} \|\nabla u^\e\|_{L^2} + \|w^\e\|_{L^2} + \|b^\e(r^\e)\|_{L^\infty} \|\nabla r^\e\|_{L^2} + \e\|\calR^\e\|_{H^{-\alpha}} \rt) \le C_\chi.
\]
Consequently,
\[
\sup_{0<\e\le1} \lt( \|\chi(D)\nabla\cdot(a^\e u^\e)\|_{L^1(0,T;L^2)} + \|\chi(D)G^\e\|_{L^1(0,T;L^2)} \rt) \le C_{\chi,T}.
\]
Thus, Lemma \ref{lem:acoustic-decay} applies and gives
\bq\label{eq:middle-acoustic-decay}
\chi(D)w^\e \to0 \quad\mbox{strongly in } L^2(0,T;L^2(K))
\eq
for every compact set $K\Subset\R^d$.

It remains to remove the frequency localization. Let $0<\eta<1<R$ and choose a radial function $\varphi\in C_c^\infty(\R^d)$ such that
\[
\varphi(\xi)=1 \quad\mbox{for }|\xi|\le1, \quad \varphi(\xi)=0 \quad\mbox{for }|\xi|\ge2.
\]
Define
\[
P_{\le\eta} := \varphi(D/\eta), \quad P_{\ge R} := I-\varphi(D/R), \quad  \chi_{\eta,R}(\xi) := \varphi(\xi/R)-\varphi(\xi/\eta).
\]
Then
\[
\chi_{\eta,R} \in C_c^\infty(\R^d\setminus\{0\})
\quad \text{and} \quad
I = P_{\le\eta} + \chi_{\eta,R}(D) + P_{\ge R}.
\]
Consequently,
\[
w^\e = P_{\le\eta}w^\e + \chi_{\eta,R}(D)w^\e + P_{\ge R}w^\e.
\]
By Bernstein's inequality and \eqref{eq:uniform-acoustic}, we obtain
\[
\|P_{\le\eta}w^\e\|_{L^2(0,T;L^2(K))} \le C_{K,T}\eta^{d/2} \|w^\e\|_{L^\infty(0,T;L^2)} \le C_{K,T}\eta^{d/2}.
\]
On the other hand, we find
\[
\|P_{\ge R}w^\e\|_{L^2(0,T;L^2)} \le C_TR^{-s} \|w^\e\|_{L^\infty(0,T;H^s)} \le C_TR^{-s}.
\]
For fixed $\eta$ and $R$, \eqref{eq:middle-acoustic-decay} gives
\[
\chi_{\eta,R}(D)w^\e \to0 \quad\mbox{strongly in } L^2(0,T;L^2(K)).
\]
This yields
\[
\limsup_{\e\to0} \|w^\e\|_{L^2(0,T;L^2(K))} \le C_{K,T}\lt(\eta^{d/2}+R^{-s}\rt).
\]
Letting first $\eta\to0$ and then $R\to\infty$, we obtain
\bq\label{eq:Q-local-L2}
\bbq u^\e=w^\e \to0 \quad\mbox{strongly in } L^2(0,T;L^2_{loc}(\R^d)).
\eq
We next treat the incompressible component. Since
\[
(1+h^\e)^{\gamma-2}\nabla h^\e = \nabla \lt( \frac{(1+h^\e)^{\gamma-1}-1}{\gamma-1} \rt),
\]
the pressure term vanishes after applying the Leray projection $\bbp$. Since $\bbp$ commutes with $\Lambda^{2\alpha}$, we obtain
\[
\pa_t\bbp u^\e = -\bbp(u^\e\cdot\nabla u^\e) -\Lambda^{2\alpha}\bbp u^\e +\e\bbp\calR^\e.
\]
Since $s>\frac d2$ and $d\ge2$, we have $s>1$. Hence,
\[
\|u^\e\cdot\nabla u^\e\|_{L^2} \le \|u^\e\|_{L^\infty} \|\nabla u^\e\|_{L^2} \le C.
\]
Moreover, since $0<\alpha<1<s$, we get
\[
\|\Lambda^{2\alpha}\bbp u^\e\|_{H^{-\alpha}} \le C\|u^\e\|_{H^s},
\]
while \eqref{eq:alignment-rem-neg} gives $\|\e\bbp\calR^\e\|_{H^{-\alpha}} \le C\e$. Hence, we have
\bq\label{eq:P-time-compact}
\sup_{0<\e\le1} \|\pa_t\bbp u^\e\|_{L^\infty(0,T;H^{-\alpha})} \le C_T.
\eq

Let $\Omega\Subset\R^d$ be a bounded smooth domain. For every $0<\delta<s$, $H^s(\Omega) \Subset H^{s-\delta}(\Omega) \hookrightarrow H^{-\alpha}(\Omega)$. Thus, \eqref{eq:P-time-compact}, together with the uniform bound of $\bbp u^\e$ in $L^\infty(0,T;H^s)$, allows us to apply the Aubin--Lions compactness lemma on $\Omega$. Using a diagonal argument over a smooth exhaustion of $\R^d$ and over a sequence $\delta_m\downarrow0$, we obtain, up to a subsequence,
\bq\label{eq:P-local-compact}
\bbp u^\e \to u \quad\mbox{strongly in } L^2(0,T;H^{s-\delta}_{loc}(\R^d))
\eq
for every $0<\delta<s$.

We now return to the acoustic component. Let $\Omega\Subset\Omega_1\Subset\R^d$ be bounded smooth domains and choose $\zeta\in C_c^\infty(\Omega_1)$ such that $\zeta=1$ on $\Omega$. By Sobolev interpolation,
\[
\|\bbq u^\e\|_{H^{s-\delta}(\Omega)} \le C \|\zeta\bbq u^\e\|_{L^2}^{\delta/s} \|\zeta\bbq u^\e\|_{H^s}^{1-\delta/s}.
\]
The second factor is uniformly bounded by \eqref{eq:uniform-acoustic}, whereas the first factor converges to zero in $L^2(0,T)$ by \eqref{eq:Q-local-L2}. Hence, after integration in time and an application of H\"older's inequality,
\[
\bbq u^\e \to0 \quad\mbox{strongly in } L^2(0,T;H^{s-\delta}_{loc}(\R^d))
\]
for every $0<\delta<s$. Combining this with \eqref{eq:P-local-compact}, we conclude that
\bq\label{eq:u-local-compact}
u^\e \to u \quad\mbox{strongly in } L^2(0,T;H^{s-\delta}_{loc}(\R^d))
\eq
for every $0<\delta<s$. Since $\bbp u^\e$ is divergence-free, \eqref{eq:P-local-compact} also yields
\[
\nabla\cdot u=0
\]
in the sense of distributions.

We also record the regularity inherited by the limit. By \eqref{main_uniform_est},
\[
\sup_{0<\e\le1} \|u^\e\|_{L^\infty(0,T;H^s)} \le C_T \quad \text{and} \quad \sup_{0<\e\le1} \|\Lambda^\alpha u^\e\|_{L^2(0,T;H^s)} \le C_T.
\]
Since
\[
\|f\|_{H^{s+\alpha}}^2 \le C\lt( \|f\|_{H^s}^2 + \|\Lambda^\alpha f\|_{H^s}^2 \rt),
\]
we obtain
\[
\sup_{0<\e\le1} \|u^\e\|_{L^2(0,T;H^{s+\alpha})} \le C_T.
\]
Thus, after passing to a further subsequence if necessary,
\[
u^\e\rightharpoonup^\ast \widetilde u \quad\mbox{in } L^\infty(0,T;H^s(\R^d))
\quad \text{and} \quad 
u^\e\rightharpoonup \widetilde u \quad\mbox{in } L^2(0,T;H^{s+\alpha}(\R^d)).
\]
The local strong convergence \eqref{eq:u-local-compact} identifies $\widetilde u$ with $u$. Consequently,
\[
u\in L^\infty(0,T;H^s(\R^d)) \cap L^2(0,T;H^{s+\alpha}(\R^d)),
\]
and, in particular,
\bq\label{eq:weak-frac-limit}
\Lambda^\alpha u^\e \rightharpoonup \Lambda^\alpha u \quad\mbox{weakly in } L^2(0,T;L^2(\R^d)).
\eq
We finally identify the limiting equation. By \eqref{main_uniform_est},
\[
\|\rho^\e-1\|_{L^\infty(0,T;H^s)} \le C\e,
\]
and hence
\bq\label{eq:rho-strong-limit}
\rho^\e -1 \to 0 \quad\mbox{strongly in } L^\infty(0,T;H^s(\R^d)).
\eq
After extracting a further subsequence if necessary, we may also assume that
\[
u_0^\e \rightharpoonup u_0 \quad\mbox{weakly in } H^s(\R^d).
\]

Let $\psi\in C_c^\infty([0,T)\times\R^d;\R^d)$,  $\nabla\cdot\psi=0$. Testing the momentum equation in \eqref{main_eq_e} against $\psi$, the singular pressure term vanishes identically. Using the symmetry of the alignment force, we obtain
\begin{align*}
0 &= \int_0^T\intr \rho^\e u^\e\cdot\pa_t\psi \,\dx\dt + \int_0^T\intr \rho^\e u^\e\otimes u^\e:\nabla\psi \,\dx\dt  + \intr \rho_0^\e u_0^\e\cdot\psi(0) \,\dx\\
&\quad -\frac12 \int_0^T\intor \phi(x-y) \rho^\e(x)\rho^\e(y) \lt(u^\e(x)-u^\e(y)\rt) \cdot \lt(\psi(x)-\psi(y)\rt) \dx\dy\dt.
\end{align*}
By \eqref{eq:u-local-compact} and \eqref{eq:rho-strong-limit}, we get
\[
\rho^\e u^\e \to u \quad\mbox{strongly in } L^1_{loc}((0,T)\times\R^d)
\]
and
\[
\rho^\e u^\e\otimes u^\e \to u\otimes u \quad\mbox{strongly in } L^1_{loc}((0,T)\times\R^d).
\]
Indeed, the second convergence follows from
\[
\|u^\e\otimes u^\e-u\otimes u\|_{L^1((0,T)\times K)} \le \|u^\e-u\|_{L^2((0,T)\times K)} \lt( \|u^\e\|_{L^2((0,T)\times K)} + \|u\|_{L^2((0,T)\times K)} \rt)
\]
for every compact set $K\Subset\R^d$.

For the alignment term, \eqref{eq:rho-strong-limit} and the Sobolev embedding $H^s\hookrightarrow L^\infty$ give
\[
\|\rho^\e-1\|_{L^\infty((0,T)\times\R^d)} \le C\e.
\]
Since $\rho^\e$ is uniformly bounded in $L^\infty$, we also find
\[
\|\rho^\e(x)\rho^\e(y)-1\|_{L^\infty_{t,x,y}} \le C\e.
\]
Consequently, for almost every $t\in(0,T)$,
\[
\lt| \frac12\intor \phi(x-y) \lt( \rho^\e(x)\rho^\e(y)-1 \rt) \lt( u^\e(x)-u^\e(y) \rt) \cdot \lt( \psi(x)-\psi(y) \rt) \dx\dy \rt|  \le C\e \|\Lambda^\alpha u^\e\|_{L^2} \|\Lambda^\alpha\psi\|_{L^2}.
\]
Hence, we obtain
\begin{align*}
&\int_0^T \lt| \frac12\intor \phi(x-y) \lt( \rho^\e(x)\rho^\e(y)-1 \rt) \lt( u^\e(x)-u^\e(y) \rt) \cdot \lt( \psi(x)-\psi(y) \rt) \dx\dy \rt| \dt\\
&\quad \le C\e \|\Lambda^\alpha u^\e\|_{L^2(0,T;L^2)} \|\Lambda^\alpha\psi\|_{L^2(0,T;L^2)} \to0.
\end{align*}
Using the normalization of $\phi$ and \eqref{eq:weak-frac-limit}, we also have
\[
-\frac12 \int_0^T\intor \phi(x-y) \lt( u^\e(x)-u^\e(y) \rt) \cdot \lt( \psi(x)-\psi(y) \rt) \dx\dy\dt \to -\int_0^T\intr \Lambda^\alpha u\cdot\Lambda^\alpha\psi \,\dx\dt.
\]
Thus the full alignment contribution converges to
\[
-\int_0^T\intr \Lambda^\alpha u\cdot\Lambda^\alpha\psi \,\dx\dt.
\]
For the initial term, by the uniform assumptions on the initial data,
\[
\|\rho_0^\e-1\|_{L^\infty} \le C\e.
\]
Hence
\[
\lt| \intr (\rho_0^\e-1)u_0^\e\cdot\psi(0) \,\dx \rt| \le \|\rho_0^\e-1\|_{L^\infty} \|u_0^\e\|_{L^2} \|\psi(0)\|_{L^2} \le C\e.
\]
Together with $u_0^\e\rightharpoonup u_0$ weakly in $H^s(\R^d)$, this yields
\[
\intr \rho_0^\e u_0^\e\cdot\psi(0) \,\dx \to \intr u_0\cdot\psi(0) \,\dx.
\]
Passing to the limit in the weak formulation, we obtain
\[
\int_0^T\intr \lt( u\cdot\pa_t\psi + u\otimes u:\nabla\psi - \Lambda^\alpha u\cdot\Lambda^\alpha\psi \rt) \dx\dt + \intr u_0\cdot\psi(0) \,\dx =0.
\]
Since $\psi$ is divergence-free,
\[
\intr u_0\cdot\psi(0) \,\dx = \intr \bbp u_0\cdot\psi(0) \,\dx.
\]
Thus the initial datum of the limiting incompressible equation is $\bbp u_0$. Moreover, by the de Rham theorem, there exists a distribution $\pi$ such that
\[
\pa_t u + u\cdot\nabla u + \nabla\pi + \Lambda^{2\alpha}u =0, \quad \nabla\cdot u=0
\]
in the sense of distributions. This proves the first part of Theorem \ref{thm:incomp-limit}.

%
%
%
%
%
%
\subsection{Well-prepared data: relative-energy convergence}

We now prove the second part of Theorem \ref{thm:incomp-limit}. Let $v$ be the sufficiently regular solution to the fractional incompressible Navier--Stokes system \eqref{eq:frac_NS} on $[0,T]$ appearing in the theorem. Let $(\rho^\e,u^\e)$ be the global strong solution to \eqref{main_eq_e} given by Theorem \ref{thm:global}, and set $h^\e=\rho^\e-1$. Then Theorem \ref{thm:global}, in particular \eqref{main_uniform_est}, gives, uniformly in $\e$,
\[
\sup_{0\le t\le T}\lt( \frac{\|h^\e(t)\|_{H^s}}{\e} +\|u^\e(t)\|_{H^s} \rt)\le C.
\]
Choose $0<\delta_0<s-\frac d2$. Then
\bq\label{rho-H-delta}
\|\rho^\e-1\|_{L^\infty(0,T;H^{\frac d2+\delta_0})}
\le C\e.
\eq
Recall the relative energy
\[
\mathscr{H}^\e(t)= \frac12\intr \rho^\e |u^\e-v|^2\,\dx +\frac{1}{\e^2(\gamma-1)} \intr \lt[(\rho^\e)^\gamma-1-\gamma(\rho^\e-1)\rt] \dx.
\]
A direct computation using \eqref{main_eq_e} and \eqref{eq:frac_NS} gives
\begin{align*}
&\frac\rd\dt\mathscr{H}^\e(t) +\|\Lambda^\alpha(u^\e-v)\|_{L^2}^2 \\ 
&\quad = -\intr \rho^\e (u^\e-v)\otimes(u^\e-v):\nabla v\,\dx  -\intr h^\e (u^\e-v)\cdot\Lambda^{2\alpha}(u^\e-v)\,\dx\\
&\quad \quad +\intr \rho^\e(u^\e-v)\cdot\nabla \pi\,\dx  -\intr \rho^\e(u^\e-v)\cdot \lt( \Lambda^{2\alpha}(h^\e u^\e) -u^\e\Lambda^{2\alpha}h^\e \rt) \dx.
\end{align*}
We estimate the terms on the right-hand side. First,
\[
-\intr \rho^\e (u^\e-v)\otimes(u^\e-v):\nabla v\,\dx \le C\|\nabla v\|_{L^\infty}\intr \rho^\e |u^\e-v|^2\,\dx.
\]
Next, by the  self-adjointness of $\Lambda^\alpha$ and the fractional product  estimate \eqref{eq:frac_prod1}, we obtain
\begin{align*}
\lt| \intr h^\e (u^\e-v)\cdot\Lambda^{2\alpha}(u^\e-v)\,\dx\rt|  &\le C\|\Lambda^\alpha(u^\e-v)\|_{L^2} \|\Lambda^\alpha(h^\e(u^\e-v))\|_{L^2}\\
&\le C\|h^\e\|_{H^{\frac d2+\delta_0}} \|\Lambda^\alpha(u^\e-v)\|_{L^2}^2\\
&\le C\e\|\Lambda^\alpha(u^\e-v)\|_{L^2}^2.
\end{align*}
We again use the self-adjointness of $\Lambda^\alpha$ to have

\begin{align*}
&\lt|\intr \rho^\e(u^\e-v)\cdot\lt(\Lambda^{2\alpha}(h^\e u^\e)-u^\e\Lambda^{2\alpha}h^\e\rt)\dx\rt|\\
&\quad\le \|\Lambda^\alpha(\rho^\e(u^\e-v))\|_{L^2}\|\Lambda^\alpha(h^\e u^\e)\|_{L^2}
+\|\Lambda^\alpha(\rho^\e u^\e\cdot(u^\e-v))\|_{L^{\frac{1}{1-\frac\alpha d}}}\|\Lambda^\alpha h^\e\|_{L^\frac d\alpha}\\
&\quad \le C\lt( \|\Lambda^\alpha(u^\e-v)\|_{L^2}+ \|\Lambda^\alpha(h^\e(u^\e-v))\|_{L^2}\rt)\|\Lambda^\alpha(h^\e u^\e)\|_{L^2}\\
&\quad \quad+C\| h^\e\|_{H^s}\lt(\|\Lambda^\alpha(u^\e\cdot(u^\e-v))\|_{L^{\frac{1}{1-\frac\alpha d}}} + \|\Lambda^\alpha(h^\e u^\e\cdot(u^\e-v))\|_{L^{\frac{1}{1-\frac\alpha d}}}\rt)
\end{align*}
We use \eqref{eq:hu-low} to have
\[
\|\Lambda^\alpha(h^\e u^\e)\|_{L^2} \le C\|h^\e\|_{H^s} \|\Lambda^\alpha u^\e\|_{L^2} \le C\e,
\]
and
\[
\|\Lambda^\alpha(h^\e(u^\e-v))\|_{L^2} \le C\|h^\e\|_{H^s}\|\Lambda^\alpha (u^\e - v)\|_{L^2} \le C\e \|\Lambda^\alpha (u^\e - v)\|_{L^2}.
\]
Furthermore, the fractional product estimate \eqref{eq:frac_prod1}  implies
\[
\begin{aligned}
\|\Lambda^\alpha(u^\e\cdot(u^\e-v))\|_{L^{\frac{1}{1-\frac\alpha d}}} &\le C\lt(\|\Lambda^\alpha u^\e\|_{L^2} \|u^\e - v\|_{L^{\frac{1}{\frac12 -\frac\alpha d}}} + \|u^\e\|_{L^{\frac{1}{\frac12-\frac\alpha d}}} \|\Lambda^\alpha (u^\e - v)\|_{L^2}\rt)\\
&\le C\|\Lambda^\alpha u^\e\|_{L^2}\|\Lambda^\alpha (u^\e - v)\|_{L^2}\\
&\le C\|\Lambda^\alpha (u^\e - v)\|_{L^2},
\end{aligned}
\]
and similarly,
\[
\|\Lambda^\alpha(h^\e u^\e\cdot(u^\e-v))\|_{L^{\frac{1}{1-\frac\alpha d}}} \le C\|\Lambda^\alpha (h^\e u^\e) \|_{L^2} \|\Lambda^\alpha (u^\e - v)\|_{L^2} \le C\e \|\Lambda^\alpha (u^\e - v)\|_{L^2},
\]
where we also used \eqref{eq:hu-low}.

Hence, we use Young's inequality to obtain
\[
\lt|\intr \rho^\e(u^\e-v)\cdot\lt(\Lambda^{2\alpha}(h^\e u^\e)-u^\e\Lambda^{2\alpha}h^\e\rt)\dx\rt|  \le  \frac14\|\Lambda^\alpha (u^\e - v)\|_{L^2}^2 + C\e^2.
\]
For the pressure term, since $\nabla\cdot v=0$, we have

\begin{align*}
\intr \rho^\e(u^\e-v)\cdot\nabla \pi\,\dx &= -\intr \nabla\cdot(\rho^\e u^\e)\pi\,\dx -\intr (\rho^\e-1)v\cdot\nabla \pi\,\dx\\
&= \intr \pa_t(\rho^\e-1)\pi\,\dx -\intr (\rho^\e-1)v\cdot\nabla \pi\,\dx.
\end{align*}
Therefore, using Young's inequality and taking $\e>0$ sufficiently small, we obtain
\bq\label{relative-energy-diff}
\begin{aligned}
\frac\rd\dt\mathscr{H}^\e(t) +\frac12\|\Lambda^\alpha(u^\e-v)\|_{L^2}^2 &\le C\|\nabla v\|_{L^\infty}\intr \rho^\e |u^\e-v|^2\,\dx +C\e^2\\
&\quad +\intr \pa_t(\rho^\e-1)\pi\,\dx -\intr (\rho^\e-1)v\cdot\nabla \pi\,\dx.
\end{aligned}
\eq
Integrating the pressure contribution by parts in time gives
\[
\int_0^t\intr \pa_t(\rho^\e-1)\pi\,\dx \rd\tau = \intr (\rho^\e-1)\pi(t)\,\dx -\intr (\rho_0^\e-1)\pi_0\,\dx  -\int_0^t\intr (\rho^\e-1)\pa_t \pi\,\dx \rd\tau.
\]
Hence, by \eqref{rho-H-delta},
\[
\lt| \int_0^t\intr \pa_t(\rho^\e-1)\pi\,\dx \rd\tau -\int_0^t\intr (\rho^\e-1)v\cdot\nabla \pi\,\dx \rd\tau \rt| \le C\sup_{0\le \tau\le t}\|\rho^\e(\tau)-1\|_{L^2} \le C\e.
\]
Integrating \eqref{relative-energy-diff} over $(0,t)$, using the preceding bound on the pressure contribution, and applying Gr\"onwall's lemma, we conclude that, for all $t\in[0,T]$,
\[
\mathscr{H}^\e(t) +\frac12\int_0^t\|\Lambda^\alpha(u^\e-v)\|_{L^2}^2\,\rd\tau \le C_T\mathscr{H}^\e(0)+C_T\e.
\]
In particular, if $\mathscr{H}^\e(0)\to0$ as $\e \to 0$, then 
\[
\sup_{0\le t\le T}\mathscr{H}^\e(t) \to 0 \quad \text{and} \quad u^\e\to v \quad\text{in } L^2(0,T;\dot H^\alpha(\R^d)).
\]
This completes the proof.

%
%
%
%
%
%

\section*{Acknowledgments}
The work of Y.-P. Choi was supported by NRF grant no. 2022R1A2C1002820 and RS-2024-00406821.  The work of J. Jung was supported by NRF grant no. RS-2026-25589791.

%
%
%
%
%
%


%
%
%
%
%
%
\bibliographystyle{abbrv}
\bibliography{EA_incomp_lim}

\end{document}